\documentclass[11pt, A4paper, leqno]{amsart}

\pdfoutput=1

\usepackage{graphicx}
\usepackage{amssymb,amsmath}
\usepackage{epstopdf}
\usepackage{sfmath}
\usepackage{url}
\usepackage{mathtools}

\usepackage[small, nohug, heads-vee]{diagrams} 
\diagramstyle[labelstyle=\scriptstyle]

\newcommand*{\Scale}[2][4]{\scalebox{#1}{$#2$}}%
\newtheorem{theorem}{Theorem}[section]
\newtheorem{lemma}[theorem]{Lemma}
\newtheorem{corollary}[theorem]{Corollary}
\newtheorem{proposition}[theorem]{Proposition}

\theoremstyle{remark}
\newtheorem{remark}[theorem]{Remark}
\newtheorem{hypothesis}[theorem]{Hypothesis}
\newtheorem{example}[theorem]{Example}

\theoremstyle{definition}
\newtheorem{definition}[theorem]{Definition}

\newcommand\bR{{\mathbb{R}}}
\newcommand\bC{{\mathbb C}}
\newcommand\bZ{{\mathbb Z}}

\newcommand\bU{{\mathbb{U}}}

\newcommand\dev{{\bf dev}}
\newcommand\SI{{\mathbb{S}}}

\newcommand\Bd{{\rm bd}}
\newcommand\clo{{\rm Cl}}
\newcommand\bdd{{\mathbf{d}}}

\newcommand\ra{\rightarrow}

\newcommand\emp{\emptyset}
\newcommand\eps{\epsilon}

\newcommand\Aff{{\mathbf{Aff}}}
\newcommand\ovl{\overline}
\newcommand\Aut{{\mathbf{Aut}}}
\newcommand\Idd{{\rm I}}

\newcommand\bv{{\mathbf{v}}}

\newcommand\CN{{\mathcal{N}}}
\newcommand\Pgl{{\mathrm{PGL}}(n+1, \bR)}
\newcommand\Ag{{\mathrm{Ag}}}

\newcommand\rpn{\mathbb{RP}^n}

\newcommand\SL{{\mathsf{SL}}}

\newcommand\SO{{\mathsf{SO}}}

\newcommand\PGL{{\mathsf{PGL}}}
\newcommand\SLnp{{\mathsf{SL}}_\pm(n+1, \bR)}

\newcommand\GL{{\mathsf{GL}}}

\newcommand\PGLnp{{\mathsf{PGL}}(n+1, \bR)}

\newcommand\orb{\mathcal{O}} 
\newcommand\torb{\tilde{\mathcal{O}}}
\newcommand\bGamma{{\boldsymbol \Gamma}}
\newcommand\leng{{\mathrm{length}}}

\newcommand{\ranK}{{\mathrm{rank}}}

\newcommand\mx{{\mathrm{max}}}

\usepackage{xr} 

\begin{document}


\title[Ends of real projective orbifolds III]
{A classification of radial and totally geodesic ends of properly convex real projective orbifolds III: 
the convex but nonproperly convex and non-complete-affine radial ends}


\author{Suhyoung Choi}
\address{ Department of Mathematics \\ KAIST \\
Daejeon 305-701, South Korea 
}
\email{schoi@math.kaist.ac.kr}

\date{\today}







\begin{abstract} 
Real projective structures on $n$-orbifolds are useful in understanding the space of 
representations of discrete groups into $\SL(n+1, \bR)$ or $\PGL(n+1, \bR)$. A recent work shows that many hyperbolic manifolds 
deform to manifolds with such structures not projectively equivalent to the original ones. 
The purpose of this paper is to understand 
the structures of ends of real projective $n$-dimensional orbifolds. 
In particular, these have the radial or totally geodesic ends. 
In previous papers, we classified 
properly convex or complete radial ends under suitable conditions. 
In this paper, we will study radial ends that are convex but not properly convex nor complete affine. 
The main techniques are the theory of Fried and Goldman on affine manifolds, and a generalization of 
the work on Riemannian foliations by Molino, Carri\`ere, and so on. 
We will show that these are quasi-joins of horospheres and totally geodesic radial ends.  
These are deformations of joins of horospheres and totally geodesic radial ends. 

\end{abstract}

\thanks{This work was supported by the National Research Foundation
of Korea (NRF) grant funded by the Korea government (MEST) (No.2010-0027001).} 

\maketitle

\tableofcontents




\section{Introduction}

\subsection{Preliminary definitions.} 
We will briefly review the definitions already found in previous papers 
\cite{EDC1} and \cite{EDC2}. 

\subsubsection{Topology of orbifolds and their ends.}  
An {\em orbifold} $\orb$ is a topological space with charts modeling open sets by quotients of Euclidean open sets or half-open sets 
by finite group actions and compatibly patched with one another. 
The boundary $\partial \orb$ of an orbifold is defined as the set of points with only half-open sets as models. 
Orbifolds are stratified by manifolds. 
Let $\orb$ denote an $n$-dimensional orbifold with finitely many ends 
where end-neighborhoods are homeomorphic to closed  $(n-1)$-dimensional orbifolds times an open interval. 
We will require that $\orb$ is {\em strongly tame}; that is, $\orb$ has a compact suborbifold $K$ 
so that $\orb - K$ is a disjoint union of end-neighborhoods homeomorphic to 
closed $(n-1)$-dimensional orbifolds multiplied by open intervals.
Hence $\partial \orb$ is a compact suborbifold. 
(See \cite{Cbook} for an introduction to the geometric orbifold theory.)

\subsubsection{Real projective structures and ends} 
We will consider an orbifold $\orb$ with a real projective structure: 
This can be expressed as 
\begin{itemize}
\item having a pair $(\dev, h)$ where 
$\dev:\torb \ra \rpn$ is an immersion equivariant with respect to 
\item the homomorphism $h: \pi_1(\orb) \ra \PGLnp$ where 
$\torb$ is the universal cover and $\pi_1(\orb)$ is the group of deck transformations acting on $\torb$. 
\end{itemize}
$(\dev, h)$ is only determined up to an action of $\PGLnp$ 
given by 
\[ g \circ (\dev, h(\cdot)) = (g \circ \dev, g h(\cdot) g^{-1}) \hbox{ for } g \in \PGLnp. \]
We will use only one pair where $\dev$ is an embedding for this paper and hence 
identify $\torb$ with its image. 
A {\em holonomy} is an image of an element under $h$. 
The {\em holonomy group} is the image group $h(\pi_1(\orb))$.  

We will assume that our real projective orbifold
$\orb$ is a strongly tame orbifold and some of the ends 
are {\em radial}.  Each radial end has a neighborhood $U$, and each 
component $\tilde U$ of the inverse image $p_\orb^{-1}(U)$
has a foliation by properly embedded projective geodesics ending at a common point $\bv_{\tilde U} \in \rpn$. 
We call such a point a {\em pseudo-end vertex}. 
Given an end $E$ of $\orb$, we can define a pseudo-end $\tilde E$ corresponding to it.
$\pi_{1}(\orb)$ acts on the set of pseudo-ends corresponding to $E$ transitively. 
The subgroup fixing a pseudo-end $\tilde E$ is denoted by $\pi_{1}(\tilde E)$. 
See \cite{EDC1} for detail. Heuristically, a pseudo-end is a class of ``equivalent'' system of 
connected open sets covering end neighborhoods of $E$. 

\begin{itemize} 
\item The {\em space of directions} of oriented projective geodesics through $\bv_{\tilde E}$ forms
an $(n-1)$-dimensional real projective space. 
We denote it by $\SI^{n-1}_{\bv_{\tilde E}}$, called a {\em linking sphere}. 
\item Two lines in $\torb$ from $\bv_{\tilde E}$ are regarded equivalent if they are identical near $\bv_{\tilde E}$. 
Let $\tilde \Sigma_{\tilde E}$ denote the space of equivalence classes of lines from $\bv_{\tilde E}$ in $\tilde U$.
$\tilde \Sigma_{\tilde E}$ projects to a convex open domain in an affine space in  $\SI^{n-1}_{\bv_E}$
by the convexity of $\torb$. Then by Proposition \ref{I-prop-projconv} of \cite{EDC1}, 
$\tilde \Sigma_{\tilde E}$ is projectively diffeomorphic to
\begin{itemize}
\item either a complete affine space $A^{n-1}$, 
\item a properly convex domain, 
\item or a convex but not properly convex 
and not complete affine domain in $A^{n-1}$. 
\end{itemize} 
\item We denote by $\Sigma_{\tilde E} $ the real projective $(n-1)$-orbifold $\tilde \Sigma_{E}/\bGamma_{E}$. 
Since we can find a transversal orbifold $\Sigma_{\tilde E}$ to the radial foliation in 
a pseudo-end-neighborhood for each pseudo-end $\tilde E$ of $\mathcal{O}$,
it lifts to a transversal surface $\tilde \Sigma_{\tilde E}$ in $\tilde U$. 
We can also simply denote it by $\Sigma_{E}$. 
\item We say that a radial pseudo-end $\tilde E$ is  {\em convex} (resp. {\em properly convex}, and {\em complete affine}) 
if $\tilde \Sigma_{\tilde E}$ is convex  (resp. properly convex, and complete affine). 
\end{itemize}

Thus, a radial end is either
\begin{description}
\item[CA] complete affine, 
\item[PC] properly convex, or 
\item[NPCC] convex but not properly convex and not complete affine. 
\end{description}

In \cite{EDC1}, we described an NPCC-end $E$ as 
a R-end $E$ with $\Sigma_{E}$ foliated by complete affine spaces of dimension $i_{0}$ for $0 < i_{0} < n-1$. 
For a p-R-end $\tilde E$ corresponding to $E$, $\tilde \Sigma_{\tilde E} \subset \SI^{n-1}_{\tilde E}$ is a convex 
but not properly convex and not complete affine. Then it is foliated by complete affine spaces of dimension $i_{0}$
with common boundary great sphere $\SI^{i_{0}-1}_{\infty}$ of dimension $i_{0}-1$. 
The space of such leaves can be identified with a properly convex open domain of dimension $n-i_{0}-1$. 
Here, we will call $i_{0}$ the {\em fiber-dimension} of the NPCC-end $E$. 

From now on, instead of the term ``pseudo-end'', we will use the term ``p-end''. 

\subsection{Main results.}  
 Recall from \cite{EDC1} that the universal cover $\tilde \Sigma_{\tilde E}$ of 
 the end orbifold $\Sigma_{\tilde E}$ is foliated by $i_{0}$-dimensional totally geodesic leaves 
 for $i_{0} > 1$. 
 The end fundamental group $\pi_{1}(\tilde E)$ acts on a properly convex domain $K$
 that is the space of $i_{0}$-dimensional totally geodesic leaves foliating $\tilde \Sigma_{\tilde E}$.
 
 Given a properly convex domain $K$, 
 $\Aut(K)$ is virtually isomorphic to 
 \[\bR^{l-1}\times \Gamma_{1} \times \dots \times \Gamma_{l}\]
 for strongly irreducible semisimple groups $\Gamma_{i}, i=1, \dots, l$
 if and only if $K$ is a strict join 
 \[K_{1}\ast \cdots \ast K_{l}\] where 
 $K_{i}$ is a properly convex domain of dimension $j_{i}$ where 
 $j_{1}+\cdots + j_{l} + l-1 = n$. (Of course, it can be $l = 1$.)
 The virtual center of $\Aut(K)$ is the diagonalizable group corresponding to $\bR^{l-1}$. 
 (See Section \ref{subsec-conv} of \cite{EDC1}. Here there is no condition on strict convexity of $K_{i}$).
 
The main result of this paper is: 

\begin{theorem}\label{thm-thirdmain} 
Let $\mathcal{O}$ be a strongly tame properly convex real projective orbifold with radial or totally geodesic ends.
Assume that the holonomy group of $\mathcal{O}$ is strongly irreducible.
\begin{itemize}
\item Let $\tilde E$ be an NPCC p-R-end. 
\item Let $K$ be the convex $n-i_{0}-1$-dimensional domain that is the space of 
$i_{0}$-dimensional totally geodesic affine spaces foliating 
the universal cover $\tilde \Sigma_{\tilde E}$ of the end orbifold $\Sigma_{\tilde E}$. 
\end{itemize} 
We assume that 
\begin{itemize}
\item a virtual center of $\bGamma_{\tilde E}$ goes to a Zariski dense subgroup of  
the virtual center of the group $\Aut(K)$ of projective automorphisms of $K$ and 
\item the p-end fundamental group $\pi_{1}(\tilde E)$ satisfies the weak middle-eigenvalue condition
for NPCC-ends. 
\end{itemize}
Then $\tilde E$ is of  quasi-joined type p-R-end. 
\end{theorem}
See Definition \ref{defn:weakmec} for the weak middle-eigenvalue condition
for NPCC-ends. Without this condition, we doubt we can obtain this type of results. 
However, it is open to investigations. 
In this case, $\tilde E$ does not satisfy the uniform middle-eigenvalue condition
as stated in \cite{EDC1} for properly convex ends. 


We will explain the quasi-joined type in Section \ref{subsec-qjoin}.  (See Definition \ref{defn-qjoin}.)


We remark that Cooper and Leitner has classified the properly convex ends when 
the end fundamental group is amenable. (See Leitner \cite{Leitner1} and \cite{Leitner2}.) 
Also, Ballas \cite{Ballas2012} and \cite{Ballas2014} has found some examples of joined ends 
when the semisimple part is a trivial group.



Recall the dual orbifold $\orb^{\ast}$ given a properly convex real projective orbifold $\orb$. 
(See \cite{EDC1} and Section \ref{II-sub-dualend} in \cite{EDC2}.) 
The set of ends of $\orb$ is in one-to-one correspondence with the set of ends of $\orb^{\ast}$. 
We show that a dual of a quasi-joined NPCC p-R-end is a quasi-joined NPCC p-R-end. 

\begin{corollary}\label{cor-dualNPCC} 
Let $\mathcal{O}$ be a strongly tame properly convex real projective orbifold with radial or totally geodesic ends.
Let $\tilde E$ be a quasi-joined NPCC p-R-end for an end $E$ of $\orb$ satisfying 
the weak middle-eigenvalue condition. 
Let $\orb^{\ast}$ denote the dual real projective orbifold of $\orb$. 
Let $\tilde E^{\ast}$ be a p-end corresponding to a dual end of $E$. 
Then $\tilde E^{\ast}$ has a p-end neighborhood of a quasi-joined type p-R-end. %
\end{corollary} 
In short, we are saying that $\tilde E^{\ast}$ can be considered a quasi-joined type p-R-end 
by choosing its p-end vertex well. However, this does involve artificially introducing a radial foliation structure
in an end neighborhood.

\subsection{Outline.} 

In Section \ref{sec-prelim}, we will briefly review the real projective geometry and convex sets. 

In Section \ref{sec-notprop}, we discuss the R-ends that are NPCC. 
First, we show that the end holonomy group for an end $E$ will have an exact sequence 
\[ 1 \ra N \ra h(\pi_1(\tilde E)) \longrightarrow N_K \ra 1\] 
where $N_K$ is in the projective automorphism group $\Aut(K)$ of a properly convex compact set $K$ 
and $N$ is the normal subgroup mapped to the trivial automorphism of $K$
and $K^o/N_K$ is compact. 
We show that $\Sigma_{\tilde E}$ is foliated by complete affine spaces of dimension $\geq 1$. 
We will explain the main eigenvalue estimates following from the weak middle eigenvalue condition for 
NPCC-ends. Then we will explain our plan to prove Theorem \ref{thm-thirdmain}.

In Section \ref{sec-gendiscrete}, 
we introduce the example of joining of horospherical and totally geodesic R-ends. 
We will now study a bit more general situation introducing a Hypothesis \ref{h-norm}. 
By computations involving the normalization conditions, 
we show that the above exact sequence is virtually split and we can surprisingly show that 
the p-R-ends are of joined or quasi-joined types. 
Then we show using the irreducibility of the holonomy group
of $\pi_{1}(\orb)$ that they can only be of quasi-joined type 
using the irreducibility. 
As a final part of this section, 
we discuss the case when $N_K$ is a discrete. We prove Theorem \ref{thm-thirdmain} for this case.

In Section \ref{sec-indiscrete}, we discuss when $N_K$ is not discrete. There is a foliation 
by complete affine spaces as above. 
We use some estimates on eigenvalues to show that each leaf 
is of polynomial growth.
The leaf closures are suborbifolds $V_l$ 
by the theory of Carri\`ere \cite{Car} and Molino \cite{Molbook} on Riemannian foliations. 
They form the fibration with compact fibers. 
$\pi_1(V_l)$ is solvable using the work of Carri\`ere \cite{Car}. 
One can then take the syndetic closure to 
obtain a bigger group that act transitively on each leaf following Fried and Goldman \cite{FG}. 
We find a standard nilpotent group acting on each leaf transitively normalized by $\bGamma_{\tilde E}$. 
Then we show that 
the end also splits virtually using the theory of Section \ref{sec-gendiscrete}. This proves Theorem \ref{thm-thirdmain}. 

In Section \ref{sec-dualNPCC}, we prove Corollary \ref{cor-dualNPCC}.  

\begin{remark}
Note that the results are stated in the space $\SI^n$ or $\bR P^n$. Often the result for $\SI^n$ implies 
the result for $\bR P^n$. In this case, we only prove for $\SI^n$. In other cases, we can easily modify 
the $\SI^n$-version proof to one for the $\bR P^n$-version proof.  
\end{remark}

\subsection{Acknowledgements} 
We thank 
Yves Carri\`ere with the general approach to study the indiscrete cases for nonproperly convex ends
using his and Molino's work considering the Riemannian foliation with leaves of polynomial growth.
I thank Sam Ballas and Daryl Cooper for explain their theory as related to our. 





\section{Preliminaries} \label{sec-prelim}

In this paper, we will be using the smooth category: that is, we will be using smooth maps and smooth charts and so on. 
We explain the material in the introduction again.  
We will establish that the universal cover $\torb$ of our orbifold $\orb$ is a domain 
in $\SI^n$ with a projective automorphism group $\Gamma \subset \SLnp$ acting on it. 
In this case, $\orb$ is projectively diffeomorphic to $\torb/\Gamma$. 




\subsection{Real projective structures}

Let $\bdd$ denote the standard spherical metric on $\SI^n$ {\rm (}resp. $\bR P^n${\rm )}.
Let $O$ denote the origin of any vector space here.
Given a vector space $V$, we denote by ${\mathcal P}(V)$ the projective space 
$(V -\{O\})/\sim$ where $\vec{v} \sim \vec{w}$ iff $\vec{v} = s \vec{w}$ for $s \in \bR -\{0\}$ 
and we denote by ${\mathcal S}(V)$ the sphere $(V-\{O\})/\sim$ where $\vec{v} \sim \vec{w}$ for $s \in \bR_+$. 
We denote $\rpn = {\mathcal P}(\bR^{n+1})$ and $\SI^n = {\mathcal S}(\bR^{n+1})$. 
A {\em subspace} of ${\mathcal P}(V)$ or ${\mathcal S}(V)$ is the image of a subspace in $V$ with $O$ removed.
Given any linear isomorphism $f: V \ra W$, we denote by ${\mathcal P}(f)$ the induced 
projective isomorphism ${\mathcal P}(V) \ra {\mathcal P}(W)$ and ${\mathcal S}(f)$ the induced map 
${\mathcal S}(V) \ra {\mathcal S}(W)$. These maps are called {\em projective maps}. 

The complement of a codimension-one subspace $W$ in $\rpn$ can be considered an affine 
space $A^n$ by correspondence 
\[[1, x_1, \dots, x_n] \ra (x_1, \dots, x_n)\] for a coordinate system where $W$ is given by $x_0=0$. 
The group $\Aff(A^n)$ of projective automorphisms acting on $A^n$ is identical with 
the group of affine transformations of form 
\[ \vec{x} \mapsto A \vec{x} + \vec{b} \] 
for a linear map $A: \bR^n \ra \bR^n$ and $\vec{b} \in \bR^n$. 
The projective geodesics and the affine geodesics agree up to parametrizations.

A cone $C$ in $\bR^{n+1} -\{O\}$ is a subset so that given a vector $x \in C$, 
$s x \in C$ for every $s \in \bR_+$. 
A {\em convex cone} is a cone that is a convex subset of $\bR^{n+1}$ in the usual sense. 
A {\em proper convex cone} is a convex cone not containing a complete affine line. 


Note that we can double-cover $\rpn$ by $\SI^n$ the unit sphere in $\bR^{n+1}$ 
and this induces a real projective structure on $\SI^n$. 

We can think of $\SI^n$ as ${\mathcal S}(\bR^{n+1})$.
We call this the real projective sphere. The antipodal map
\[\mathcal{A}: \SI^n \ra \SI^n \hbox{ given by } [\vec{v}] \ra [-\vec{v}] \hbox{ for } \vec{v} \in \bR^{n+1} -\{O\}\]
which generates the covering automorphism group of $\SI^n \ra \rpn$.
The group $\Aut(\SI^n)$ of projective automorphisms of $\SI^n$ is isomorphic to $\SLnp$.

A {\em great segment} is a geodesic segment with antipodal end vertices, which is convex but not properly convex. 
A segment has $\bdd$-length $=\pi$ if and only if it is a great segment. 

Given a projective structure where $\dev: \torb \ra \rpn$ is an embedding to a properly convex 
open subset as in this paper, 
$\dev$ lifts to an embedding $\dev': \torb \ra \SI^n$ to an open domain $D$ without any pair of antipodal 
points. $D$ is determined up to $\mathcal{A}$.

Let $\bGamma$ denote the group of deck transformations of $\torb$. 

\subsection{Convexity and convex domains}\label{subsec-conv}








A complete real line in $\rpn$ is a $1$-dimensional subspace of $\rpn$ with one point removed. 
That is, it is the intersection of a $1$-dimensional subspace by an affine space. 
An {\em affine $i$-dimensional subspace} is a submanifold of $\SI^n$ or $\rpn$ projectively diffeomorphic 
to an $i$-dimensional affine subspace of a complete affine space.
A {\em convex} projective geodesic is a projective geodesic in a real projective orbifold which lifts to 
a projective geodesic, the image of whose composition with a developing map does not contain a complete real line. 
A real projective orbifold is {\em convex} if every path can be homotopied to a convex projective geodesic with endpoints fixed. 

In the double cover $\SI^n$ of $\rpn$, an affine space $A^n$ is the interior of a hemisphere. 
A domain in $\rpn$ or $\SI^n$ is {\em convex} if it lies in some affine subspace and satisfies the convexity 
property above. Note that a convex domain in $\rpn$ lifts to ones in $\SI^n$ up to the antipodal map 
$\mathcal{A}$. A convex domain in $\SI^n$ not containing an antipodal pair maps to one in $\rpn$ homeomorphically. 
(Actually from now on, we will only be interested in convex domains in $\SI^n$.)

\section{The weak middle eigenvalue conditions for NPCC ends} \label{sec-notprop}

We will now study the ends where the transverse real projective structures are not properly convex 
but not projectively diffeomorphic to a complete affine subspace.
Let $\tilde E$ be a p-R-end of $\orb$ and let $U$ the corresponding p-end-neighborhood in $\torb$
with the p-end vertex $\bv_{\tilde E}$. 
Let $\tilde \Sigma_{\tilde E}$ denote the universal cover of the p-end orbifold $\Sigma_{\tilde E}$ as 
a domain in $\SI^{n-1}_{\bv_{\tilde E}}$. 

In Section \ref{sub-general}, we will discuss the general setting that the NPCC-ends satisfy. 
In Section \ref{sub-plan}, we will give a plan to prove Theorem \ref{thm-thirdmain}. 
This will be accomplished in Sections \ref{sec-gendiscrete} and \ref{sec-indiscrete}. 

\subsection{General setting} \label{sub-general} 

The closure $\clo(\tilde \Sigma_{\tilde E})$ contains 
a great $(i_0-1)$-dimensional sphere and the convex open domain $\tilde \Sigma_{\tilde E}$ is foliated by $i_0$-dimensional hemispheres 
with this boundary. These follow from Section 1.4 of \cite{ChCh}. (See also \cite{GV}.) 
Let $\SI^{i_0-1}_\infty$ denote the great $(i_0-1)$-dimensional sphere in $\SI^{n-1}_{\bv_{\tilde E}}$ of $\tilde \Sigma_{\tilde E}$. 
The space of $i_0$-dimensional hemispheres in $\SI^{n-1}_{\bv_{\tilde E}}$ with boundary $\SI^{i_0-1}_\infty$ form 
a projective sphere $\SI^{n-i_0-1}$. 
The projection 
\begin{alignat}{2} \label{eqn-pik}
\Pi_{K}:\SI^{n-1}_{\bv_{\tilde E}} - \SI^{i_0-1}_\infty &  \, \longrightarrow  \,\, && \SI^{n-i_0-1} \quad  \quad\\
 \quad \uparrow \quad  &                     && \uparrow \quad\quad \nonumber \\
\tilde \Sigma_{\tilde E}    & \, \longrightarrow  \,\, &&  K^{o}  \quad \quad\nonumber
\end{alignat}
gives us an image of $\tilde \Sigma_{\tilde E}$ 
that is the interior $K^{o}$ of a  properly convex compact set $K$. 

Let $\SI^{i_0}_\infty$ be a great $i_0$-dimensional sphere in $\SI^{n}$ containing $\bv_{\tilde E}$ corresponding to the directions
of $\SI^{i_0-1}_\infty$ from $\bv_{\tilde E}$. 
The space of  $(i_0+1)$-dimensional hemispheres 
with boundary $\SI^{i_0}_\infty$ again has the structure of the projective sphere $\SI^{n-i_0-1}$, 
identifiable with the above one. 
We have the projection $\Pi_K$ 
giving us the image $K^o$ of a p-end-neighborhood $U$. 


Each $i_0$-dimensional 
hemisphere $H^{i_0}$ in $\SI^{n-1}_{\bv_{\tilde E}}$ with 
$\Bd H^{i_0} = \SI^{i_0-1}_\infty$ corresponds to an $(i_0+1)$-dimensional hemisphere 
$H^{i_0+1}$ in $\SI^n$ with common boundary $\SI^{i_0}_\infty$ that contains $\bv_{\tilde E}$. 

Let $\SL_\pm(n+1, \bR)_{\SI^{i_0}_\infty, \bv_{\tilde E}}$ 
denote the subgroup of $\Aut(\SI^n)$ acting on 
$\SI^{i_0}_\infty$ and $\bv_{\infty}$.  
The projection $\Pi_K$ induces a homomorphism 
\[\Pi_K^*: \SL_\pm(n+1, \bR)_{\SI^{i_0}_\infty, \bv_{\tilde E}} 
\ra \SL_\pm( n-i_{0}-1, \bR).\]

Suppose that $\SI^{i_0}_\infty$ is $h(\pi_1(\tilde E))$-invariant. 
We let $N$ be the subgroup of $h(\pi_1(\tilde E))$ of elements inducing trivial actions on $\SI^{n-i_0-1}$. 
The above exact sequence 
\begin{equation}\label{eqn-exact}
1 \ra N \ra h(\pi_1(\tilde E)) \stackrel{\Pi^*_K}{\longrightarrow} N_K \ra 1
\end{equation}
is so that the kernel normal subgroup $N$ acts trivially on $\SI^{n-i_0-1}$ but acts on each hemisphere with 
boundary equal to $\SI^{i_0}_\infty$
and $N_K$ acts faithfully by the action induced from $\Pi^*_K$.

Since $K$ is a properly convex domain, $K^{o}$ admits a Hilbert metric $d_{K}$ and 
$\Aut(K)$ is a subgroup of isometries of $K^{o}$. 
Here $N_K$ is a subgroup of the group $\Aut(K)$ of the group of projective automorphisms of $K$, and 
$N_{K}$ is called the {\em semisimple quotient } of $h(\pi_1(\tilde E))$ or $\bGamma_{\tilde E}$. 

\begin{theorem}\label{thm-folaff}
Let $\Sigma_{\tilde E}$ be the end orbifold of an NPCC p-R-end $\tilde E$ of a strongly tame  
properly convex $n$-orbifold $\orb$ with radial or totally geodesic ends. Let $\torb$ be the universal cover in $\SI^n$. 
We consider the induced action of $h(\pi_1(\tilde E))$ 
on $\Aut(\SI^{n-1}_{\bv_{\tilde E}})$ for the corresponding end vertex $\bv_{\tilde E}$. 
Then 
\begin{itemize} 
\item $\Sigma_{\tilde E}$ is foliated by complete affine subspaces of dimension $i_0$, $i_0 > 0$.
\item $h(\pi_1(\tilde E))$ fixes the great sphere $\SI^{i_0-1}_\infty$ of dimension $i_0-1$ in $\SI^{n-1}_{\bv_{\tilde E}}$. 
\item There exists an exact sequence 
\[ 1 \ra N \ra \pi_1(\tilde E) \stackrel{\Pi^*_K}{\longrightarrow} N_K \ra 1 \] 
where $N$ acts trivially on quotient great sphere $\SI^{n-i_0-1}$ and 
$N_K$ acts faithfully on a properly convex domain $K^o$ in $\SI^{n-i_0-1}$ isometrically 
with respect to the Hilbert metric $d_K$. 
\end{itemize} 
\end{theorem}

We denote by $\mathcal{F}_{\tilde E}$ the foliation on $\tilde \Sigma_{\tilde E}$ or the corresponding one in $\Sigma_{\tilde E}$.

\subsubsection{The main eigenvalue estimations}

We denote by $\bGamma_{\tilde E}$ the p-end fundamental group acting on $U$ fixing $\bv_{\tilde E}$. 
Denote the induced foliations on $\Sigma_{\tilde E}$ and $\tilde \Sigma_{\tilde E}$ by ${\mathcal F}_{\tilde E}$.
For each element $g \in \bGamma_{\tilde E} $, we define $\leng_K(g)$ to 
be $\inf \{ d_K(x, g(x))| x \in K^o \}$.

\begin{definition}\label{defn-jordan}
Given an eigenvalue $\lambda$ of an element $g \in \SLnp$, 
a $\bC$-eigenvector $\vec v$ is a nonzero vector in 
\[\bR E_{\lambda}(g) := \bR^{n+1} \cap (\ker (g - \lambda I) + \ker (g - \bar \lambda I)), \lambda \ne 0, {\mathrm{Im}} \lambda \geq 0\]

Any element of $g$ has a Jordan decomposition. An irreducible Jordan-block corresponds to a unique subspace
in $\bC^{n+1}$, called an {\em elementary Jordan subspace}. We denote by $J_{\mu, i} \subset \bC^{n+1}$ 
for an eigenvalue $\mu \in \bC$
for $i$ in an index set.
\begin{itemize}
\item A {\em real elementary Jordan subspace} is defined as 
\[ R_{\mu, i} := \bR^{n+1} \cap (J_{\mu, i} + J_{\bar \mu, i}),
\mu \ne 0, {\mathrm{Im}} \mu \geq 0\]
of Jordan subspaces with 
 $\ovl{J_{\mu, i}} = J_{\bar \mu, i}$ in $\bC^{n+1}.$
\item We define the  {\em real sum} of elementary 
Jordan-block subspaces is defined to be 
\[\bigoplus_{i \in I}R_{\mu, i}\]
for a finite collection $I$.

\item A point $[\vec v], \vec v \in \bR^{n+1},$ is {\em affiliated} with a norm $\mu$ of an eigenvalue if 
$\vec v \in \bigoplus_{|\lambda| = \mu, i \in I_{\lambda}} R_{\lambda, i}$ 
for a sum of all real elementary Jordan subspaces $R_{\lambda, i}$, $\mu = |\lambda|$. 
\end{itemize} 
\end{definition}

Let $V^{i+1}_\infty$ denote the subspace of $\bR^{n+1}$ corresponding 
to $\SI^i_\infty$. 
By invariance of $\SI^i_\infty$, 
if \[\oplus_{(\mu, i) \in J} R_{\mu, i} \cap V^{i+1}_\infty \ne \emp\] for some finite collection $J$, 
then $\oplus_{(\mu, i) \in J} R_{\mu, i} \cap V^{i+1}_\infty$ always contains a $\bC$-eigenvector. 

\begin{definition}\label{defn-eig} 
Let $\Sigma_{\tilde E}$ be the end orbifold of a nonproperly convex and p-R-end $\tilde E$ of a 
strongly tame properly convex $n$-orbifold $\orb$ with radial or totally geodesic ends. Let $\bGamma_{\tilde E}$ 
be the p-end fundamental group. 
We fix a choice of a Jordan decomposition of $g$ for each $g \in \bGamma_{\tilde E}$. 
\begin{itemize}
\item Let $\lambda_1(g)$ denote the largest norm of the eigenvalue of $g \in \bGamma_{\tilde E}$ affiliated
with $\vec{v} \ne 0$, $[\vec{v}] \in \SI^n - \SI^{i_0}_\infty$,
i.e., \[\vec v \in \bigoplus_{(\mu_1(g), i) \in J} R_{\mu_1(g), i} -  V^{i_0+1}_\infty, |\mu_1| = \lambda_1(g)\]
where $J$ indexes all elementary Jordan subspaces of $\lambda_1(g)$.
\item Also, let $\lambda_{n+1}(g)$ denote the smallest one affiliated with a nonzero vector $\vec v$, $[\vec{v}] \in \SI^n - \SI^{i_0}_\infty$,
i.e., \[\vec v \in \bigoplus_{(\mu_1(g), i) \in  J'} R_{\mu_{n+1}(g), i} - V^{i_0+1}_\infty, |\mu_{n+1}|=\lambda_{n+1}(g)\]
where $J'$ indexes all real elementary Jordan subspaces of $\lambda_{n+1}(g)$.
\item Let $\lambda(g)$ be the largest of the norm of the eigenvalue of $g$ with a $\bC$-eigenvector $\vec v$, $[\vec v] \in \SI^{i_0}_\infty$
and $\lambda'(g)$ the smallest such one. 
\end{itemize} 
\end{definition}
(The sums of the Jordan subspaces here are of course well-defined.) 


Suppose that $K$ has a decomposition into $K_1 * \cdots * K_{l_0}$ for properly convex domains 
$K_i$, $i =1, \dots, l_0$. 
Let $K_i, i=1, \dots, s$, be the ones with dimension $\geq 2$.
$N_K$ is virtually isomorphic to
a cocompact subgroup of 
 the product \[ \bZ^{l_0-1} \times \Gamma_1 \times \dots \times \Gamma_{s}\]
where $\Gamma_i$ is obtained from $N_K$ by restricting to $K_i$ and 
$A$ is a free abelian group of finite rank. 
(Note that $\Gamma_{i}$ are not necessarily discrete.) 


The {\em virtual center} of $\bGamma_{\tilde E}$ is the elements that corresponds to $\bZ^{l_{0}}$. 

As in Section \ref{I-sec-notprop} of \cite{EDC1}, each $K_i$ is a properly convex domain or a point
by Theorem 1.1 of  \cite{Ben2}.  


We will assume that the p-end fundamental group $\pi_{1}(\tilde E)$ satisfies 
the weak middle eigenvalue condition for NPCC-ends: 
\begin{definition}\label{defn:weakmec}
Let $\bar \lambda(g)$ denote the largest norm of the eigenvalues 
of $g \in \Gamma_{\tilde E}$. 
Let $\lambda_{\bv_{\tilde E}}(g)$ denote the eigenvalue of $g$ at $\bv_{\tilde E}$. 
The weak middle eigenvalue condition  is that for each element $g$ of $\pi_{1}(\tilde E)$, 
\begin{equation} \label{eqn:mec}
\bar \lambda(g) \geq \lambda_{1}(g) \geq \lambda_{\bv_{\tilde E}}(g)
\end{equation} 
\end{definition} 
Hueristically, this can be considered ``transversal middle eigenvalue condition''.

The following proposition is very important in this article and shows that 
$\lambda_1(g)$ and $\lambda_{n+1}(g)$ are true largest and smallest norms of the eigenvalues of $g$. 
We will sharpen the following to inequality in the discrete and indiscrete cases.  
\begin{proposition}\label{prop-eigSI}  
Let $\Sigma_{\tilde E}$ be the end orbifold of a nonproperly convex p-R-end $\tilde E$ of a strongly tame 
properly convex $n$-orbifold $\orb$ with radial or totally geodesic ends. 
Suppose that $\torb$ in $\SI^n$ (resp. $\bR P^n$) 
covers $\orb$ as a universal cover. 
Let $\bGamma_{\tilde E}$ be the p-end fundamental group satisfying the weak middle-eigenvalue condition.
Let $g \in \bGamma_{\tilde E}$. 
Then
\[\lambda_1(g) \geq \lambda(g) \geq  \lambda'(g) \geq \lambda_{n+1}(g)\]
holds. 
\end{proposition} 
\begin{proof}
We may assume that $g$ is of infinite order. 
Suppose that $\bar \lambda(g) > \lambda_1(g)$. 
We have $\bar \lambda(g) \geq  \lambda_{\bv_{\tilde E}}(g)$ by the weak uniform middle eigenvalue condition. 
If $\bar \lambda(g) = \lambda_{\bv_{\tilde E}}(g)$, then $\lambda_{\bv_{\tilde E}}(g) > \lambda_1(g)$ contradicts 
the weak uniform middle eigenvalue condition. Thus, $\bar \lambda(g) > \lambda_{\bv_{\tilde E}}(g) .$

Now, $\lambda_1(g) < \bar \lambda(g)$ implies that
\[R_{\bar \lambda(g)}:= \bigoplus_{(\mu_1(g), i) \in J} R_{\mu, i}(g),  |\mu|=\bar \lambda(g)\]
is a subspace of $V^{i_0+1}_\infty$ and corresponds to a great sphere $\SI^j$. 
Hence, a great sphere $\SI^j$, $j \geq 0$, in $\SI^{i_0}_\infty$ is disjoint from 
$\{\bv_{\tilde E}, \bv_{\tilde E-}\}.$ 
Since $\bv_{\tilde E} \in \SI^{i_0}_\infty$ is not contained in $\SI^j$, we obtain $j+1 \leq i_0$.

A vector space $V_1$ corresponds the real sum of Jordan-block subspaces where 
$g$ has strictly smaller norm eigenvalues and is complementary to $R_{\bar \lambda(g)}$. 
Let $C_{1} =\mathcal{S}(V_{1})$. 
The great sphere $C_1$ is disjoint from $\SI^j$ but $C_1$ contains $\bv_{\tilde E}$. 
Moreover, $C_1$ is of complementary dimension to $S^j$, i.e., $\dim C_1 = n - j-1$. 


Since $C_1$ is complementary to $\SI^j \subset \SI^{i_0}_\infty$, a complementary subspace 
$C'_1$ to $\SI^{i_0}_\infty$ of dimension $n-i_0-1$ is in $C_1$. 
Considering the sphere $\SI^{n-1}_{\bv_{\tilde E}}$ at $\bv_{\tilde E}$, 
it follows that $C'_1$ goes to an $n-i_0-1$-dimensional subspace $C''_1$ in 
$\SI^{n-1}_{\bv_{\tilde E}}$ disjoint from $\partial l$ for any complete affine leaf $l$. 
Each complete affine leaf $l$ of $\tilde \Sigma_{\tilde E}$ has the dimension $i_0$ 
and meets $C''_1$ in $\SI^{n-1}_{\bv_{\tilde E}}$ by the dimension consideration. 


\begin{figure}
\centering
\includegraphics[trim = 15mm 30mm 10mm 10mm, clip,  height=5cm]{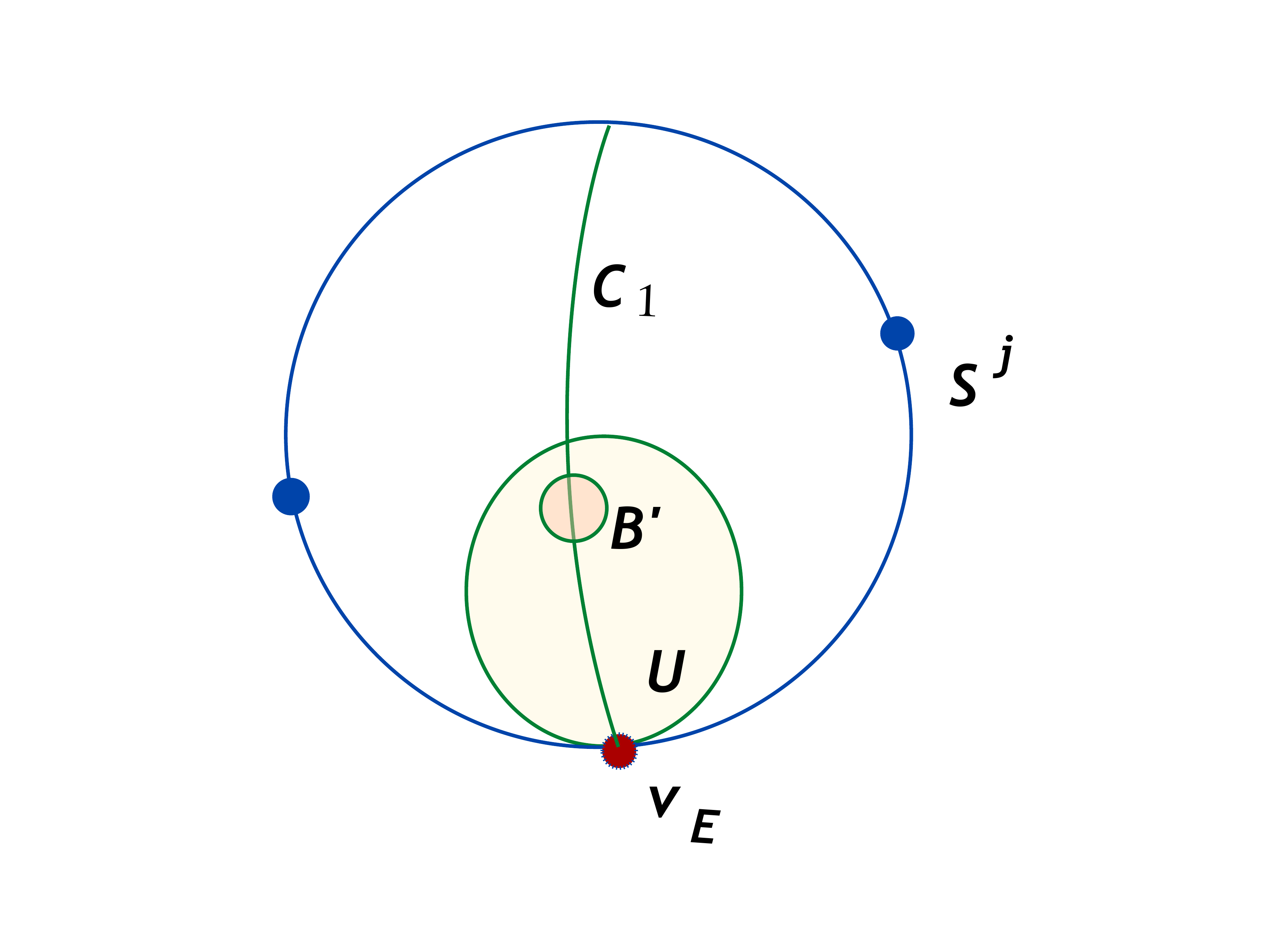}
\caption{The figure for the proof of Proposition \ref{prop-eigSI}. }
\label{fig:compl}
\end{figure}

Hence, a small ball $B'$ in $U$ meets $C_1$ in its interior.
\begin{align}
& \hbox{For any } [v] \in B', v \in \bR^{n+1}, v = v_1 + v_2 \hbox{ where } [v_1] \in C_1 \hbox{ and } [v_2] \in \SI^j. \nonumber \\
& \hbox{We obtain } g^k([v]) = [ g^k(v_1) + g^k(v_2)],
\end{align}
where we used $g$ to represent the linear transformation of determinant $\pm 1$ as well. 
By the real Jordan decomposition consideration, 
the action of $g^k$ as $k \ra \infty$ makes the former vectors very small compared to the latter ones, i.e., 
\[ ||g^k(v_1)||/||g^k(v_2)|| \ra 0 \hbox{ as } k \ra \infty.\]
Hence, $g^k([v])$ converges to the limit of $g^k([v_2])$ if it exists. 

Now choose $[w]$ in $C_1 \cap B'$ and $v, [v]\in \SI^j$. 
We let $w_1 = [w + \eps v]$ and $w_2 =[w -\eps v]$ in $B'$ for small $\eps> 0$. 
Choose a subsequence $\{k_i\}$ so that $g^{k_i}(w_1)$ converges to a point of $\SI^n$. 
The above estimation shows that $\{g^{k_i}(w_1)\}$ and $\{g^{k_i}(w_2)\}$ converge to 
an antipodal pair of points in $\clo(U)$ respectively. 
This contradicts the proper convexity of $U$ 
as $g^k(B'') \subset U$ and the geometric limit is in $\clo(U)$. 

Also the consideration of $g^{-1}$ completes the inequality.

\end{proof}



\subsection{The plan of the proof of Theorem \ref{thm-thirdmain}} \label{sub-plan}
We will show that our NPCC-ends are quasi-joined type ones; i.e., we prove
Theorem \ref{thm-thirdmain} by proving 
Theorems \ref{thm-NPCCcase} in Section \ref{sec-gendiscrete} and Theorem \ref{thm-NPCCcase2}
in Section \ref{sec-indiscrete}. 
We divide into two case: we study first the case when $N_{K}$ is discrete and 
when $N_{K}$ is indiscrete. 
 
We will discuss some general results. 
For results in Section \ref{sec-gendiscrete} except for Section \ref{sub-discretecase} 
we do not use a discreteness assumption on the semisimple quotient group $N_K$. 
We will use Hypotheses \ref{h-norm} and \ref{h-qjoin} generalizing our situation. 
  
\begin{itemize} 
\item We show that $\bGamma_{\tilde E}$ acts as scalar times orthogonal group 
on $\CN$ as realized as an real abelian group $\bR^{i_{0}}$. See Lemmas \ref{lem-similarity} and 
\ref{lem-conedecomp1}.  This is done by computations and coordinate change arguments and 
the distal group theory of Fried \cite{Fried86}. 
\item We refine the matrix forms in Lemma \ref{lem-similarity} when $\mu_{g}=1$. 
Here the matrices are in almost desired forms. 
\item Proposition \ref{prop-decomposition} shows the splitting of the representation of $\bGamma_{\tilde E}$. 
Basically, one uses the weak middle eigenvalue condition to realize the compact $(n-i_{0}-1)$-dimensional set 
where $\bGamma_{\tilde E}$ acts on. 
\item In Section \ref{subsec-qjoin}, we discuss joins and quasi-joins. The idea is to show that 
the join cannot occur by Propositions  \ref{II-prop-joinred}  and \ref{II-prop-decjoin} of \cite{EDC2}. 
This will settle for the cases of discrete $N_{K}$ by Theorem \ref{thm-NPCCcase} in Section \ref{sub-discretecase}. 
\end{itemize} 
In Section \ref{sec-indiscrete}, we will settle for the cases of indiscrete $N_{K}$
where we will use these methods.

\section{The general theory and the discrete case } \label{sec-gendiscrete}

Now, we will be working on projective sphere $\SI^n$ only for while. 
Suppose that the semisimple quotient group $N_K$ is a discrete subgroup of $\Aut(\SI^{n-i_0-1})$, 
which is a much simpler case to start. 
$N_K$ virtually equals the cocompact subgroup of the group 
\[\bZ^{l_0-1} \times \Gamma_1 \times \cdots \times \Gamma_{l_0}\] 
since each factor $\Gamma_i$ commutes with the other factors 
and acts trivially on $K_j$ for $j \ne i$ as was shown in the proof of Theorem \ref{II-thm-redtot} of \cite{EDC2}
and $N_K$ acts cocompactly on $K$.
To begin with we do not assume $N_{K}$ is discrete until the last subsection. 

\subsection{Examples} 
First, we give some examples. 

\subsubsection{The standard quadric in $\bR^{i_0+1}$ and the group acting on it.} \label{subsub-quadric}

Let us consider an affine subspace $A^{i_0+1}$ of $\SI^{i_0+1}$ with 
coordinates $x_0, x_1, \dots, x_{i_0+1}$ given by $x_0 > 0$.
The standard quadric in $A^{i_0+1}$ is given by 
\[ x_{i_0+1} = x_1^2 + \cdots + x_{i_0}^2. \] 
Clearly the group of the orthogonal maps $O(i_0)$ acting on the planes given by $x_{i_0+1} = const$ acts on
the quadric also. Also, the matrices of the form 
\[
\left(
\begin{array}{ccc}
 1                                 & 0                    & 0  \\
 \vec{v}^T                    & \Idd_{i_0}        & 0   \\
\frac{||\vec{v}||^2}{2}  & \vec{v}   & 1  
\end{array}
\right)
\]
induce and preserve the quadric. 
They are called the standard cusp group. 

The group of affine transformations that acts on the quadric is exactly the Lie group
generated by the cusp group and $O(i_0)$. The action is transitive and each of the stabilizer is a conjugate of
$O(i_0)$ by elements of the cusp group. 

The proof of this fact is simply that such an affine transformation is conjugate to an element a parabolic group in
the $i_0+1$-dimensional complete hyperbolic space $H$ where the ideal fixed point is identified with 
$[0, \dots, 0, 1] \in \SI^{i_0+1}$ and with $\Bd H$ tangent to $\Bd A^{i_0}$.

A  {\em cusp group} is a group of projective automorphisms of form
\begin{equation} \label{eqn-generator}
\CN'(\vec{v}):= \left( \begin{array}{ccc}                
                                                         1 &  \vec{0}  & 0 \\
                                                        \vec{v}^T  & \Idd_{i_0-1} & \vec{0}^T \\ 
                                                \frac{||\vec{v}||^2}{2}  &  \vec{v}       & 1 
                                    \end{array} \right) \hbox{ for } \vec{v} \in \bR^{i_0}. 
 \end{equation}
 (see \cite{CM2} for details.)
 We can make each translation direction of generators of $\CN$ in $\tilde \Sigma_{\tilde E}$ to be one of the standard vector. 
Therefore, we can find 
a coordinate system of $V^{i_0+2}$ so that the generators are of $(i_0+2) \times (i_0+2)$-matrix forms 
\begin{equation} \label{eqn-generator2}
\CN'_j :=  \left( \begin{array}{cccc}                1 &   \vec{0} & 0 \\ 
                                                        \vec{e}^T_j  & \Idd_{i_0}    & 0 \\
                                                       \frac{1}{2} &  \vec{e}_j     & 1 
                                    \end{array} \right)
 \end{equation}
 where $(\vec{e}_{j})_{k} = \delta_{jk}$ a row $i$-vector for $j=1, \dots, i_{0}$.  
 That is, 
 \[\CN'(\vec{v}) = \CN'(v_{1})\cdots \CN'(v_{i_{0}}).\]

\subsubsection{Example of joined ends}

We first begin with examples. 
In the following, we will explain the joined type end. 

\begin{example} \label{exmp-joined} 
Let us consider two ends $E_1$, a totally geodesic R-end, 
with the p-end-neighborhood $U_1$ in the universal cover of a real projective orbifold $\mathcal{O}_1$ in $\SI^{n-i_0-1}$
of dimension $n-i_0-1$ with the p-end vertex $\bv_1$, 
and $E_2$  the p-end-neighborhood $U_2$ , a horospherical type one, in the universal cover
of a real projective orbifold $\mathcal{O}_2$ of dimension $i_0+1$
with the p-end vertex $\bv_2$. 
\begin{itemize} 
\item Let $\bGamma_1$ denote the projective automorphism group in $\Aut(\SI^{n-i_0-1})$ acting on $U_1$ 
corresponding to $E_1$. 
We assume that $\bGamma_1$ acts on a great sphere $\SI^{n-i_0-2} \subset \SI^{n-i_0-1}$ disjoint from $\bv_1$.
There exists a properly convex open domain $K'$ in $\SI^{n-i_0-2}$ where $\bGamma_1$ acts cocompactly
but not necessarily freely. 
We change $U_1$ to be the interior of the join of $K'$ and $\bv_1$. 
\item Let $\bGamma_2$ denote the one in $\Aut(\SI^{i_0+1})$ acting on $U_2$
unipotently and hence it is a cusp action. 
\item We embed $\SI^{n-i_0-1}$ and $\SI^{i_0+1}$ in $\SI^n$ meeting transversally at $\bv = \bv_1 = \bv_2$. 
\item We embed $U_2$ in $\SI^{i_0+1}$ and $\bGamma_2$ in $\Aut(\SI^n)$ fixing each point of 
$\SI^{n-i_0-1}$. 
\item We can embed $U_1$ in $\SI^{n-i_0-1}$ and $\bGamma_1$ in $\Aut(\SI^n)$ acting 
on the embedded $U_1$ so that 
$\bGamma_1$ acts on $\SI^{i_0-1}$ normalizing $\bGamma_2$ and acting on $U_1$. 
One can find some such embeddings by finding an arbitrary 
homomorphism $\rho: \bGamma_1 \ra N(\bGamma_2)$ for a normalizer $N(\bGamma_2)$ of 
$\bGamma_2$ in $\Aut(\SI^n)$. 
\end{itemize}

We find an element $\zeta \in \Aut(\SI^n)$ fixing each point of $\SI^{n-i_0-2}$ and 
acting on $\SI^{i_0+1}$ as a unipotent element normalizing $\bGamma_2$ 
so that the corresponding matrix has only two norms of eigenvalues. 
Then $\zeta$ centralizes $\bGamma_1| \SI^{n-i_0-2}$ and normalizes $\bGamma_2$. 
Let $U$ be the strict join of $U_1$ and $U_2$, a properly convex domain. 
$U/\langle \bGamma_1, \bGamma_2, \zeta \rangle$ gives us a p-R-end of dimension $n$
diffeomorphic to $\Sigma_{E_1} \times \Sigma_{E_2} \times \SI^1 \times \bR$
and the transversal real projective manifold is diffeomorphic to $\Sigma_{E_1} \times \Sigma_{E_2} \times \SI^1$. 
We call the results the {\em joined} end and the joined end-neighborhoods. 
Those ends with end-neighborhoods finitely covered by these are also called 
a {\em joined} end. The generated group $\langle \bGamma_1, \bGamma_2, \zeta \rangle$ is 
called a {\em joined group}. 

Now we generalize this construction slightly: 
Suppose that $\bGamma_1$ and $\bGamma_2$ are Lie groups and they have compact stabilizers at points of $U_1$ and $U_2$ respectively, 
and we have a parameter of $\zeta^t$ for $ t\in \bR$ centralizing $\bGamma_1| \SI^{n-i_0-2}$ and 
normalizing $\bGamma_2$ and restricting to a unipotent action on $\SI^{i_0}$ acting on $U_2$. 
The other conditions remain the same. We obtain 
a {\em joined homogeneous action} of the semisimple and cusp actions. 
Let $U$ be the properly convex open subset obtained as above as a join of 
$U_1$ and $U_2$.  Let $G$ denote the constructed Lie group by taking the embeddings of 
$\bGamma_1$ and $\bGamma_2$ as above. $G$ also has a compact stabilizer on $U$. 
Given a discrete cocompact subgroup of $G$, we obtained a p-end-neighborhood of a {\em joined p-end}
by taking the quotient of $U$. An end with an end-neighborhood finitely covered by such a one 
are also called a {\em joined end}. 

\end{example} 

\begin{remark}
Later we will show this case cannot occur. 
We will modify this construction to a construction of quasi-joined ends to be defined in Definition \ref{defn-qjoin}. 
Here, $\Gamma_2$ is not required to act on $U_2$. 
\end{remark}

We continue the above example to a more specific situation. 

\begin{example}\label{exmp:nonexmp}
Let $N$ be as in equation \eqref{eqn-nilmat}.
In fact, we let $C_1 =0$ to simplify arguments and let $N$ be a nilpotent group in 
conjugate to $\SO(i_0+1, 1)$ acting on an $i_0$-dimensional ellipsoid in $\SI^{i_0+1}$. 

We find a closed properly convex real projective 
orbifold $\Sigma$ of dimension $n-i_0-2$ and find a homomorphism from $\pi_1(\Sigma)$
to a subgroup of $\Aut(\SI^{i_0+1})$ normalizing $N$ or even $N$ itself. 
(We will use  a trivial one to begin with. )
Using this, we obtain a group $\Gamma$ so that 
\[ 1 \ra N \ra \Gamma \ra \pi_1(\Sigma) \ra 1. \] 
Actually, we assume that this is ``split'', 
i.e., $\pi_1(\Sigma)$ acts trivially on $N$.

We now consider an example where $i_0 = 1$. 
Let $N$ be $1$-dimensional and be generated by $N_1$ as in Equation \eqref{eqn-nilmat3}. 
\renewcommand{\arraystretch}{1.5}
\begin{equation} \label{eqn-nilmat3}
\newcommand*{\temp}{\multicolumn{1}{r|}{}}
N_1:= \left( \begin{array}{ccccc}         \Idd_{n-i_0-1} & 0 & \temp &  0 & 0 \\ 
                                                       \vec{0}           & 1  & \temp &  0  & 0\\
                                                       \cline{1-5}
                                                       \vec{0}         & 1 & \temp & 1   & 0 \\
                                                       \vec{0}       &\frac{1}{2} & \temp & 1  & 1 
                                    \end{array} \right)
 \end{equation}
 where $i_0 = 1$ and we set $C_1=0$.

We take a discrete faithful proximal 
representation 
\[\tilde h: \pi_1(\Sigma) \ra \GL(n-i_0, \bR)\]
acting on a convex cone $C_\Sigma$ in $\bR^{n-i_0}$.
We define 
\[h: \pi_1(\Sigma) \ra \GL(n+1, \bR)\] by matrices
\begin{equation} \label{eqn-gammaJp}
h(g):= \left( \begin{array}{ccc}  \tilde h(g)         & 0                          & 0\\ 
                                                 \vec{d}_1(g)    & a_1(g)           & 0 \\
                                                 \vec{d}_2(g)    & c(g)  & \lambda_{\bv_{\tilde E}}(g)    
                                                               \end{array} \right)
\end{equation}
where $\vec{d}_1(g)$ and $\vec{d}_2(g)$ are $n-i_0$-vectors
and $g \mapsto \lambda_{\bv_{\tilde E}}(g)$ is a homomorphism as defined above
for the p-end vertex and $\det \tilde h(g) a_1(g) \lambda_{\bv_{\tilde E}}(g) = 1$.
 \begin{equation} \label{eqn-gammaJpi}
h(g^{-1}):= \left( \begin{array}{cc}  \tilde h(g)^{-1}         &  \begin{array}{cc} 0  & 0 \end{array}  \\ 
                             -   \left(   \begin{array}{cc} \frac{\vec{d}_1(g)}{a_1(g)} \\ \frac{-c(g)\vec{d}_1(g)}{a_1(g) \lambda_{\bv_{\tilde E}}(g)} 
                                + \frac{\vec{d}_2(g)}{\lambda_{\bv_{\tilde E}}(g)}
                                                                 \end{array}   \right) \tilde h(g)^{-1} 

                                          & \begin{array}{cc} \frac{1}{a_1(g)} & 0 \\ 
                                          \frac{-c(g)}{a_1(g) \lambda_{\bv_{\tilde E}}(g)} &  \frac{1}{\lambda_{\bv_{\tilde E}}(g)} \end{array} \\ 
                         
                                                               \end{array} \right).
\end{equation}
 
Then the conjugation of $N_1$ by $h(g)$ gives us 
 \begin{equation} \label{eqn-nilmat4}
 \left( \begin{array}{cc}         \Idd_{n-i_0}  & \begin{array}{cc} 0  & 0 \end{array} \\ 
                                                      \left(\begin{array}{cc} \vec{0} & a_1(g) \\ \vec{\ast} & \ast \end{array} \right) \tilde h(g)^{-1}&
                                                      \begin{array}{cc} 1   & 0 \\ \frac{\lambda_{\bv_{\tilde E}}(g)}{a_1(g)}  & 1 \end{array} 
                                                      \end{array}
                                                      \right).
 \end{equation}
 Our condition on the form of $N_1$ shows that 
 $(0,0,\dots,0,1)$ has to be a common eigenvector by $\tilde h(\pi_1(\tilde E))$
 and we also assume that $a_1(g) = \lambda_{\bv_{\tilde E}}(g)$. 
 (Actually, we will study the case when $a_{1}(g) >  \lambda_{\bv_{\tilde E}}(g)$.)
 The last row of $\tilde h(g)$ equals 
 $(\vec{0}, \lambda_{\bv_{\tilde E}}(g))$. Thus, the semisimple part of $h(\pi_1(\tilde E))$ is reducible. 

Some further computations show that 
we can take any 
\[h: \pi_1(\tilde E) \ra \SL(n-i_0, \bR)\] with matrices of form 
\renewcommand{\arraystretch}{1.2}
\begin{equation} \label{eqn-nilmat5}
\newcommand*{\temp}{\multicolumn{1}{r|}{}}
h(g):= \left( \begin{array}{ccccc}   S_{n-i_0-1}(g) & 0 & \temp &  0 & 0 \\ 
                                                       \vec{0}           & \lambda_{\bv_{\tilde E}}(g)  &   \temp &  0  & 0\\
                                                        \cline{1-5}
                                                       \vec{0}         & 0 & \temp & \lambda_{\bv_{\tilde E}}(g)   & 0 \\
                                                       \vec{0}          & 0 & \temp & 0  & \lambda_{\bv_{\tilde E}}(g) 
                                    \end{array} \right)
 \end{equation}
 for $g \in \pi_1(\tilde E) - N$ by a choice of coordinates by the semisimple property of the $(n-i_0) \times (n-i_0)$-upper left part of $h(g)$. 
 (Of course, these are not all example we wish to consider but we will modify later to quasi-joined ends.)
 
 Since $\tilde h(\pi_1(\tilde E))$ has a common eigenvector, 
 Theorem 1.1 of Benoist \cite{Ben2} shows that the open convex domain $K$ 
 that is the image of $\Pi_K$ in this case is decomposable and 
 $N_K = N'_K \times \bZ$ for another subgroup $N'_1$ 
and the image of the homomorphism $g \in N'_K \ra S_{n-i_0-1}(g)$ can be assumed to give 
a discrete projective automorphism group acting properly discontinuously on a properly convex subset 
$K'$ in $\SI^{n-i_0-2}$ with a compact quotient. 

Let $\mathcal{E}$ be the one-dimensional ellipsoid where lower right $3\times 3$-matrix of $N_K$ acts on. 
From this, the end is of the join form 
$K^{\prime o}/N'_K \times \SI^1 \times {\mathcal{E}}/\bZ$ by taking a double cover if necessary 
and $\pi_1(\tilde E)$ is isomorphic to $N'_K \times \bZ \times \bZ$
up to taking an index two subgroups. 
(In this case, $N_K$ centralizes $\bZ \subset N'_K$
and the second $\bZ$ is in the centralizer of $\Gamma$. )

We can think of this as the join of $K^{\prime o}/N'_K$ with $\mathcal{E}/\bZ$ as
$K'$ and $\mathcal{E}$ are on disjoint complementary projective spaces
of respective dimensions $n-3$ and $2$ to be denoted $S(K')$ and $S(\mathcal{E})$ respectively. 


\end{example}

\subsection{Hypotheses to derive the splitting result}  
These hypotheses will help us to obtain the splitting. 
Afterwards, we will show the NPCC-ends with weak middle eigenvalue conditions will satisfy these. 

In Section \ref{subsub-matrixform},  we will introduce a standard coordinate system to work on, where we introduce the 
standard nilpotent group $\CN \cong \bR^{i_{0}}$ to work with. 
$\bGamma_{\tilde E}$ normalizes $\CN$ by the hypothesis. 
{\em Similarity} Lemma \ref{lem-similarity} shows that 
the conjugation in $\CN$ by an element of $\bGamma_{\tilde E}$ acts as a similarity, 
a simple consequence of the normalization property. 
We use this similarity and the Benoist theory \cite{Ben2} to prove
{\em $K$-is-a-cone} Lemma \ref{lem-conedecomp1} that $K$ decomposes into a cone $\{k\} \ast K'' $ where 
$\CN$ has a nice expression for the adopted coordinates.
(If an orthogonal group acts cocompactly on an open manifold, then the manifold is zero-dimensional.) 
In Section \ref{subsub-split}, 
{\em Splitting} Proposition \ref{prop-decomposition} shows that 
the end fundamental group splits. To do that we find a sequence of elements of the virtual center 
expanding neighborhoods of a copy  of $K''$.   
Here, we explicitly find a part corresponding 
to $K'' \subset \Bd \torb$ explicitly and $k$ is realized by an $(i_{0}+1)$-dimensional hemisphere 
where $\CN$ acts on. 


\subsubsection{The matrix form of $\bGamma_{\tilde E}$.} \label{subsub-matrixform}

Let $\bGamma_{\tilde E}$ be a p-R-end fundamental group. 
Let $V^{i_0+1}$ denote the subspace corresponding to $\SI^{i_0}_\infty$ containing $\bv_{\tilde E}$
and $V^{i_0+2}$ the subspace corresponding to $\SI^{i_0+1}_l$.
We choose the coordinate system so that 
\[\bv_{\tilde E} = \underbrace{[0, \cdots, 0, 1]}_{n+1}\]
and points of $V^{i_0+1}$ and those of $V^{i_0+2}$ respectively correspond to 
\[ \overbrace{[0, \dots, 0}^{n-i_0}, \ast, \cdots, \ast], \quad \overbrace{[0, \dots, 0}^{n-i_0-1}, \ast, \cdots, \ast].\]

Since $\SI^{i_0}_\infty$ is invariant, 
$g$, $g\in \bGamma_{\tilde E}$, is of {\em standard} form 
\renewcommand{\arraystretch}{1.2}
\begin{equation}\label{eqn-matstd}
\newcommand*{\temp}{\multicolumn{1}{r|}{}}
\left( \begin{array}{ccccccc} 
S(g) & \temp & s_1(g) & \temp & 0 & \temp & 0 \\ 
 \cline{1-7}
s_2(g) &\temp & a_1(g) &\temp & 0 &\temp & 0 \\ 
 \cline{1-7}
C_1(g) &\temp & a_4(g) &\temp & A_5(g) &\temp & 0 \\ 
 \cline{1-7}
c_2(g) &\temp & a_7(g) &\temp & a_8(g) &\temp & a_9(g) 
\end{array} 
\right)
\end{equation}
where $S(g)$ is an $(n-i_0-1)\times (n-i_0-1)$-matrix
and $s_1(g)$ is an $(n-i_0-1)$-column vector, 
$s_2(g)$ and $c_2(g)$ are $(n-i_0-1)$-row vectors, 
$C_1(g)$ is an $i_0\times (n-i_0-1)$-matrix, 
$a_4(g)$ is an  $i_0$-column vectors, 
$A_5(g)$ is an $i_0\times i_0$-matrix, 
$a_8(g)$ is an $i_0$-row vector, 
and $a_1(g), a_7(g)$, and $ a_9(g)$ are scalars. 

Denote 
\[ \hat S(g) =
\left( \begin{array}{cc} 
S(g) & s_1(g)\\ 
s_2(g) & a_1(g)
\end{array} \right),\]
and is called a semisimple part of $g$.  

Let $\CN$ be a unipotent group acting on $\SI^{i_{0}}_{\infty}$ and 
inducing $\Idd$ on $\SI^{n-i_{0}-1}$ also restricting to a cusp group for at least one great  $(i_{0}+1)$-dimensional sphere
$\SI^{i_{0} + 1}$ containing $\SI^{i_{0}}_{\infty}$.

We can write each element  $g \in \CN$ as an $(n+1)\times (n+1)$-matrix
\begin{equation} \label{eqn-nilmat}
 \left( \begin{array}{ccc}                \Idd_{n-i_0-1} &  0 &   0 \\ 
                                                       \vec{0}      & 1  &  0 \\
                                                     C_g          & *   & U_g 
                                    \end{array} \right)
 \end{equation}
where $C_g > 0$ is an $(i_0+1)\times (n-i_0-1)$-matrix, 
$U_g$ is a unipotent $(i_0+1) \times (i_0+1)$-matrix, 
$0$ indicates various zero row or column vectors, 
$\vec{0}$ denotes the zero row-vector of dimension $n-i_0-1$, and 
$\Idd_{n-i_0-1}$ is  the $(n-i_0-1) \times (n-i_0-1)$ identity-matrix. 
This follows since $g$ acts trivially on $\bR^{n+1}/V^{i_0+1}$ 
and $g$ acts as a unipotent matrix on the subspace $V^{i_0+2}$.  



For $\vec{v} \in \bR^{i_0}$, we define 
\renewcommand{\arraystretch}{1.5}
\begin{equation} \label{eqn-nilmatstd}
\newcommand*{\temp}{\multicolumn{1}{r|}{}}
\CN(\vec{v}):= \left( \begin{array}{ccccccc}         \Idd_{n-i_0-1} & 0 &\temp &  0 & 0& \dots & 0 \\ 
                                                       \vec{0}           & 1  &\temp &  0  & 0&  \dots & 0\\
                                                       \cline{1-7}
                                                       \vec{c}_1(\vec{v})    & {v}_1 &\temp & 1   & 0 &  \dots & 0 \\
                                                       \vec{c}_2(\vec{v})   & {v}_2 &\temp & 0   & 1 & \dots & 0\\
                                                       \vdots   & \vdots &\temp & \vdots & \vdots & \ddots & \vdots \\
                                                       \vec{c}_{i_0+1}(\vec{v})  & \frac{1}{2}||\vec{v}||^2& \temp & {v}_1 & v_2 & \dots  & 1
                                                        
                                    \end{array} \right)
 \end{equation}
where $||v||$ is the norm of $\vec{v} = (v_1, \cdots, v_i) \in \bR^{i_0}$. 
We assume that \[\CN := \{ \CN(\vec{v})| \vec v \in \bR^{i_{0}}\}\] is a group, which must be nilpotent. 
The elements of our nilpotent group 
$\CN$ are of this form since $\CN(\vec{v})$ is the product  
$\prod_{j=1}^{i_0} \CN(e_j)^{v_j}$.  
By the way we defined this, 
for each $k$, $k=1, \dots, i_0$, $\vec{c}_k:\bR^{i_0} \ra \bR^{n-i_0-1}$ are  linear functions of $\vec{v}$ 
defined as 
\[\vec{c}_k(\vec{v}) = \sum_{j=1}^{i_0} \vec c_{kj} v_j \hbox{ for } \vec{v} = (v_1, v_2, \dots, v_{i_0})\]
so that we form a group. 
(We do not need the property of $\vec{c}_{i_0+1}$ at the moment.)

We denote by $C_1(\vec{v})$ the $(n-i_0-1) \times i_0$-matrix given by 
the matrix with rows $\vec{c}_j(\vec{v})$ for $j= 1, \dots, i_0$
and by $c_2(\vec{v})$ the row $(n-i_0-1)$-vector $\vec{c}_{i_0+1}(\vec{v})$. 
The lower-right $(i_0+2)\times (i_0+2)$-matrix is form is called the {\em standard cusp matrix form}.

The assumptions for this subsection are as follows: 
\begin{hypothesis} \label{h-norm}
\begin{itemize} 
\item  Let $K$ be defined as above for a p-R-end $\tilde E$.
Assume that $K^{o}/N_K$ is a compact set.
\item $\bGamma_{\tilde E}$ satisfies the weak middle eigenvalue condition.
And elements are in the matrix form of equation \eqref{eqn-matstd} under a common coordinate system. 
\item A group $\CN$ of form \eqref{eqn-nilmatstd}  
acts on each hemisphere with boundary $\SI^i_{\infty}$, and fixes
$\bv_{\tilde E} \in \SI^i_\infty$. 
\item The p-end fundamental group $\bGamma_{\tilde E}$ normalizes 
$\CN$ also in the above coordinate system. 
\item $\CN$ acts on a p-end neighborhood $U$ of $\tilde E$.
\item $\CN$ acts on the space of $i_{0}$-dimensional leaves of $\tilde \Sigma_{\tilde E}$
by an induced action. 
\end{itemize}
\end{hypothesis}

Let $U$ be a p-end neighborhood of $\tilde E$. 
Let $l'$ be an $i_{0}$-dimensional leaf of $\tilde  \Sigma_{\tilde E}$. 
The consideration of the projection $\Pi_{K}$ shows us that 
the leaf $l'$ corresponds to a hemisphere $H^{i_0+1}_{l'}$ where
\[(H^{i_0+1}_{l'} - \SI^{i_0}_\infty) \cap U \ne \emp\] holds. 

\begin{lemma}[Cusp]\label{lem-bdhoro} 
Assume Hypothesis \ref{h-norm}. 
Let $l'$ be an $i_{0}$-dimensional leaf of $\tilde  \Sigma_{\tilde E}$. 
Let $H^{i_{0}+1}_{l'}$ denote the $i_{0}+1$-dimensional hemisphere with boundary $\SI^{i_{0}}_{\infty}$ 
corresponding to $l'$.  
Then $\CN$ acts on the open ball $U_{l'}$ in $U$ bounded by an ellipsoid in 
a component of $H^{i_0+1}_{l'} - \SI^{i_0}_\infty$. 
\end{lemma}
\begin{proof} 


Since $l'$ is an $i_{0}+1$-dimensional leaf of $\tilde \Sigma_{\tilde E}$, we obtain
$H^{i_{0}}_{l'} \cap U \ne \emp$. 
Let $J_{l'}:= H^{i_{0}+1}_{l'} \cap U \ne \emp$. 

$l'$ corresponds to an interior point of $K$. 
We need to change coordinates of $\SI^{n-i_{0}-1}$ so that $l'$ goes to $[0, 0, \dots, 1]$ under $\Pi_{K}$. 
This involves the coordinate changes of the first $n-i_{0}$ coordinates. 
Now, we can restrict $g$ to $H^{i_{0}+1}_{l'}$ so that the matrix form is truly what acts on $A_{l'}$. 

Using equation \eqref{eqn-nilmatstd} and the fact that $\vec{c}_{i}, i=1, \dots, i_{0}$ are linear on $\vec{v}$, 
we obtain that 
each $g \in \CN$ then has the form in $H_{l'}^{i_0+1}$ as 
\[
\left(
\begin{array}{ccc}
1              & 0         &  0 \\
L(\vec{v}^T)  & \Idd_{i_0}  & 0  \\
\kappa(\vec{v})  & \vec{v}  & 1  
\end{array}
\right)
\]
since the $\SI^{i_0}_\infty$-part, i.e., the last $i_0+1$ coordinates, is not changed from one for 
equation \eqref{eqn-nilmatstd}
where $L: \bR^{i_0} \ra \bR^{i_0}$ is a linear map. The linearity of $L$ is the consequence of the group property. 
$\kappa: \bR^{i_0} \ra \bR$ is some function. We consider $L$ as an $i_0\times i_0$-matrix. 

If there exists a kernel $K_1$ of $L$, then we use $t\vec{v} \in K_1-\{O\}$ and 
as $t \ra \infty$, we can show that $\CN(J_{l'})$ cannot be properly convex. 

Also, since $\CN$ is abelian, the computations of 
\[\CN(v)\CN(w) = \CN(w)\CN(v)\] 
shows that $\vec v L \vec w^T = \vec w L \vec v^T$
for every pair of vectors $\vec v$ and $\vec w$ in $\bR^{i_0}$. 
Thus, $L$ is a symmetric matrix. 

We may obtain new coordinates $x_{n-i_0+1}, \dots, x_n$ by taking 
linear combinations of these. 
Since $L$ hence is nonsingular, we can find 
new coordinates $x_{n-i_0+1}, \dots, x_n$ so that $\CN$ is now of standard form: 
We conjugate $\CN$ by 
\[
\left(
\begin{array}{ccc}
1              & 0         &  0 \\
0 & A  & 0  \\
0 &  0 & 1  
\end{array}
\right)
\]
for nonsingular $A$. 
We obtain
\[
\left(
\begin{array}{ccc}
1              & 0         &  0 \\
AL\vec{v}^T  & \Idd_{i_0}  & 0  \\
\kappa(\vec{v})  & \vec{v} A^{-1}  & 1  
\end{array}
\right).
\]
We thus need to solve for $A^{-1} A^{-1 T} = L$, which can be done. 

We can factorize each element of $\CN$ into forms 
\[
\left(
\begin{array}{ccc}
1              & 0         &  0 \\
0  & \Idd_{i_0}  & 0  \\
\kappa(\vec{v}) - \frac{||\vec{v}||^2}{2}  & 0  & 1  
\end{array}
\right)
\left(
\begin{array}{ccc}
1              & 0         &  0 \\
\vec{v}^T  & \Idd_{i_0}  & 0  \\
 \frac{||\vec{v}||^2}{2}   & \vec{v}  & 1  
\end{array}
\right).
\]
Again, by the group property, $\alpha_7(\vec{v}) : = \kappa(\vec{v}) - \frac{||\vec{v}||^2}{2}$ 
gives us a linear function $\alpha_7: \bR^{i_0} \ra \bR$. 
Hence $\alpha_7(\vec{v}) = \kappa_\alpha \cdot \vec{v}$
for $\kappa_\alpha \in \bR^{i_0}$. 
Now, we conjugate $\CN$ by 
the matrix 
\[
\left(
\begin{array}{ccc}
 1 & 0  & 0  \\
 0 & \Idd_{i_0}  & 0  \\
0  & -\kappa_\alpha  & 1  
\end{array}
\right)
\]
and this will put $\CN$ into the standard form. 

Now it is clear that the orbit of $\CN(x_0)$ for a point $x_0$ of $J_{l'}$ is an ellipsoid with a point removed. 
as we can conjugate so that the first column entries from the second one to the $(i_0+1)$-th one equals those of the last row.
Since $\clo(U)$ is $\CN$-invariant, we obtain that $\CN(x_0) \subset J_{l'}$.

\end{proof}

Let $a_5(g)$ denote $\left| \det(A^5_g)\right|^{\frac{1}{i_0}}$.
Define $\mu_g:= \frac{a_5(g)}{a_1(g)} = \frac{a_9(g)}{a_5(g)}$ for $g \in \bGamma_{\tilde E}$
from Lemma \ref{lem-similarity}. 

\begin{lemma}[Similarity] \label{lem-similarity}
Assume Hypothesis \ref{h-norm}. 
Then 
any element $g \in \bGamma_{\tilde E}$ 
induces an $(i_0\times i_0)$-matrix $M_g$ given by
\[g \CN(\vec{v}) g^{-1} = \CN(\vec{v}M_g) \hbox{ where } \] 
\[M_g = \frac{1}{a_1(g)} (A_5(g))^{-1} = \mu_g O_5(g)^{-1} \]
for $O_5(g)$ in a compact Lie group $G_{\tilde E}$, and 
the following hold. 
\begin{itemize} 
\item $(a_5(g))^2 = a_1(g) a_9(g)$ or equivalently $\frac{a_5(g)}{a_1(g)}= \frac{a_9(g)}{a_5(g)}$.
\item Finally, $a_1(g), a_5(g),$ and $a_9(g)$ are all nonzero. 
\end{itemize}  
\end{lemma} 
\begin{proof}
Since the conjugation by $g$ sends elements of $\CN$ to itself in a one-to-one manner, 
the correspondence between the set of $\vec{v}$ for $\CN$ and $\vec{v'}$ is one-to-one.

Since we have $g \CN({\vec{v}}) = \CN({\vec{v}'}) g$ for vectors $\vec{v}$ and $\vec{v'}$ 
in $\bR^{i_0}$ by Hypothesis \ref{h-norm},
we consider
\begin{equation} \label{eqn-firstm}
\newcommand*{\temp}{\multicolumn{1}{r|}{}}
\left( \begin{array}{ccccccc} 
S(g) & \temp & s_1(g) & \temp & 0 & \temp & 0 \\ 
 \cline{1-7}
s_2(g) &\temp & a_1(g) &\temp & 0 &\temp & 0 \\ 
 \cline{1-7}
C_1(g) &\temp & a_4(g) &\temp & A_5(g) &\temp & 0 \\ 
 \cline{1-7}
c_2(g) &\temp & a_7(g) &\temp & a_8(g) &\temp & a_9(g) 
\end{array} 
\right)
\left( \begin{array}{ccccccc} 
\Idd_{n-i_0-1} & \temp & 0 & \temp & 0 & \temp & 0 \\ 
 \cline{1-7}
0 &\temp & 1 &\temp & 0 &\temp & 0 \\ 
 \cline{1-7}
C_1({\vec{v}}) &\temp & \vec{v}^T &\temp & \Idd_{i_0} &\temp & 0 \\ 
 \cline{1-7}
c_2({\vec{v}}) &\temp & \frac{||\vec{v}||^2}{2} &\temp & \vec{v} &\temp & 1 
\end{array} 
\right)
\end{equation}
where $C_1({\vec{v}})$ is an $(n-i_0-1)\times i_0$-matrix where
each row is a linear function of $\vec{v}$, 
$c_2({\vec{v}})$ is a $(n-i_0-1)$-row vector, and
$\vec{v}$ is an $i_0$-row vector. 
This must equal the following matrix
for some $\vec{v'}\in \bR$
\begin{equation} \label{eqn-secondm}
\newcommand*{\temp}{\multicolumn{1}{r|}{}}
\left( \begin{array}{ccccccc} 
\Idd_{n-i_0-1} & \temp & 0 & \temp & 0 & \temp & 0 \\ 
 \cline{1-7}
0 &\temp & 1 &\temp & 0 &\temp & 0 \\ 
 \cline{1-7}
C_1({\vec{v'}}) &\temp & \vec{v'}^T &\temp & \Idd_{i_0} &\temp & 0 \\ 
 \cline{1-7}
c_2({\vec{v'}}) &\temp & \frac{||\vec{v'}||^2 }{2} &\temp & \vec{v'} &\temp & 1 
\end{array} 
\right)
\left( \begin{array}{ccccccc} 
S(g) & \temp & s_1(g) & \temp & 0 & \temp & 0 \\ 
 \cline{1-7}
s_2(g) &\temp & a_1(g) &\temp & 0 &\temp & 0 \\ 
 \cline{1-7}
C_1(g) &\temp & a_4(g) &\temp & A_5(g) &\temp & 0 \\ 
 \cline{1-7}
c_2(g) &\temp & a_7(g) &\temp & a_8(g) &\temp & a_9(g) 
\end{array} 
\right).
\end{equation}
From equation \eqref{eqn-firstm}, 
we compute the $(4, 3)$-block of the result 
to be $a_8(g) + a_9(g) \vec{v}$. 
From Equation \eqref{eqn-secondm}, 
the $(4, 3)$-block is
$\vec{v'} A_5(g) + a_8(g)$. We obtain the relation
$a_9(g) \vec{v} = \vec{v'} A_5(g)$ for every $\vec{v}$. 
Since the correspondence between $\vec{v}$ and $\vec{v'}$ is one-to-one, 
we obtain 
\begin{equation}\label{eqn-vp1}
\vec{v'} = a_9(g) \vec{v} (A_5(g))^{-1}
\end{equation}
for the $i_0\times i_0$-matrix $A_5(g)$ and we also infer $a_9(g) \ne 0$
and $\det(A_5(g)) \ne 0$. 
The $(3, 2)$-block of the result of Equation \eqref{eqn-firstm} 
equals 
\[a_4(g) + A_5(g) \vec{v}^T \]  
The $(3, 2)$-block of the result of equation \eqref{eqn-secondm} 
equals
\begin{equation}
C_1({\vec{v}'}) s_1(g) + a_1(g) \vec{v}^{\prime T} + a_4(g). 
\end{equation}
Thus,
\begin{equation} \label{eqn-sim0}
A_5(g) \vec{v}^T =  C_1({\vec{v}'}) s_1(g)  + a_1(g) \vec{v}^{\prime T}.
\end{equation} 


For each $g$, we can choose a coordinate system so that $s_1(g) = 0$ 
as $\hat S(g)$ is semisimple, 
which involves the coordinate changes of the first $n-i_0$ coordinate functions only. 



Since $\CN$ acts on $\SI^{i_0+1}_{l'}$ for some leaf $l'$ as a cusp group by Lemma \ref{lem-bdhoro},
there exists a coordinate change involving the last $(i_0+1)$-coordinates 
\[x_{n-i_0+1}, \dots, x_n, x_{n+1}\]
so that the matrix form of the lower-right $(i_0+2)\times(i_0+2)$-matrix of each element $\CN$ is of the standard cusp form.
This will not affect $s_1(g) = 0$ as we can check from the proof of Lemma \ref{lem-bdhoro}
as the change involves the above coordinates only. 
Denote this coordinate system by $\Phi_{g, l'}$. 

Let us use $\Phi_{g, l'}$ for a while
using primes for new set of coordinates functions. 
Now $A'_5(g)$ is conjugate to $A_5(g)$ as we can check in the proof of Lemma \ref{lem-bdhoro}. 
Under this coordinate system for given $g$, 
we obtain $a'_1(g) \ne 0$ and we can recompute to show that 
$a'_9(g) \vec{v} = \vec{v'} A'_5(g)$ for every $\vec{v}$ as in equation \eqref{eqn-vp1}. 
By equation \eqref{eqn-sim0} recomputed for this case, we obtain 
\begin{equation}\label{eqn-vp2}
\vec{v'} = \frac{1}{a'_1(g)} \vec{v} (A'_5(g))^T
 \end{equation}
 as $s'_1(g) = 0$ here since we are using the coordinate system $\Phi_{g, l'}$.
 Since this holds for every $\vec{v} \in \bR^{i_0}$, 
 we obtain 
 \[a'_9(g) (A'_5(g))^{-1} = \frac{1}{a'_1(g)} (A'_5(g))^T.\] 
 Hence $\frac{1}{ |\det(A'_5(g))|^{1/i_0}} A'_5(g) \in O(i_0)$. 
 Also, \[\frac{a'_9(g)}{a'_5(g)} = \frac{a'_5(g)}{a'_1(g)}.\] 
 Here, $A'_5(g)$ is a conjugate of the original matrix $A_5(g)$ by linear coordinate changes 
 as we can see from the above processes to obtain the new coordinate system.
 
 This implies that the original matrix $A_5(g)$ is conjugate to an orthogonal matrix multiplied by a positive scalar for every $g$. 
The set of matrices $ \{ A_5(g)| g \in \bGamma_{\tilde E}\}$ forms a group since every $g$ is of a standard matrix form 
(see equation \eqref{eqn-matstd}). 
  Given such a group of matrices normalized to have determinant $\pm 1$, we obtain 
 a compact group 
 \[G_{\tilde E}:= \Bigg\{ \frac{1}{|\det A_5(g)|^{\frac{1}{i_0}}} A_5(g) \Bigg| g \in \Gamma_{\tilde E} \Bigg\}\]
  by Lemma \ref{lem-cpt}. 
This group has a coordinate system where every element is orthogonal by a coordinate change
of coordinates $x_{n-i_0+1}, \dots, x_n$. 

 
\end{proof} 


\begin{lemma}\label{lem-cpt} 
Suppose that $G$ is a subgroup of a linear group $\GL(i_0, \bR)$ where each 
element is conjugate to an orthogonal element. Then $G$ is a compact group. 
\end{lemma}
\begin{proof} 
Clearly, the norms of eigenvalues of $g \in G$ are all $1$. 
$G$ is virtually an orthopotent group by \cite{CG2} or \cite{Moore}. 
Since the group is linear and for each element $g$, 
$\{g^n| n \in \bZ\}$ is a bounded collection of matrices, 
$G$ is a subgroup of an orthogonal group under a coordinate system. 
\end{proof}

We denote by $(C_1({\vec{v}}), \vec{v}^T)$ the matrix obtained 
from $C_1(\vec{v})$ by adding a column vector $\vec{v}^T$.


\begin{lemma}[$K$ is a cone] \label{lem-conedecomp1}
Assume Hypothesis \ref{h-norm}. 
Then the following hold{\rm :} 
\begin{itemize} 
\item $K$ is a cone over a totally geodesic 
$(n-i_0-2)$-dimensional domain $K''$. 
\item The rows of $(C_1({\vec{v}}), \vec{v}^T)$ are proportional to a single vector and 
we can find a coordinate system where $C_1({\vec v}) = 0$
not changing any entries of the lower-right $(i_0+2)\times (i_0+2)$-submatrices for all $\vec v \in \bR^{i_0}$. 
\item We can find a common coordinate system where 
\begin{equation}\label{eqn-O5coor}
O_5(g)^{-1} = O_5(g)^T, O_5(g) \in O(i_0), s_1(g) = s_2(g) = 0 \hbox{ for all } g \in \bGamma_{\tilde E}.
\end{equation} 
\item In this coordinate system, we have
\begin{equation} \label{eqn-conedecomp1}
 a_9(g) c_2({\vec{v}})  
 = c_2({\mu_g\vec{v}O_5(g)^{-1}}) S(g)  + \mu_g \vec{v} O_5(g)^{-1} C_1(g).
\end{equation} 


\end{itemize}
\end{lemma}
\begin{proof} 

The assumption implies that $M_g =  \mu_g O_5(g)^{-1}$ by Lemma \ref{lem-similarity}.
We consider the equation 
\begin{equation} \label{eqn-conj0}
g \CN(\vec{v}) g^{-1} = \CN(\mu_g \vec{v} O_5(g)^{-1}).
\end{equation}



We change to 
\begin{equation}\label{eqn-gN}
g \CN(\vec v) = \CN(\mu_g \vec v O_5(g)^{-1}) g.
\end{equation}
Considering the lower left $(n-i_0) \times (i_0+1)$-matrix of the left side of equation \eqref{eqn-gN},  
we obtain 
\begin{equation} \label{eqn-mm1}
\left(\begin{array}{cc} 
C_1(g) & a_4(g) \\ c_2(g) & a_7(g) \end{array} \right) 
+ 
\left( \begin{array}{cc} 
a_5(g) O_5(g) C_1({\vec{v}}) & a_5(g) O_5(g) \vec{v}\\ 
a_8(g) C_1({\vec{v}}) + a_9 c_2({\vec{v}})  
& a_8(g) {\cdot} \vec{v}^T + a_9(g) \vec{v}\cdot\vec{v}/2
\end{array} \right) \end{equation}
where the entry sizes are clear. 
From the right side of equation \eqref{eqn-gN}, we obtain
\begin{equation} \label{eqn-mm2} 
\begin{split} 
\left( \begin{array}{cc} 
C_1({\mu_g \vec{v} O_5(g)^{-1}})& \mu_g O_5(g)^{-1, T}\vec{v}^T \\ 
c_2({\mu_g\vec{v}}O_5(g)^{-1})     & \vec{v}\cdot\vec{v}/2
\end{array} \right) \hat S(g) + \\
 \left(\begin{array}{cc} 
C_1(g) & a_4(g) \\ \vec{v} \cdot C_1(g) + c_2(g) & a_7(g) + \vec{v} \cdot a_4(g) 
\end{array} \right).
\end{split}
\end{equation}
From the top rows of equations \eqref{eqn-mm1} and \eqref{eqn-mm2},  we obtain that 
\begin{equation}\label{eqn-C1} 
\begin{split}
 \bigg(a_5(g) O_5(g) C_1({\vec{v}}) , \, & a_5(g) O_5(g) \vec{v}^T \,\bigg) = \\
& \bigg( \mu_g C_1\left({\vec{v}}  O_5(g)^{-1}\right),  \mu_g O_5(g)^{-1, T} \vec{v}^T  \bigg) \hat S(g). 
 \end{split}
 \end{equation} 
 We multiplied the both sides by $O_5(g)^{-1}$ from the right and by $\hat S(g)^{-1}$ from the left to obtain
 \begin{equation}\label{eqn-C2} 
 \begin{split} 
\bigg(a_5(g) C_1({\vec{v}}) , \, & a_5(g) \vec{v}^T\bigg) \hat S(g^{-1})
=   \\
& \bigg(\mu_g O_5(g)^{-1} C_1({\vec{v}}  O_5(g)^{-1} ),  \mu_g O_5(g)^{-1}  O_5(g)^{-1, T} \vec{v}^T\,\bigg).
\end{split}
\end{equation}
Let us form the subspace $V_C$ in the dual sphere $\bR^{n-i_0 \ast}$ spanned by 
row vectors of $(C_1({\vec{v}}), \vec{v}^T)$. Let $\SI_C^\ast$ denote the corresponding subspace 
in $\SI^{n-i_0-1 \ast}$. Then 
\[\left\{\frac{1}{\det \hat S(g)^{\frac{1}{n-i_0-1}}}\hat S(g) \vert g \in \bGamma_{\tilde E}\right\}\] 
acts on $V_C$ as a group of bounded 
linear automorphisms since $O_5(g) \in G$ for a compact group $G$.
Therefore, $\{\hat S(g)| g\in \bGamma_{\tilde E}\}$ on $\SI_C^\ast$ is in a compact group of 
projective automorphisms by equation \eqref{eqn-C2}.

We recall that the dual group $N_K^*$ of $N_K$ acts on the properly convex dual domain $K^*$ of $K$ by Theorem \ref{thm-dualdiff}.
Then $g$ acts as an element of a compact group on $\SI_C^\ast$. Thus, $N_K^*$ is reducible. 

We claim that $\dim(\SI^{\ast}_C) = 0$. 
Let $\SI^{\ast}_M$ be the maximal invariant subspace containing $\SI^{\ast}_C$  where each $g\in N_K^*$ acts orthogonally.  
Now, we apply the theory of Benoist \cite{Ben3}.
Since $N_K^*$ is semisimple, 
$N_K^*$  acts on 
a complementary subspace of $\SI^{\ast}_N$. 
$K^*$ has an invariant subspace $K^*_1$ and $K^*_2$ 
so that we have strict join 
 \[K^* = K^*_1 \ast K^*_2 \hbox{ where }  \dim K^*_1 = \dim \SI_M^{\ast}, \dim K^*_2 = \dim \SI_N^{\ast}\] 
where
\[K^*_1 = K^* \cap \SI_M^{\ast}, K^*_2 = K^{\ast}\cap \SI_N^{\ast}.\] 
Also, $N_K^*$ is isomorphic to a cocompact subgroup of 
\[N_{K, 1}\times N_{K, 2} \times A, A \subset \bR\] 
and $N_{K, i}$ acts on a properly convex domain that is
the interior of $K^*_i$ properly and cocompactly for $i=1,2$.
But since $N_{K, 1}$ acts orthogonally on $\SI_M$, 
the only possibility is that $\dim \SI_M = 0$.  Hence, $\dim \SI_{C} =0$. 

Rows of $(C_{1}(\vec{v}), \vec{v}^{T})$ are elements of the $1$-dimensional subspace in $\bR^{n-i_{0}-1\ast}$ 
corresponding to $\SI_{C}^{\ast}$. 
Therefore this shows that the rows of $(C_1({\vec{v}}), \vec{v}^T)$ are proportional to a single row vector.

Since $(C_1({\vec{e}_j}), \vec{e}_j^T)$ has $0$ as the last column element except 
for the $j$th one, only the $j$th row of $C_{1}(\vec{e}_{j})$ is nonzero. 
Let $C_1({1, \vec{e}_1})$ be the first row of $C_1({\vec{e}_1})$. 
Thus, each row of $(C_{1}(\vec{e}_{j}), \vec{e}_{j}^{T})$ equals to a scalar multiple of 
$(C_1({1, \vec{e}_1}), 1)$ for every $j$. 
Now we can choose coordinates of $\bR^{n-i_0 \ast}$ so that 
this row vector now has a coordinate $(0, \dots, 0, 1)$. 
We can also choose so that $K^*_1$ is given by setting the last coordinate be zero.  
With this change, we need to do conjugation by matrices 
with the top left $(n-i_0-1)\times (n-i_0-1)$-submatrix being different from $\Idd$ and 
the rest of the entries staying the same. 
This will not affect the expressions of matrices of 
 lower right $(i_0+2)\times (i_0+2)$-matrices involved here. 
Thus, $C_1({\vec{v}}) =0$ in this coordinate for all $\vec{v} \in \bR^{i_0}$ and $g \in \bGamma_{\tilde E} - N$. 
Also, $\overbrace{ [0, \dots, 0, 1]}^{n-i_0}$ is an eigenvector of every elements of $N_{K}^{\ast}$. 

The hyperspace containing $K^{\ast}_{1}$ is also $N_{K}^{\ast}$-invariant. 
Thus,  $\overbrace{ [0, \dots, 0, 1]}^{n-i_0}$ corresponds to an eigenvector of every elements of $N_{K}$. 

And in this coordinate system, $K$ is a strict join of a point 
\[k =\overbrace{ [0, \dots, 0, 1]}^{n-i_0}\] 
and a domain $K''$ given by
setting $x_{n-i_0} = 0$ in a totally geodesic 
sphere of dimension $n-i_0-2$ by duality. 
We also obtain
\[s_1(g) =0, s_2(g) =0.\]

For the final item we have under our coordinate system. 
\begin{equation} \label{eqn-form1}
g = \left( \begin{array}{cccc} 
 S(g) & 0 & 0 & 0 \\ 
 0 & a_1(g) & 0 & 0 \\ 
C_1(g) & a_4(g) & a_5(g) O_5(g) & 0  \\
c_2(g) & a_7(g) & a_8(g) & a_9(g) 
\end{array} \right), 
\end{equation} 
\begin{equation} \label{eqn-form2}
 \CN(\vec{v}) = \left( \begin{array}{cccc} 
 \Idd & 0 & 0 & 0 \\ 
 0 &    1 & 0 & 0 \\ 
0 & \vec{v}^T & \Idd & 0 \\ 
 c_2({\vec{v}}) & \frac{1}{2} ||\vec{v}||^2 & \vec{v} & 1 
 \end{array} \right). 
 \end{equation} 
 Here we might need to change the last $i_{0}$ coordinates as done in the last part of the proof of Lemma \ref{lem-bdhoro}.  
  
 The normalization of $\CN$ shows as in the proof of Lemma \ref{lem-similarity} that $O_5(g)$ is orthogonal now.
 (See equations  \eqref{eqn-vp1} and  \eqref{eqn-sim0}.)
By equation \eqref{eqn-conj0}, we have 
 \[g\CN(\vec{v}) = \CN(\vec{v'})g, v'= \mu_g \vec{v} O_5(g)^{-1}.\]
 We consider the lower-right $(i_0+1)\times (n-i_0)$-submatrices of 
 $g\CN(\vec{v})$ and $\CN(\vec{v'})g$. 
 For the first one, we obtain 
 \[
 \left( \begin{array}{cc} C_1(g) & a_4(g) \\ c_2(g) & a_7(g) \end{array} \right ) 
 + \left( \begin{array}{cc} a_5(g) O_5(g) & 0 \\ a_8(g) & a_9(g) \end{array} \right) 
 \left( \begin{array}{cc} 0 & \vec{v}^T \\ c_2({\vec{v}}) &  \frac{1}{2} ||\vec{v}||^2 \end{array} \right) 
 \]
 For $\CN(\vec{v'})g$, we obtain 
  \[
  \left( \begin{array}{cc} 0 & \vec{v'}^T \\ c_2({\vec{v'}}) & \frac{1}{2} ||\vec{v}'||^2 \end{array} \right) 
 \left( \begin{array}{cc} S(g) & 0 \\ 0 & a_1(g) \end{array} \right ) 
 + \left( \begin{array}{cc}  \Idd & 0 \\ \vec{v'} & 1 \end{array} \right) 
 \left( \begin{array}{cc} C_1(g) & a_4(g) \\ c_2(g) & a_9(g) \end{array} \right).
 \]
Considering $(2, 1)$-blocks, we obtain 
\[ c_2(g) + a_9(g) c_2({\vec{v}}) = c_2({\vec{v'}}) S(g) 
+ \vec{v'} C_1(g) 
+ c_2(g).\]

\end{proof}


\begin{lemma} \label{lem-matrix}
Assume Hypothesis \ref{h-norm}. 
Then we can find coordinates so that the following holds for all $g$\, {\rm :} 
\begin{align} 
\frac{a_{9}(g)}{a_{5}(g)} O_5(g)^{-1} a_4(g) &= a_8(g)^{T} \hbox{ or }\frac{a_{9}(g)}{a_{5}(g)}   a_4(g)^T O_5(g) = a_8(g),\\
\hbox{If } \mu_{g} =1, \hbox{ then } &
a_1(g) = a_9(g)  = \lambda_{\bv_{\tilde E}}(g) \hbox{ and } 
A_5(g) = \lambda_{\bv_{\tilde E}}(g) O_5(g). 
\end{align}
\end{lemma} 
\begin{proof} 

Again, we use equations \eqref{eqn-firstm} and \eqref{eqn-secondm}. 
We need to only consider lower right $(i_0+2)\times (i_0+2)$-matrices. 
\begin{align} 
&\left( \begin{array}{ccc} 
a_1(g) & 0 & 0 \\ a_4(g) & a_5(g) O_5(g)  & 0 \\ a_7(g) & a_8(g) & a_9(g) \end{array} \right) 
\left(\begin{array}{ccc} 
1 & 0 & 0 \\ \vec{v}^T & \Idd & 0 \\ \frac{1}{2}||\vec{v}||^2 & \vec{v} & 1 \end{array} \right) \\
& = \left( \begin{array}{ccc} 
a_1(g) & 0 & 0 \\ a_4(g) + a_5(g)  O_5(g) \vec{v}^T & a_5(g) O_5(g) & 0 
\\ a_7(g) + a_8(g) \vec{v}^T + \frac{a_9(g)}{2} ||\vec{v}||^2 & a_8(g) + a_9(g) \vec{v} & a_9(g) 
\end{array} \right). 
\end{align}
This equals 
\begin{align}
& \left( \begin{array}{ccc} 
1 & 0 & 0 \\ \vec{v'}^T & \Idd & 0 \\ \frac{1}{2}||\vec{v'}||^2 & \vec{v'} & 1 \end{array} \right) 
\left( \begin{array}{ccc} 
a_1(g) & 0 & 0 \\ a_4(g) & a_5(g) O^5_g & 0 \\ a_7(g) & a_8(g) & a_9(g) \end{array} \right) \\
& = \left(\begin{array}{ccc} 
a_1(g) & 0 & 0 \\ a_1(g) \vec{v'}^T + a_4(g) & a_5(g) O_5(g) & 0 \\ 
\frac{a_1(g)}{2} ||\vec{v'}||^2 + \vec{v'} a_4(g) + a_7(g) & 
a_5(g) \vec{v'} O_5(g) + a_8(g) & a_9(g) \end{array} \right).
\end{align}
Then by comparing the $(3, 2)$-blocks, 
we obtain 
\[a_8(g) + a_9(g) \vec{v} = a_8(g) + a_5(g) \vec{v'} O_5(g) .\]
Thus, $\vec{v} =  \frac{a_{5}(g)}{a_{9}(g)}\vec{v'} O_5(g).$

From the $(3, 1)$-blocks, we obtain 
\[ a_1(g) \vec{v'}\cdot\vec{v'}/2 + \vec{v'}a_4(g) = a_8(g)\vec{v}^T + a_9(g) \vec{v}\cdot\vec{v}/2. \]
Since the quadratic forms have to equal each other, we obtain 
\[\frac{a_{9}(g)}{a_{5}(g)}\vec{v} O_5(g)^{-1} \cdot a_4(g) = \vec{v} \cdot a_8(g) \hbox{ for all } \vec{v} \in \bR^{i_{0}}.\] 
Thus, $\frac{a_{9}(g)}{a_{5}(g)}(O_5(g)^T a_4(g))^T = a_8(g)^T$.

Since we have $\mu_g = 1$, we obtain $a_1(g) = a_9(g) = a_5(g) = \lambda_{\bv_{\tilde E}}(g)$ and
$A_5(g) = \lambda_{\bv_{\tilde E}}(g) O_5(g)$ by Lemma \ref{lem-similarity}. 
Also, $a_1(g) = a_9(g) = a_5(g) = \lambda_{\bv_{\tilde E}}(g)$. 
\end{proof}

Thus, 
we conclude that each $g \in \bGamma_{\tilde E} $ has the form 
\begin{equation} \label{eqn-formgi}
\newcommand*{\temp}{\multicolumn{1}{r|}{}}
\left( \begin{array}{ccccccc} 
S(g) & \temp & 0 & \temp & 0 & \temp & 0 \\ 
 \cline{1-7}
0 &\temp & a_{1}(g) &\temp & 0 &\temp & 0 \\ 
 \cline{1-7}
C_1(g) &\temp & a_{1}(g) \vec{v}^T_g &\temp & a_{5}(g) O_5(g) &\temp & 0 \\ 
 \cline{1-7}
c_2(g) &\temp & a_7(g) &\temp &  a_{5}(g) \vec{v}_g O_5(g) &\temp & a_{9}(g)
\end{array} 
\right).
\end{equation}

Thus,  when $\mu_{g}=1$ for all $g \in \bGamma_{\tilde E}$, by 
taking a finite index subgroup of $\bGamma_{\tilde E}$, 
we conclude that each $g \in \bGamma_{\tilde E} $ has the form 
\begin{equation} \label{eqn-formgii}
\newcommand*{\temp}{\multicolumn{1}{r|}{}}
\left( \begin{array}{ccccccc} 
S(g) & \temp & 0 & \temp & 0 & \temp & 0 \\ 
 \cline{1-7}
0 &\temp & \lambda_{\bv_{\tilde E}}(g) &\temp & 0 &\temp & 0 \\ 
 \cline{1-7}
C_1(g) &\temp & \lambda_{\bv_{\tilde E}}(g)\vec{v}^T_g &\temp & \lambda_{\bv_{\tilde E}}(g) O_5(g) &\temp & 0 \\ 
 \cline{1-7}
c_2(g) &\temp & a_7(g) &\temp & \lambda_{\bv_{\tilde E}}(g) \vec{v}_g O_5(g) &\temp & \lambda_{\bv_{\tilde E}}(g)
\end{array} 
\right).
\end{equation}



\begin{corollary}\label{cor-formg2} 
If $g$ of form of equation \eqref{eqn-formgi} centralizes a Zariski dense subset $A'$ of $\CN$, 
then $\mu_{g}=1$ and $O_5(g) = \Idd_{i_0}$. 
\end{corollary} 
\begin{proof} 
$\CN$ is isomorphic to $\bR^{i_{0}}$.
The subset $A''$ of $\bR^{i_{0}}$ corresponding to $A'$ is also Zariski dense in $\bR^{i_{0}}$. 
$g \CN(\vec{v}) = \CN(\vec{v}) g$ shows that 
$\vec{v} = \vec{v} O_5(g)$ for all $\vec{v} \in A''$. 
Hence $O_5(g) =\Idd$. 
\end{proof}

\subsubsection{Splitting the NPCC end} \label{subsub-split}

\begin{proposition}[Splitting] \label{prop-decomposition}
Assume Hypothesis \ref{h-norm}. 
Suppose additionally the following{\rm :} 
\begin{itemize} 
\item Suppose that $a_{1}(g) \geq a_{5}(g), a_{9(g)}$ whenever $a_{1}(g)$ is the largest eigenvalue 
of the semisimple part  $\hat S(g)$ of $g$. 
\item $K = \{k\} \ast K''$ a strict join, and $K^{o}/N_{K}$ is compact. 
\item A center of $\bGamma_{\tilde E}$ maps to $N_{K}$
going to a Zariski dense group of the virtual center of $\Aut(K)$.  
\end{itemize} 
Then $K''$ embeds projectively in the closure of $\Bd \torb$
invariant under $\bGamma_{\tilde E}$, and 
one can find a coordinate system so that for every $\CN(\vec v)$ and each element $g$ of 
$\bGamma_{\tilde E}$ is written so that
\begin{itemize}
\item $C_1(\vec v)=0, c_2({\vec v})=0$, and 
\item $C_1(g)=0$ and $c_2(g) = 0$.
\end{itemize}
\end{proposition}
\begin{proof}
Let $\bGamma'_{\tilde E}$ denote the finite index subgroup of $\bGamma_{\tilde E}$ centralizing $\CN$
and a product of cyclic and hyperbolic groups. 



The cone $K$ is foliated by open lines from a point $k \in K$ to points of $K''$.
Call these {\em $k$-radial lines}. 
Take such a line $l$ and a sequence of points 
$\{k_m\}$ in $K^{o}$ so that 
\[k_m \ra  k_\infty \in K^{\prime \prime o} \hbox{ as } m \ra \infty.\] 
By the last condition, 
$\bGamma'_{\tilde E}$ contains a sequence $\{\gamma_m\}$ in the virtual center
so that 
\begin{itemize}
\item $\gamma_m(k_m) \ra x_{0}\in K^{o}$, 
\item $\gamma_m(\partial_1 l) \ra k_\infty \in K^{\prime \prime o}$ for the endpoint $\partial_1 l$ of $l$ in $K''$. 
\end{itemize} 

Since $K''$ is properly convex, $\{\gamma_m| K''\}$ is a bounded sequence of 
transformations and hence $\gamma_m$ is of form: 
\renewcommand{\arraystretch}{1.2}
\begin{equation}\label{eqn-gamman}
\newcommand*{\temp}{\multicolumn{1}{r|}{}}
\left( \begin{array}{ccccccc} 
\delta_m O_m & \temp & 0 & \temp & 0 & \temp & 0 \\ 
 \cline{1-7}
0 &\temp & a_{1}(g) &\temp & 0 &\temp & 0 \\ 
 \cline{1-7}
C_1(g) &\temp & a_{1}(g) \vec{v}^T_g &\temp & a_{5}(g) O_5(g) &\temp & 0 \\ 
 \cline{1-7}
c_2(g) &\temp & a_7(g) &\temp &  a_{5}(g) \vec{v}_g O_5(g) &\temp & a_{9}(g)
\end{array} 
\right)
\end{equation}
where $\{O_m\}$ is a bounded sequence of matrices in 
\[\Aut(K'') \subset \SL_{\pm}(n-i_0-1, \bR)\]
since the set of projective automorphisms of 
$K''$ moving interior points uniformly bounded distances is bounded. 


We choose $m$ so that the norms of eigenvalues of $\delta_m O_m$ are strictly much smaller 
than the norm of $\lambda_m$, the unique norm of the eigenvalues of the lower-right $(i_0+2)\times(i_0+2)$-matrix. 
We fix one such $m_0$. 
Let $S(K''_{m_0})$ denote the $\gamma_{m_0}$-invariant subspace corresponding to subspaces 
associated with the real sum of the real Jordan-block subspaces with norms of eigenvalues $< \lambda_{m_0}$. 
We choose a coordinate system of $\SI^n$ so that $\gamma_{m_0}$ is of form 
so that $C_{1, m_0} =0, c_{2, m_0}=0$. 
Then a compact proper convex domain $K''_{m_0}$ in $S(K''_{m_0})$ maps to $K''$ under 
under the projection $\Pi_K: \SI^n - \SI^{i_0}_\infty \ra \SI^{n-i_0-1}$. 

Since every element $g$ of $\bGamma'_{\tilde E}$ commutes with $\gamma_{m}$, 
$g(S(K''_{m_{0}})) = S(K''_{m_{0}})$ by considering the Jordan blocks associated with 
eigenvalues $< \lambda_{m}$. Since $K''_{m_{0}}$ is the unique space mapping 
to $K''$, we obtain that $\bGamma'_{\tilde E}$ acts on $K''_{m_{0}}$. 

Since $\bGamma_{\tilde E}/\bGamma'_{\tilde E}$ is finite, 
we obtain finitely many sets of form $g(K''_{m_0})$ for $g \in \bGamma_{\tilde E}$. 
If they are not identical, at least one $g'$ satisfies
$g'(K''_{m_0}) \ne K''_{m_0}$. Then $\gamma_{m_0}^i(g'(K''_{m_0}))$ then produces 
infinitely many distinct sets of form $g(K''_{m_0})$, which is a contradiction. 
Hence $g(K''_{m_0}) = K''_{m_0}$ for all $g \in \bGamma_{\tilde E}$. 
This implies that $C_1(g)=0$ and $c_2(g) = 0$.

\end{proof}



\subsection{Joins and quasi-joined ends for $\mu\equiv 1$} \label{subsec-qjoin}

We will now discuss about joins and their generalizations in depth in this subsection. 
That is we will only consider when $\mu_{g}= 1$ for all $g\in \bGamma_{\tilde E}$.
We will use a hypothesis and later show that the hypothesis is true in our cases
to prove the main results. 


\begin{hypothesis}[$\mu_{g}\equiv1$]\label{h-qjoin} 
Let $G$ be a p-end fundamental group. 
We continue to assume as in Hypothesis \ref{h-norm} for $G$.  
\begin{itemize} 
\item Every $g \in \Gamma \ra M_g$ is 
so that $M_g$ is in a fixed compact group $O(i_0)$. Thus, $\mu_g = 1$ identically. 
\item $G$ acts on the subspace $\SI^{i_0}_\infty$ containing 
$\bv_{\tilde E}$ and the properly convex 
domain $K'''$ in the subspace $\SI^{n-i_0-2}$ 
disjoint from $\SI^{i_0}_\infty$ 
mapping homeomorphic to the factor $K'' = \{k\} \ast K $ under $\Pi_{K}$.
\item $\CN$ acts on these two subspaces fixing every points 
of $\SI^{n-i_0-2}$.
\end{itemize} 
\end{hypothesis} 



We assume $\bv_{\tilde E}$ to have coordinates $[0, \dots, 0, 1]$.
$\SI^{n-i_0-2}$ contains the standard points $[e_i]$ for $i=1, \dots, n-i_0 -1$ 
and $\SI^{i_0+1}$ contains $[e_i]$ for $i=n-i_0, \dots, n+1$. 
Let $H$ be the open $n$-hemisphere defined by $x_{n-i_0} > 0$. Then by convexity of $U$, 
we can choose $H$ so that $K'' \subset H$ and $\SI^{i_0}_\infty \subset \clo(H)$. 

By Hypothesis \ref{h-qjoin}, 
elements of $\CN$ have the form of equation \eqref{eqn-nilmatstd} with 
\[C_1(\vec{v})=0, c_2(\vec{v})=0 \hbox{ for all } \vec{v} \in \bR^{i_0}\] 
and the group $G$ of form of equation \eqref{eqn-formgii} with 
\[s_1(g) =0, s_2(g) = 0, C_1(g) = 0, \hbox{ and } c_2(g) = 0.\] 
We assume further that $O_5(g) = \Idd_{i_0}$. 


Again we recall the projection $\Pi_K: \SI^n - \SI^{i_0}_\infty \ra \SI^{n-i_0-1}$. 
$G$ has an induced action on $\SI^{n-i_0 -1}$ and
acts on a properly convex set $K''$ in $\SI^{n-i_0-1}$ so that 
$K$ equals a strict join $k* K''$ for $k$ corresponding to $\SI^{i_0+1}$. 
(Recall the projection $\SI^n - \SI^{i_0}_\infty$ to $\SI^{n-i_0 -1}$. )


\begin{figure}
\centering
\includegraphics[height=6cm]{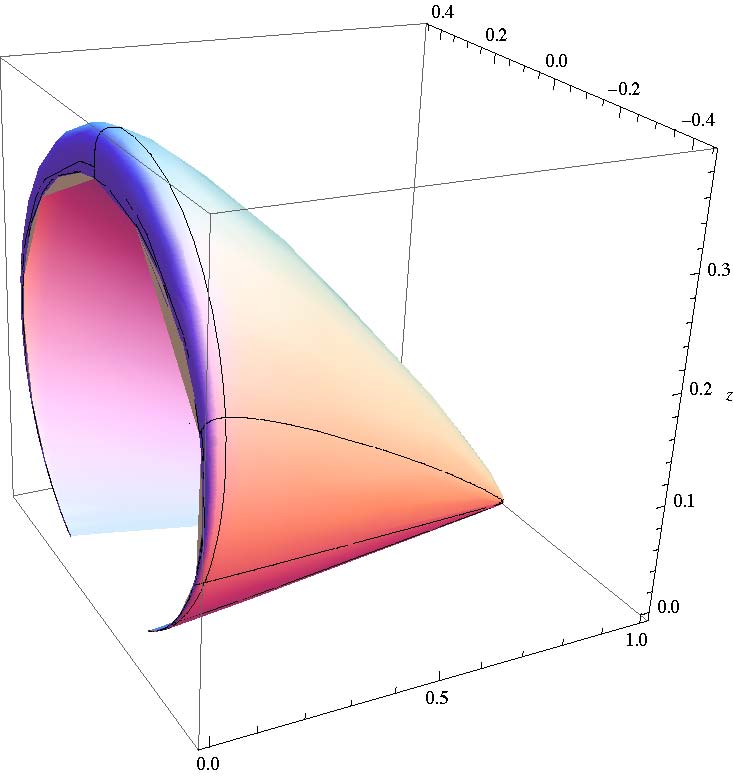}
\caption{A figure of a quasi-joined p-R-end-neighborhood}
\label{fig:quasi-j}
\end{figure}

We define invariants from the form of equation \eqref{eqn-formgii}
\[\alpha_7(g):= \frac{a_7(g)}{\lambda_{\bv_{\tilde E}}(g)} - \frac{||\vec{v}_g||^2}{2} \]
for every $g \in G$. 
\[\alpha_7(g^n) = n \alpha_7(g) \hbox{ and }
\alpha_7(gh) = \alpha_7(g) + \alpha_7(h), \hbox{ whenever }g, h, gh \in G.\] 

Here $\alpha_7(g)$ is determined by factoring  
the matrix of $g$ into commuting matrices of form
\renewcommand{\arraystretch}{1.5}
\begin{multline}\label{eqn-kernel}
\newcommand*{\temp}{\multicolumn{1}{r|}{}}
\left( \begin{array}{ccccccc}
\Idd_{n-i_0-1} & \temp & 0 & \temp & 0 & \temp & 0 \\   
 \cline{1-7}
0                    &\temp & 1 & \temp & 0 & \temp & 0  \\ 
 \cline{1-7}
0                    & \temp &   0 &\temp & \Idd_{i_0} &\temp & 0 \\ 
 \cline{1-7}
0                    &\temp &  \alpha_7(g) &\temp & \vec{0} &\temp & 1 \\ 
\end{array} 
\right) \times \\
\newcommand*{\temp}{\multicolumn{1}{r|}{}}
\left( \begin{array}{ccccccc}
S_g & \temp & 0 & \temp & 0 &\temp & 0 \\ 
\cline{1-7}
0 & \temp & \lambda_{\bv_{\tilde E}}(g) & \temp & 0 & \temp & 0  \\ 
 \cline{1-7}
0& \temp & \lambda_{\bv_{\tilde E}}(g) \vec{v}_g &\temp & \lambda_{\bv_{\tilde E}}(g) O_5(g) &\temp & 0 \\ 
 \cline{1-7}
0& \temp & \lambda_{\bv_{\tilde E}}(g)  \frac{||\vec{v}||^2}{2}  
&\temp & \lambda_{\bv_{\tilde E}}(g) \vec{v}_g O_5(g) &\temp & \lambda_{\bv_{\tilde E}}(g) \\ 
\end{array} 
\right).
\end{multline}


\begin{remark} \label{rem:alpha7}
We give a bit more explanations. 
Recall that the space of segments in a hemisphere $H^{i_0+1}$ with the vertices $\bv_{\tilde E}, \bv_{\tilde E-}$ 
forms an affine space $A^i$ one-dimension lower, and the group $\Aut(H^{i_0+1})_{\bv_{\tilde E}}$ of projective automorphism 
of the hemisphere fixing $\bv_{\tilde E}$ maps to $\Aff(A^{i_0})$ with kernel $K$ equal to 
transformations of an $(i_0+2)\times (i_0+2)$-matrix form
\renewcommand{\arraystretch}{1.5}
\begin{equation}\label{eqn-kernel2}
\newcommand*{\temp}{\multicolumn{1}{r|}{}}
\left( \begin{array}{ccccc} 
1 & \temp & 0 & \temp & 0  \\ 
 \cline{1-5}
0 &\temp & \Idd_{i_0} &\temp & 0 \\ 
 \cline{1-5}
b &\temp & \vec{0} &\temp & 1 \\ 
\end{array} 
\right)
\end{equation}
where $\bv_{\tilde E}$ is given coordinates $[0, 0, \dots, 1]$ and 
a center point of $H^{i_0+1}_l$ the coordinates $[1, 0, \dots, 0]$. 
In other words the transformations are of form 
\begin{align}\label{eqn-temp}
\left[
\begin{array}{c}
 1\\
 x_1  \\
 \vdots \\ 
 x_{i_0} \\ 
 x_{i_0+1}   
\end{array}
\right]
\mapsto 
\left[
\begin{array}{c}
 1 \\
 x_1\\
 \vdots \\ 
 x_{i_0} \\ 
 x_{i_0+1}+b
\end{array}
\right]
\end{align}
and hence $b$ determines the kernel element. 
Hence $\alpha_7(g)$ indicates the translation towards $\bv_{\tilde E}=[0,\dots, 1]$. 
\end{remark}

We assumed $\mu\equiv 1$. 
We define $\lambda_{k}(g):=\lambda_{\bv_{\tilde E}}(g)$ for $k$.
We define $\lambda_{K''}(g)$ to be the maximal norm of the eigenvalue occurring for $\hat S(g)$. 

We define $G_+$ to be a subset of $G$ consisting of elements $g$ so that 
the largest norm $\lambda_1(g)$ of the eigenvalues occurs at the vertex $k$, i.e., $\lambda_{1}(g) = \lambda_{k}(g)$. 
Then since $\mu_g=1$, we necessarily have $\lambda_1(g) = \lambda_{\bv_{\tilde E}}(g)$
with all other norms of the eigenvalues occurring at $K''$ is strictly less than $\lambda_{\bv_{\tilde E}}(g)$.  
The second largest norm $\lambda_2(g)$ equals $\lambda_{K''}(g)$. 
Thus, $G_+$ is a semigroup.     
The condition that $\alpha_7(g) \geq 0$ for $g \in G_+$ is said to be the 
{\em nonnegative translation condition}.


Again, we define \[\mu_7(g) : = \frac{\alpha_7(g)}{\log\frac{\lambda_{\bv_{\tilde E}}(g)}{\lambda_2(g)}}\] where 
$\lambda_2(g)$ denote the second largest norm of the  eigenvalues of $g$ and 
we restrict $g \in G_+$. 
The condition 
\begin{equation}\label{eqn:uptc}
\mu_7(g) > C_0, g \in  G_+ \hbox{ for a uniform constant } C_0 
\end{equation}
is called the {\em uniform positive translation condition}. 
(Heuristically, the condition means that 
we don't translate in the negative direction by too much
for bounded $\frac{\lambda_{\bv_{\tilde E}}(g)}{\lambda_2(g)}$.)

Suppose that $G$ is a p-end fundamental group.



For this proposition, we do not assume $N_K$ is discrete. 
The assumptions below are just Hypotheses \ref{h-norm} and \ref{h-qjoin}. 
We fully state for a change. 

\begin{proposition}[Quasi-joins] \label{prop-qjoin}
Let $\Sigma_{\tilde E}$ be the end orbifold of an NPCC R-end $\tilde E$ of a strongly tame 
properly convex $n$-orbifold $\orb$. 
Let $G$ be the p-end fundamental group. 
Let $\tilde E$ be an NPCC p-R-end
and $G$ and $\mathcal N$ acts on a p-end-neighborhood 
$U$ fixing $\bv_{\tilde E}$. Let $K, K'', \SI^{i_0}_\infty,$ and $\SI^{i_0+1}$ be as above.
We assume that $K^o/G$ is compact, $K= K''* k$ in $\SI^{n-i_0}$ with $k$ 
corresponding to $\SI^{i_0+1}$ under the projection $\Pi_K$. 
Assume that 
\begin{itemize} 
\item $G$ satisfies the weak middle-eigenvalue condition. 
\item $\mu_{g} = 1$ for all $g \in G$. 
\item Elements of $G$ and $\CN$ are of form of equations \eqref{eqn-form1} and \eqref{eqn-form2}.
with \[C_1(\vec{v}) = 0, c_2(\vec{v})=0, C_1(g) = 0, c_2(g) =0\]
for every $\vec{v} \in \bR^{i_0}$ and $g \in G$.
\item $G$ normalizes $\CN$, and $\CN$ acts on $U$ 
and each leaf of $\mathcal{F}_{\tilde E}$ of $\tilde \Sigma_{\tilde E}$. 
\end{itemize} 
Then
\begin{itemize} 
\item[(i)] The condition $\alpha_7 \geq 0$ is a necessary condition 
that $G$ acts on a properly convex domain in $H$.
\item[(ii)] The uniform positive translation condition is equivalent to the existence of
a properly convex p-end-neighborhood $U'$ whose closure meets $\SI^{i_0+1}_k$ at $\bv_{\tilde E}$ only.
\item[(iii)] $\alpha_7$ is identically zero if and only if $U$ is a join and $U$ is properly convex. 
\end{itemize}
\end{proposition}
\begin{proof} 
Let $H$ be a hemisphere containing $U$ where $\partial H$ contains $\SI^{i_{0}}_{\infty}$.
$A^{n} := H^{o}$ is an affine space. 
Let $H_{l}$ denote the hemisphere with boundary $\SI^{i_{0}}_{\infty}$ and corresponding to a leaf $l$
of the foliation on $\tilde \Sigma_{\tilde E}$.  
Let $\widetilde{\mathcal{F}}$ denote the leaf space. 
We first projectively identify 
\[\bigcup_{l \, \in \, \widetilde{\mathcal{F}}} H_{l}^{o} = K^{o}\times \bR^{i_{0}+1} \subset \bR^{n}\] 
for a product of a bounded convex set in an affine space 
equivalent to $K^o$ multiplied by a complete affine space of dimension $i_0+1$ 
in an affine space given by $H^o$. 
Each of $E_l := H_l \cap U$ is given by \[x_{n+1} > x_{n-i_0+1}^2 + \cdots + x_{n}^2 + C_l \] since $\CN$ acts on each
where $C_l$ is a constant depending on $l$ and $U$. (See Section \ref{subsub-quadric}.)


Let $\Pi_{i_0}: U \ra \bR^{i_0+1}$ be the projection to the last 
$i_0+1$ coordinates $x_{n-i_0+1}, \dots, x_{n+1}$. 
We obtain a commutative diagram and an induced $L_g$ 
\begin{alignat}{3} \label{eqn-affine}
H_l  \, & \, \stackrel{g}{\longrightarrow} & \, \, g(H_l) \nonumber \\
\, \Pi_{i_0}\downarrow \, &  & \, \Pi_{i_0} \downarrow\, \, \nonumber\\ 
\bR^{i_0} \, & \, \stackrel{L_g}{\longrightarrow}  & \, \, \bR^{i_0}. 
\end{alignat}
By Equation \eqref{eqn-kernel}, 
$L_g$ preserves the quadric above in
the form of the projection up to translations in the $x_{n+1}$-axis direction.

Suppose that $G$ acts with a uniform positive translation condition.
Given a point $x = [\vec v] \in U'\subset \SI^n$ where 
$\vec v = \vec v_s + \vec v_h$ 
where $\vec v_s$ is in the direction of $K''$ and 
$\vec v_h$ is in one of $H^{i_0+1}$. 
If $g \in G_+$, then  we obtain
\begin{equation} \label{eqn-gplus}
g[\vec v] = [g \vec v_s + g \vec v_h] \hbox{ where } [g \vec v_s] \in K''
\hbox{ and } [g \vec v_h] \in H_k.
\end{equation} 
 by equation \eqref{eqn-kernel}.

(i) Suppose that $\alpha_7(g) < 0$ for some $g \in G_+$.
Let $k'\in K^o$.  
Then the action by $g$ gives us that $\{g^n(E_{k'})\}$ 
converges geometrically to an $(i_0+1)$-dimensional hemisphere
since $\alpha_7(g^n) \ra -\infty$ as $n \ra \infty$ implies that 
$g$ translates the affine space $H^o_{k'}$ a component  to $H^o_{g^n(k')}$
toward $[-1,0,\dots, 0]$ in the above coordinate system by equation \eqref{eqn-kernel}.
Thus, $G$ cannot act on a properly convex domain. 


(ii) 
Let $x \in U$. 
By assumption, $\bGamma_{\tilde E}$ acts on $K = K'' \ast \{k\}$. 
Choose an element $\eta \in G_+$ so that $\lambda_1(\eta) > \lambda_2(\eta)$
where $\lambda_{1}(\eta)$ correspond to a vertex $k$ and $\lambda_{2}(\eta)$ is associated with $K''$, 
and let $F$ be the fundamental domain in $K^o$ with respect to $\langle \eta \rangle$. 
This corresponds to 
a radial subset $F$ from $\bv_{\tilde E}$ 
bounded away at a distance from $K''$ in $U$.

Choose $x_{0} \in F$. 
Let $G_F := \{g \in G| g(x_0) \in F\}$. For $g \in G_F$, 
\[ \left|\log \frac{\lambda_{\bv_{\tilde E}}(g)}{\lambda_{K''}(g)}\right| < C_{F}\] where $C_F > 0$ is a number depending of $F$ only.  
Given $g \in G_{F}$, 
we can find a number $i_{0}$ independent of $g$ such that $\eta^{i_{0}} g \in G_{+}$. 
Then $\alpha_{7}(\eta^{i_{0}} g)$ is bounded below by some negative number. 
Since $\alpha_{7}(\eta^{i_{0}} g) = i_{0}\alpha_{7}(\eta) + \alpha_{7}(g)$, 
we obtain 
\begin{equation}\label{eqn-a7}
\{\alpha_7(g)| g \in G_{F}\} > C >0
\end{equation} 
 for a constant $C$
 by the uniform positive translation condition. 
In the above affine coordinates for $k'\in F$ of equation \eqref{eqn-affine}, 
\[x_{n+1}(H^{o}_{k'}\cap U) > C \]
for a uniform constant $C \in \bR$
by equations \eqref{eqn-a7} and  \eqref{eqn-kernel}
since $F$ is covered by $\bigcup_{g\in G_{F}} g(J)$ for a compact fundamental domain $J$ of 
$K^{o}$ by $N_{K}$. 

Let $D_{F}$ be the convex hull of $\bigcup_{g' \in G_F} g(H^o_{k'} \cap U)$. 
Since by above \[\bigcup_{g' \in G_F} g(H^o_{k'} \cap U)\] 
is a lower-$x_{n+1}$-bounded set, 
$D_{F}$ as a lower-$x_{n+1}$-bounded subset 
of $K\times \bR^{i_{0}+1} \subset \bR^{n}$. 
Therefore, the convex hull $D_F$ in $\clo(\torb)$ is a properly convex set. 


Note that $K'' \ast\{ \bv_{\tilde E}\} - \{\bv_{\tilde E}\}$ identifies with
\[K'' \times [0, \infty) \subset K \times \bR\] 
in the above identification. 
Since \[\alpha_7(\eta^i) = i \alpha_7(\eta) \ra +\infty \hbox{ as } i \ra \infty,\] 
we obtain that
\[\{\eta^i(D_F)\} \ra \{ \bv_{\tilde E}\} \hbox{ for } i \ra \infty\]
geometrically, i.e., under the Hausdorff metric $\bdd_{H}$. 
Also, we can show by equation \eqref{eqn-gplus} that 
\[\{\eta^i(D_F)\} \ra K'' * \bv_{\tilde E}  \hbox{ for } i \ra -\infty\]
geometrically. 
Thus, using the above coordinates, 
the convex hull of 
\[\bigcup_{i\in \bZ} \eta^{i}(D_{F}) \subset K^{o}\times \bR^{i_{0}+1} \]
 is properly convex also since they are uniformly bounded from below in the $x_{n+1}$-coordinates.
(See Theorem \ref{II-prop-quasilens1}  in \cite{EDC2} also where we used a slightly different proof for
a similar result.)


Let $U'$ be a p-end-neighborhood of $\bv_{\tilde E}$ that is the interior of 
the convex hull of $\{g_i(D_F)\}$. By the boundedness from $\bv_{\tilde E}$ of at most distance $\pi - C$ for some $C> 0$, 
the convex hull is properly convex. 
The fact that each $H_{l} \cap U'$ is a horoball implies that 
$\clo(U') \cap \SI^{i_0+1}_k =\{ \bv_{\tilde E}\}$ holds. 



Conversely, suppose that $G$ acts on a properly convex p-end-neighborhood $U'$.

Suppose that $\alpha_7(g) =0$ for some $g\in G_+$. 
Then \[g^{i}(\clo(U) \cap H_l) \ra B \hbox{ as }  i \ra \infty \hbox{ under } \bdd_{H}\]
for a leaf $l$ and a compact domain $B$ at $H_k$ bounded by an ellipsoid.
This contradicts Lemma \ref{lem-noball}.
Therefore, $\mu_7(h) > 0$ for every $h \in G_+$ by (i). ---(*)

%

Suppose that $\mu_7(g_i) \ra 0$ for a sequence $g_i \in G_+$. 
We can assume that $ \lambda_1(g_i)/\lambda_{2}(g_i) > 1+\eps$ for a positive constant $\eps > 0$ since we can take powers 
of $g_i$ not changing $\mu_7$. 

Since $\mu_{7}(g_{i}) \ra 0$, 
we obtain a nondecreasing sequence $n_i$, $n_i > 0$, so that 
\[\alpha_7(g_i^{n_i})= n_i\alpha_7(g_i) \ra 0 \hbox{ and }
 \lambda_1(g_i^{n_i})/\lambda_{2}(g_i^{n_i}) \ra \infty.\]   
 However, from such a sequence, we use  
 equation \eqref{eqn-kernel} to shows that 
 \[\{g_i^{n_i}(\clo(U) \cap H_{l})\} \ra B\]
 to a ball $B$ with nonempty interior in $H_{k}$.
By Lemma \ref{lem-noball}, this is a contradiction. 
 Hence $\mu_7(g) > C$ for all $g \in G_+$ and a uniform constant 
 $C > 0$. 
 This proves the converse part of (ii). 


(i) and (*) in the proof (ii) proves (iii). 


\end{proof} 

\begin{definition}\label{defn-qjoin}
\begin{itemize}
\item In case (iii) of Proposition \ref{prop-qjoin}, 
$\tilde E$ is said to be a {\em joined p-R-end} ( of {\em a totally geodesic R-end and a horospherical end})
and $G$ now is called a {\em joined end group}
\item In case (ii) of Proposition \ref{prop-qjoin}, 
$\tilde E$ is said to be a {\em quasi-joined p-R-end} ( of {\em a totally geodesic R-end and a horospherical end})
and $G$ now is called a {\em quasi-joined end group}.
An end with an end-neighborhood that is covered by a p-end-neighborhood of 
such a p-R-ends is also called a {\em quasi-joined p-R-end}. 
\end{itemize}
\end{definition}


From the matrix equation \eqref{eqn-kernel}, 
we define $v_{g}$ for every $g \in \bGamma_{\tilde E}$. 
(We just need to do this under  a single coordinate system. )

\begin{lemma} \label{lem-vbound}
Given $G$ satisfying Hypotheses \ref{h-norm} and \ref{h-qjoin}, 
let $\gamma_{m}$ be any sequence of elements of $G_{+}$
so that $\lambda_{k}(\gamma_{m})/\lambda_{K''}(\gamma_{m}) \ra \infty$. 
Then we can replace it by another sequence $\gamma'_{m}$ so that 
\[||\vec{v}_{\gamma'_{m}}|| \hbox{ and } \Pi^{\ast}_{K}(\gamma'_{m} \gamma_{m}^{-1}) \in \Aut(K)\] 
are uniformly bounded. 
\end{lemma} 
\begin{proof}
Given $g\in \bGamma_{\tilde E}$, let $\Pi_{K}^{\ast}(g): K^{o}\ra K^{o}$ denote the 
induced projective automorphism of $g$ on $K^{o}$. 

Suppose that $N_{K}$ is discrete. Then $\bGamma_{\tilde E}\cap \mathcal{N}$ is a lattice 
in $\mathcal{N}$. By cocompactness of $\bGamma_{\tilde E}\cap \mathcal{N}$
in $\mathcal{N}$, we can multiply $\gamma_{m}$ by $h_{m}^{-1}$ for 
an element $h_{m}$ of $\bGamma_{\tilde E}\cap \mathcal{N}$
nearest to $\mathcal{N}(\vec{v}_{m})$.  The result follows. 

We assume that $N_K$ is indiscrete. 
$\tilde \Sigma_{\tilde E}$ has a compact fundamental domain $F$ under $\bGamma_{\tilde E}$. 
Thus, given any $\vec{v}$,  for $x \in F$, 
\[\CN(\vec{v})(x) \in g(F) \hbox{ for some } g \in \bGamma_{\tilde E}.\] 
Then $g^{-1}\CN(\vec{v})(x) \in F$. Since 
\[g(y) = \CN({\vec{v}})(x) \in g(F)  \hbox{ for } y \in F \hbox{ and } x \in F,\] 
it follows that
\begin{equation}\label{eqn-dK}
d_K\bigg(\Pi_K(y), \Pi^{\ast}_K(g)(\Pi_{K}(y))= \Pi_{K}(x)\bigg) < C_{F}
\end{equation}
for a constant $C_F$ depending on $F$. 
\begin{itemize}
\item [(i)] $g$ is of form of matrix of equation \eqref{eqn-gform}. 
\item[(ii)] $S_{g}$ is in a bounded neighbourhood of $\Idd$ by above equation \ref{eqn-dK}
since the bounded Hilbert $d_{K}$-length of $g$ implies the boundedness of the action on $K^{o}$. 
\item[(iii)] $g$ is in a bounded neighborhood of $\CN$ by (ii)
since $g$ is of form of matrix of equation \eqref{eqn-gform}. 
\end{itemize} 
From the linear block form of $g^{-1}\CN(\vec{v})$ and the fact that  $g^{-1}\CN(\vec{v})(x) \in F$, 
we obtain that the corresponding $\vec{v}_{g^{-1}\CN(\vec{v})}$ can be made uniformly bounded independent of $\vec{v}$. 

For element $\gamma_m$ above, we take its vector $\vec{v}_{\gamma_m}$ and find our $g_{m}$ for $\mathcal{N}(\vec{v}_{\gamma_m})$.
We obtain $\gamma'_{m}:= g_{m}^{-1}\gamma_m$. Then the corresponding 
$\vec{v}_{g_{m}^{-1}\gamma_m}$ is uniformly bounded as
we can see from the block multiplications. 

Since $g_{m}$ satisfies equation \eqref{eqn-dK}, norms of eigenvalues of $\Pi_{K}^{\ast}(g_{m})$ are uniformly bounded.  
\end{proof} 

\begin{lemma}\label{lem-noball} 
Suppose that the holonomy group of $\orb$ is strongly irreducible. 
Given $G$ satisfying Hypotheses \ref{h-norm} and \ref{h-qjoin}, 
let $U$ be the properly convex p-end neighborhood of $\bv_{\tilde E}$.
$\clo(U) \cap H_{k}$ cannot contain an open domain $B$ with $\Bd B \ni \bv_{\tilde E}$.
\end{lemma} 
\begin{proof} 
First of all, 
\begin{equation}\label{eqn-alpha7} 
\alpha_7(h) =0 \hbox{ for all } h \in G 
\end{equation}
by (i) since otherwise by equation \eqref{eqn-kernel}
\[h^i(B) \ra H_{k} \hbox{ as } i \ra \pm \infty \hbox{ for } h \hbox{ with } \alpha_7(h) \ne 0.\] 
Since $\tilde \Sigma_{\tilde E}/\bGamma_{\tilde E}$ is compact,
we have a sequence $h_i \in G_+$ where 
\[\frac{\lambda_{\bv_{\tilde E}}(h_i)}{\lambda_2(h_i) } \ra \infty, \alpha_{7}(h_{i})=0,
 \hbox{ and $h_i| K''$ is uniformly bounded.} \]
 Now modify $h_{i}$ by Lemma \ref{lem-vbound}.
 
 Recall that $K$ is a strict join $K'' \ast \{k\}$ for a properly convex domain $K''\subset \Bd \torb$ of dimension 
$n-i_{0}-2$ and a vertex $k$.
Denote by $S(K'')$ and $S(H)$ the subspaces spanned by $K'$ and $H_{k}$.
$S(K'')$ and $S(H_{k})$ form a pair of complementary subspaces in $\SI^n$. 

From the form of the lower-right $(i_0+2)\times (i_0+2)$-matrix of the above matrix, 
$h_{i}$ must act on the horosphere $H \subset S(H_{k})$. $\CN$ also act transitively on $H_{k}$. 
Hence, for any such matrix we can find an element of $\CN$ so that 
the product is in the orthogonal group acting on $H_{k}$. 

 Now, this is the final part of the proof: 
Let $H_{\mx}$ denote $S(H_{k}) \cap \clo(\torb)$ and $K^{\prime \prime}_{\mx}$ the set $S(K'') \cap \clo(\torb)$.
Since $\{\vec{v}_{\gamma_m}\}$ is bounded and 
$\alpha_7(\gamma_m)= 0$, 
we have the sequence $\{\gamma_m\}$
\begin{itemize}
\item acting on $K''_{\mx}$ is uniformly bounded and 
\item $\gamma_m$ acting on $H_{\mx}$ in a uniformly bounded manner 
as $m \ra \infty$. 
\end{itemize}
By Proposition \ref{II-prop-decjoin} of \cite{EDC2} for $l = 2$ case, 
$\clo(\torb)$ equals the join of $H_{\mx}$ and $K'_{\mx}$. 
This implies that $\bGamma$ is virtually reducible. Hence the joined ends cannot occur.

%
\end{proof}  


\subsection{The non-existence of split joined cases for $\mu\equiv 1$.} 
\label{sub-quasijoin}

\begin{theorem}\label{thm-NPCCcase2} 
Let $\Sigma_{\tilde E}$ be the end orbifold of an NPCC p-R-end $\tilde E$ of a strongly tame  
properly convex $n$-orbifold $\orb$ with radial or totally geodesic ends. 
Assume that the holonomy group of $\orb$ is strongly irreducible.
Let $\bGamma_{\tilde E}$ be the p-end fundamental group.
Assume Hypotheses \ref{h-norm} only and $\mu_{g} = 1$ for all $g \in \bGamma_{\tilde E}$. 
Then $\tilde E$ is not a joined end. 
\end{theorem}
\begin{proof}
Suppose that $\tilde E$ is a joined end.
By premise, $\mu_g = 1$ for all $g \in \bGamma_{\tilde E}$. 
By Lemma \ref{lem-conedecomp1} and Proposition \ref{prop-decomposition}, 
every $g \in \bGamma_{\tilde E}$ is of form: 
\renewcommand{\arraystretch}{1.5}
\begin{equation}\label{eqn-gform}
\newcommand*{\temp}{\multicolumn{1}{r|}{}}
\left( \begin{array}{ccccccc} 
S_{g} & \temp & 0 & \temp & 0 & \temp & 0 \\ 
 \cline{1-7}
0 &\temp & \lambda_g &\temp & 0 &\temp & 0 \\ 
 \cline{1-7}
 0 &\temp & \lambda_g \vec{v}^T_g &\temp & \lambda_g O_5(g) &\temp & 0 \\ 
 \cline{1-7}
0 &\temp & \lambda_m\left(\alpha_7(g_m) + \frac{||\vec{v}_g||^2}{2}\right) &\temp &
\lambda_g \vec{v}_g &\temp & \lambda_g 
\end{array} 
\right)
\end{equation}

As in the proof of Proposition \ref{prop-decomposition}, 
we obtain a sequence 
 $\gamma_m$ of form: 
\renewcommand{\arraystretch}{1.5}
\begin{equation}\label{eqn-gammao}
\newcommand*{\temp}{\multicolumn{1}{r|}{}}
\left( \begin{array}{ccccccc} 
\delta_m O_m & \temp & 0 & \temp & 0 & \temp & 0 \\ 
 \cline{1-7}
0 &\temp & \lambda_m &\temp & 0 &\temp & 0 \\ 
 \cline{1-7}
 0 &\temp & \lambda_m \vec{v}^T_m &\temp & \lambda_m O_5(\gamma_m) &\temp & 0 \\ 
 \cline{1-7}
0 &\temp & \lambda_m\left(\alpha_7(\gamma_m) + \frac{||\vec{v}_m||^2}{2}\right) &\temp &
\lambda_m \vec{v}_m &\temp & \lambda_m 
\end{array} 
\right)
\end{equation}
as  $C_{1, m} =0$ and $c_{2, m}=0$
where 
\begin{itemize}
\item $\lambda_m \ra \infty$, 
\item $\delta_m \ra 0$ and $O_m$ is in a set of bounded matrices in $\SL_{\pm}(n-i_0-1)$, 
\item $\mu_7(\gamma_m ) = 0$ by Proposition \ref{prop-qjoin} (iii).
\end{itemize}
This implies $\alpha_7(\gamma_m) = 0$ also by definition. 
Moreover, Hypothesis \ref{h-qjoin} now holds. 
By Lemma \ref{lem-noball}, we obtain a contradiction.

\end{proof}

\subsection{The proof for discrete $N_{k}$.} \label{sub-discretecase}

Now, we go to proving Theorem \ref{thm-thirdmain} when $N_{K}$ is discrete. 
By taking a finite index subgroup if necessary, we may assume that $N_{K}$ acts freely on $K^{o}$. 
We have a corresponding fibration 
\begin{eqnarray}
l/N & \ra & \tilde \Sigma_{\tilde E}/\bGamma_{\tilde E} \nonumber \\ 
     &       & \, \, \downarrow \nonumber \\ 
   &         & K^o/N_K 
\end{eqnarray}
where the fiber and the quotients are compact orbifolds
since $\Sigma_{\tilde E}$ is compact. 
Here the fiber equals $l/N$ for generic $l$. 

Since $N$ acts on each leaf $l$ of ${\mathcal F}_{\tilde E}$ in $\tilde \Sigma_{\tilde E}$, 
it also acts on a properly convex domain $\torb$ and $\bv_{\tilde E}$ in a subspace $\SI^{i_0+1}_l$ in $\SI^n$ 
corresponding to $l$. $l/N \times \bR$ is an open real projective orbifold diffeomorphic to 
$(H^{i_0+1}_l \cap \torb)/N$
for an open hemisphere $H^{i_0+1}_l$ corresponding to $l$. 
Since elements of $N$ restricts to $\Idd$ on $K$, $\lambda_{1}(g) = \lambda_{n+1}(g)$:
Otherwise, we see easily $g$ acts not trivially on $\SI^{n-i_{0}-1}$. 
By Proposition \ref{prop-eigSI}, the all norms of eigenvalues are $1$. 
Since $l$ is a complete affine space, 
Lemma \ref{I-lem-unithoro} of \cite{EDC1} shows that  
\begin{itemize} 
\item $l$ covers a horospherical end of $(\SI^{i_0+1}_l \cap \torb)/N$.
\item By Theorem \ref{I-thm-affinehoro}, $N$ is virtually unipotent and $N$ is virtually a cocompact subgroup of 
a unipotent group, conjugate to a parabolic subgroup of
$\SO(i_0+1, 1)$ in $\Aut(\SI^{i_0+1}_l)$ 
and acting on an ellipsoid of dimension $i_0$ in $H^{i_0+1}_l$. 
\end{itemize} 


Recall these from \cite{EDC1}.
By the nilpotent Lie group theory of 
Malcev, the Zariski closure $Z(N)$ of $N$ is a virtually nilpotent Lie group with finitely many components 
and  $Z(N)/N$ is compact. 
Let $\CN$ denote the identity component of the Zariski closure of $N$
so that $\CN/(\CN \cap N)$ is compact.
$\CN \cap N$ acts on the great sphere $\SI^{i_0+1}_l$ containing $\bv_{\tilde E}$ and corresponding to $l$.
Since $\CN/N$ is compact, 
we can modify $U$ so that $\CN$ acts on $U$: i.e., 
we take $\bigcap_{g\in \CN} g(U) = \bigcap_{g\in F} g(U)$
for the fundamental domain $F$ of $\CN$ by $N$. 

We remark that $N \cap \CN := \CN(L)$ for a lattice $L$ in $\bR^{i_0}$.
Since $\CN$ is the Zariski closure of $N$ and $N$ is normal in $\bGamma_{\tilde E}$, 
$\CN$ is normalized by $\bGamma_{\tilde E}$.  Thus, Hypothesis \ref{h-norm} holds. 





\begin{theorem}\label{thm-NPCCcase} 
Let $\Sigma_{\tilde E}$ be the end orbifold of an NPCC p-R-end $\tilde E$ of a strongly tame  
properly convex $n$-orbifold $\orb$ with radial or totally geodesic ends. 
Assume that the holonomy group of $\pi_{1}(\orb)$ is strongly irreducible.
Let $\bGamma_{\tilde E}$ be the p-end fundamental group,
and it satisfies the weak middle-eigenvalue condition. 
 The virtual center of $\bGamma_{\tilde E}$ goes to the the Zariski dense subgroup of 
 the virtual center of $\Aut(K)$. 
Assume also that $N_{K}$ is discrete and  $K^{o}/N_{K}$ is compact.
Then
$\tilde E$ is a quasi-join of a totally geodesic R-end and a cusp type R-end. 
\end{theorem}
\begin{proof} 
We will continue to use the notation developed above in this proof. 
By Lemma \ref{lem-similarity}, 
$h(g) \CN(\vec{v}) h(g)^{-1} = \CN( \vec{v} M_g)$ where 
$M_g$ is a scalar multiplied by an element of a copy of an orthogonal group $O(i_0)$. 

The group $\CN$ is isomorphic to $\bR^{i_{0}}$ as a Lie group. 
Since $N \subset \CN$ is a discrete cocompact, $N$ is virtually isomorphic to $\bZ^{i_0}$.
Without loss of generality, we assume that $N$ is a cocompact subgroup of $\CN$. 
$h(g)N h(g)^{-1} = N$. Since $N$ corresponds to a lattice $L \subset \bR^{i_{0}}$ by the map $\CN$, 
and the conjugation by $h(g)$ is 
to a map given by right multiplication $M_g: L \ra L$
by Lemma \ref{lem-similarity}.
Thus, $M_g: L \ra L$ is conjugate to an element of 
$\SL_\pm(i_0, \bZ)$ and 
$\{M_g| g \in \bGamma_{\tilde E}\}$ 
is a compact group as their determinant is $\pm 1$. 
Hence, 
the image of the homomorphism given by 
$g \in h(\pi_1(\tilde E)) \mapsto M_g \in \SL_\pm(i_0, \bZ))$ is a finite order group. 
Moreover, $\mu_{g}= 1$ for every $g\in \bGamma_{\tilde E}$. 
Thus, $\bGamma_{\tilde E}$ has a finite index group $\bGamma'_{\tilde E}$ centralizing $\CN$. 

We take $\Sigma_{E'}$ to be the corresponding cover of $\Sigma_{\tilde E}$. 
By Propositions  \ref{lem-conedecomp1} and \ref{prop-decomposition}, we have the result needed to apply Proposition \ref{prop-qjoin}. 
Finally, Proposition \ref{prop-qjoin}(i) and (ii) imply that $\bGamma_{\tilde E}$ virtually is either a join or a quasi-joined group.
Theorem \ref{thm-NPCCcase2} shows that a joined end cannot occur.  
\end{proof}


\section{The indiscrete case} \label{sec-indiscrete}

Let $\Sigma_{\tilde E}$ be the end orbifold of an NPCC R-end $\tilde E$ of a strongly tame 
properly convex $n$-orbifold $\orb$ with radial or totally geodesic ends. 
Let $\bGamma_{\tilde E}$ be the p-end fundamental group. 
Let $U$ be a p-end-neighborhood in $\torb$ corresponding to a p-end vertex $\bv_{\tilde E}$. 




Recall the exact sequence 
\[ 1 \ra N \ra \pi_1(\tilde E) \stackrel{\Pi^*_K}{\longrightarrow} N_K \ra 1 \] 



An element $g \in \bGamma_{\tilde E}$ is of form: 
\begin{equation} \label{eqn-g}
\newcommand*{\temp}{\multicolumn{1}{r|}{}}
g = \left( \begin{array}{ccc} 
K(g) & \temp &  0 \\ 
 \cline{1-3}
* &\temp & U(g)
\end{array} 
\right).
\end{equation}
Here $K(g)$ is an $(n-i_0)\times (n-i_0)$-matrix and 
$U(g)$ is  an $(i_0+1)\times (i_0+1)$-matrix acting on $\SI^{i_0}_\infty$. 
We note $\det K(g) \det U(g) = 1$. 


\subsection{Taking the leaf closure} 

\subsubsection{Estimations with $KA \bU$.}

Let $\bU$ denote a maximal nilpotent subgroup of $\SL_\pm(n+1, \bR)_{\SI^{i_0}_\infty, \bv_{\tilde E}}$  
given 
by lower triangular matrices with diagonal entries equal to $1$.

\begin{lemma}\label{lem-orth} 
The matrix of $g \in \Aut(\SI^n)$ can be written under a coordinate system orthogonal 
at $V^{i_0+1}_\infty$ as $k(g) a(g) n(g)$ where 
$k(g)$ is an element of $O(n+1)$, $a(g)$ is a diagonal element, and $n(g)$ is in the group  $\bU$ of unipotent 
lower triangular matrices. 
Also, diagonal elements of $a(g)$ are the norms of eigenvalues of $g$ as elements of $\Aut(\SI^n)$. 
\end{lemma} 
\begin{proof} 
%
Let $\vec{v}_1, \dots, \vec{v}_{i_0+1}, \vec{v}_{i_0+2}, \dots, \vec{v}_{n+1}$ denote the basis vectors of $\bR^{n+1}$ 
that are chosen from the real Jordan-block subspaces of $g$ with the same norms of eigenvalues
where $\vec{v}_j \in V^{i_0+1}_\infty$ for $j =1, \dots, i_0+1$. 
We require $[\vec{v}_1] = \bv_{\tilde E}$. 

Now we fix a Euclidean metric on $\bR^{n+1}$. 
We obtain vectors 
\[\vec{v}_1', \dots, \vec{v}_{i_0+1}', \vec{v}_{i_0+2}', \dots, \vec{v}_{n+1}'\]
by the Gram-Schmidt orthogonalization process using the corresponding Euclidean metric on $\bR^{n+1}$.
Then the desired result follows by writing the matrix of $g$ in terms of coordinates
given by letting the basis vectors $\vec{v}'_i = \vec{u}_{n+1 -i}$. 
(See also Proposition 2.1 of Kostant \cite{BK}.)

\end{proof}

We define 
\[ \bU' := \bigcup_{ k \in O(n+1)} k\bU k^{-1}.  \]

\begin{corollary}\label{cor-bdunip} 
Suppose that we have for a positive constant $C_1$, and $g \in \bGamma_{\tilde E}$, 
\[ \frac{1}{C_1} \leq \lambda_{n+1}(g),\lambda_1(g) \leq C_1.\]
Then $g$ is in a bounded distance from $\bU'$ with the bound depending only on $C_1$.
\end{corollary} 
\begin{proof} 
By Lemma \ref{lem-orth}, we can find an element $k\in O(n+1)$ so that 
\[ g  = k k(g) k^{-1} k a(g) k^{-1} k n(g) k^{-1}\] as above. 
Then $k k(g) k^{-1} \in O(n+1)$ and $k a(g) k^{-1}$ is uniformly bounded from $\Idd$ by 
a constant depending only on $C_1$ by Proposition \ref{prop-eigSI}. 
Finally, we obtain $k n(g) k^{-1} \in \bU'$. 
\end{proof} 

A subset of a Lie group is of {\em polynomial growth} if the volume of the ball $B_R(\Idd)$ radius $R$ is less than or
equal to a polynomial of $R$. 
As usual, the metric is given by the standard positive definite left-invariant bilinear form that is invariant 
under the conjugations by the compact group $O(n+1)$. 

\begin{lemma} \label{lem-bU} 
$\bU'$ is of polynomial growth in terms of the distance from $\Idd$. 
\end{lemma}
\begin{proof}
Let $\Aut(\SI^n)$ have a left-invariant Riemannian metric. 
Clearly $\bU$ is of polynomial growth by Gromov \cite{Gr} since $\bU$ is nilpotent.
Given $g \in O(n+1)$, the distance between $gug^{-1}$ and $u$ for $u \in \bU'$ 
is proportional to a constant multiplied by $\bdd(u, \Idd)$: 
Choose $u \in \bU'$ which is unipotent. We can write $u(s) = \exp(s \vec{u})$ where 
$\vec{u}$ is a nilpotent matrix of unit norm. $g(t) := \exp(t \vec{x})$ for $\vec{x}$ in the Lie 
algebra of $O(n+1)$ of unit norm. 
For a family of $g(t) \in O(n+1)$, we define 
\begin{equation}\label{eqn-conj}
u(t, s) = g(t) u(s) g(t)^{-1} =  \exp(s Ad_{g(t)} \vec{u}).
\end{equation}
We compute 
\[ u(t, s)^{-1} \frac{d u(t, s)}{d t} :=u(t, s)^{-1} (\vec{x} u(t, s) - u(t, s) \vec{x}) = (Ad_{u(t, s)^{-1}} - \Idd)( \vec{x} ).\]
Since $\vec{u}$ is nilpotent, $Ad_{u(t, s)^{-1}} - \Idd$ is a polynomial of variables $t, s$. 
The norm of $d u(t, s)/ dt$ is bounded above by a polynomial in $s$ and $t$. 
The conjugation orbits of $O(n+1)$ in $\Aut(\SI^n)$ are compact.
Also, the conjugation by $O(n+1)$ preserves the distances of elements from $\Idd$
since the left-invariant metric $\mu$ is preserved by conjugation at $\Idd$  
and geodesics from $\Idd$ go to geodesics from $\Idd$ of same $\mu$-lengths under 
the conjugations by equation \eqref{eqn-conj}.
Hence, we obtain a parametrization of $\bU'$ by $\bU$ and $O(n+1)$ where 
the volume of each orbit of $O(n+1)$ grows polynomially. 
Since $\bU$ is of polynomial growth, $\bU'$ is of polynomial growth
in terms of the distance from $\Idd$. 
\end{proof}


%


\begin{lemma} \label{lem-polynomial}
Each leaf $l$ is of polynomial growth. That is, each ball $B_R(x)$ in $l$ of radius $R$ has an area 
less than equal to $f(R)$ for a polynomial $f$ where we are using an arbitrary Riemannian metric on 
$F\tilde \Sigma_{\tilde E}$ induced from one on $F\Sigma_{\tilde E}$. 
\end{lemma}
\begin{proof} 
Let us choose a fundamental domain $F$ of $F\Sigma_{\tilde E}$. 
Then  for each leaf $l$ 
there exists an index set $I_{l}$ so that 
$l$ is a union of $g_i(D_i)$  $i \in I_l$ for the intersection $D_i$ of a leaf with $F$
and $g_i \in \bGamma_{\tilde E}$.  
We have that $D_i \subset D'_i$ where $D'_i$ is an $\eps$-neighborhood of $D_i$ in the leaf. 
Then \[\{g_i(D'_i)| i \in I_l \}\] cover $l$ in a locally finite manner. 
The subset $G(l):= \{g_i\in \Gamma| i \in I_l\}$ is a discrete subset. 

Choose an arbitrary point $d_i \in D_i$ for every $i \in I_l$. 
The set $\{g_i(d_i)| i \in I_l \}$ and $l$ is quasi-isometric: 
a map from $G(l)$ to $l$ is given by $f_1: g_i \mapsto g_i(d_i)$ 
and the multivalued map $f_2$ from $l$ to $G(l)$ given by sending each point $x \in l$ 
to one of finitely many $g_i$ such that $g_i(D'_i) \ni x$. 
Let $\bGamma_{\tilde E}$ be given the Cayley metric and $\tilde \Sigma_{\tilde E}$ a metric induced 
from $\Sigma_{\tilde E}$. 
Both maps are quasi-isometries since 
these maps are restrictions of quasi-isometries $\bGamma_{\tilde E} \ra \tilde \Sigma_{\tilde E}$
and $\tilde \Sigma_{\tilde E} \ra \bGamma_{\tilde E}$ defined in an analogous manner.  

The action of $g_i$ in $K$ is bounded since it sends some points of $\Pi_K(F)$ to ones of $\Pi_K(F)$. 
Thus, $\Pi^*_K(g_i)$ goes to a bounded subset of $\Aut(K)$. 
Hence in the form of equation \eqref{eqn-g},
\[K(g_i) = \det(K(g_i))^{1/(n-i_0)} \hat K(g_i) \hbox{ where } \hat K(g_i) \in \SL_{\pm}(n-i_0, \bR).\]
Let $\tilde \lambda_1(g_i)$ and $\tilde \lambda_n(g_i)$ denote the largest norm and the smallest norm of
eigenvalues of $\hat K(g_i)$. 
Since $\Pi^{\ast}_{K}(g_{i})$ are in a bounded set of $\Aut(K)$, 
these are bounded by two positive real numbers. 
The largest and the smallest eigenvalues of $g_i$ equal
\[\lambda_1(g) = \det(K(g_i))^{1/(n-i_0)}\tilde \lambda_1(g_i) \hbox{ and } 
\lambda_{n+1}(g) = \det(K(g_i))^{1/(n-i_0)}\tilde \lambda_n(g_i)\]
Denote by $a_j(g_i), j=1, \dots, i_0+1$, the norms of eigenvalues associated with $\SI^{i_0}_\infty$.  
Since 
\[\det(K(g_i)) a_1(g_i) \dots a_{i_0+1}(g_i) = 1,\] 
if $|\det(K(g_i))| \ra 0$ or $\infty$, then 
the equation in Proposition \ref{prop-eigSI} cannot hold. 
Therefore, we obtain 
\[1/C < |\det(K(g_i))| < C\] for a positive constant $C$. 
We deduce that  
the largest norm and the smallest norm of eigenvalues of $g_i$
\[\det(K(g_i))^{1/(n-i_0)}\tilde \lambda_1(g_i) \hbox{ and } 
\det(K(g_i))^{1/(n-i_0)}\tilde \lambda_n(g_i)\]
are bounded above and below by two positive numbers. 
Hence, $\lambda_1(g_i)$ and $\lambda_n(g_i)$ and  
the components of $a(g_i)$ are all bounded above and below by a fixed set of positive numbers. 

By Corollary \ref{cor-bdunip}, $\{g_i\}$ is of bounded distance from $\bU'$.
Let $N_c(\bU')$ be a $c$-neighborhood of $\bU'$. 
Then \[G(l)  \subset N_c(\bU').\] 

Let $d$ denote the left-invariant metric on $\Aut(\SI^n)$. 
By the discreteness of $\bGamma_{\tilde E}$, the set $G(l)$ is discrete 
and there exists a lower  
bound to \[\{d(g_i, g_j)| g_i, g_j \in G(l), i \ne j\}.\]
Also given any $g_i \in G(l)$, there exists 
an element $g_j \in G(l)$ so that 
$d(g_i, g_j) < C$ for a uniform constant $C$.
(We need to choose $g_j$ so that $g_j(F)$ is 
adjacent to $g_i(F)$.)
Let $B_R(\Idd)$ denote the ball in $\SL(n+1, \bR)$ of radius $R$ with the center $\Idd$. 
Then $B_R(\Idd) \cap N_c(\bU')$ is of polynomial growth with respect to $R$, and so is $G(l) \cap B_R(\Idd)$. 
Since the $\{g_i(D'_i)| g_i \in G(l)\}$ 
of uniformly bounded balls  cover $l$ in a locally finite manner, 
$l$ is of polynomial grow as well. 
\end{proof}

\subsubsection{Closures of leaves} \label{subsec-leaves}

Given a subgroup $G$ of an algebraic Lie group, 
the {\em syndetic hull} $S(G)$ of $G$ is a connected Lie group so that 
$S(G)/G$ is compact. (See Fried and Goldman \cite{FG} and D. Witte \cite{DW}.)

The properly convex open set $K, K \subset \SI^{n-i_0}$ has a Hilbert metric. Also the group $\Aut(K)$ of 
projective automorphisms of $K$ in $\SL_\pm(n-i_0+1, \bR)$ is a closed group. 


\begin{lemma}\label{lem-invmet} 
Let $D$ be a properly convex open 
domain with the closed locally compact group $\Aut(D)$ of smooth automorphisms of $D$.
Given a group $G$ acting isometrically on an open domain $D$ faithfully so that 
$G \rightarrow \Aut(D)$ is an embedding. 
Suppose that $D/G$ is compact. 
Then the closure $\bar G$ of $G$ is a Lie subgroup 
acting on $D$ properly, and there exists a smooth Riemannian metric on $D$ that is 
$\bar G$-invariant. 
\end{lemma} 
\begin{proof} 
Since $\bar G$ is in $\SL_\pm(n-i_0+1, \bR)$, the closure 
$\bar G$ is a Lie subgroup acting on $D$ properly. 
Suppose that $D \subset \SI^{n}$. 

One can construct a Riemannian metric $\mu$ with bounded entries. 
Let $\phi$ be a function supported on a compact set containing a fundamental domain $F$ of $D/G$ where $\phi|F > 0$.
Given a bounded subset of $\bar G$, the elements are in a bounded subset of the projective
automorphism group $\SL_{\pm}(n+1, \bR)$. A bounded subset of projective automorphisms 
have uniformly bounded set of derivatives on $\SI^{n}$ up to the $m$-th order for any $m$. 
We can assume  that the derivatives of the entries of $\phi \mu$ up to the $m$-th order are uniformly bounded above.
Let $d\eta$ be the left-invariant measure on $\bar G$.

Then $\{g^*\phi\mu| g \in \bar G\}$ is an equicontinuous family on any compact 
subset of $D^{o}$ up to any order. 
Thus the integral 
\[ \int_{g \in \bar G} g^* \phi\mu d \eta \]
of $g^* \phi\mu$ for $g \in \bar G$ is a $C^\infty$-Riemannian metric and that is positive definite. 
This bestows us a $C^\infty$-Riemannian metric $\mu_D$ on $D$ invariant under $\bar G$-action.
\end{proof} 

The foliation on $\tilde \Sigma_{\tilde E}$ given by fibers of $\Pi_K$
has leaves that are $i_0$-dimensional complete affine spaces.
Then $K^o$ admits a smooth Riemannian metric $\mu_{0, 1}$ invariant under $N_K$ by Lemma \ref{lem-invmet}. 
Since $N_K$ is not discrete, a component $N_{K, 0}$ of the closure of $N_K$ in $\Aut(K)$
is a Lie group of dimension $\geq 1$. 
By taking a finite index subgroup of $\pi_{1}(\orb)$, we may assume that $N_{K}$ is connected. 
We consider the orthogonal frame bundle $FK^o$ over $K^o$. 
A metric on each fiber of $FK^o$ is induces from $\mu_K$.
Since the action of $N_{K}$ is isometric on $FK^o$ with trivial stabilizers, 
we find that $N_{K}$ acts on a smooth orbit submanifold 
of $FK^o$ transitively 
with trivial stabilizers. 
(See Lemma 3.4.11 in \cite{Thbook}.)

There exists a bundle $F\tilde \Sigma_{\tilde E}$ from pulling back $FK^o$ by the projection map. 
Here, $F\tilde \Sigma_{\tilde E}$ covers $F\Sigma_{\tilde E}$. 
Since $\bGamma_{\tilde E}$ acts isometrically on $FK^o$, 
the quotient space $F\tilde \Sigma_{\tilde E}/\bGamma_{\tilde E}$ is a bundle $F\Sigma_{\tilde E}$ over 
$\Sigma_{\tilde E}$ with a subbundle with compact fibers isomorphic to the orthogonal group of dimension $n-i_0$. 
Also, $F\tilde \Sigma_{\tilde E}$ is foliated by $i_0$-dimensional affine 
spaces pulled-back from the $i_0$-dimensional leaves on the foliation $\tilde \Sigma_{\tilde E}$. 
One can think of these leaves as being the inverse images of points of $FK^{o}$.



\subsubsection{$\pi_{1}(V_{l}) $ is virtually solvable.} \label{sub-uniN}  

Recall the fibration \[\Pi_K:  \tilde \Sigma_{\tilde E} \ra K^o
\hbox{ which induces } \tilde \Pi_K: F \tilde \Sigma_{\tilde E} \ra F K^o.\]
Since $N_K$ acts as isometries of Riemannian metric on 
$K^o$, we can obtain a metric on $\Sigma_{\tilde E}$ 
so that the foliation is a Riemannian foliation. 
Let $p_{\Sigma_{\tilde E}}: F\tilde \Sigma_{\tilde E} \ra F \Sigma_{\tilde E}$ be the covering map
induced from $\tilde \Sigma_{\tilde E} \ra  \Sigma_{\tilde E}$.
The foliation on $\tilde \Sigma_{\tilde E}$ gives us a foliation of $F\tilde \Sigma_{\tilde E}$. 


Let $l$ be a leaf of $F\tilde \Sigma_{\tilde E}$,
and  $p$ be the image of $l$ in $FK^o$. 
Since $l$ maps to a polynomial growth leaf in $F\Sigma_{\tilde E}$ by Lemma \ref{lem-polynomial},
Carri\`ere \cite{Car} shows that a connected nilpotent Lie group
$A_l$ in the closure of $N_K$ in $\Aut(K)$ acts on $FK^o$ freely. 
Moreover, we have a submanifold
\begin{alignat}{3} 
\tilde \Pi_K^{-1}(A_l(p))   =: \tilde V_{l} &\, \hookrightarrow \, & F\tilde \Sigma_{\tilde E} \nonumber \\
\downarrow \, &                     &\, p_{\Sigma_{\tilde E}} \downarrow \nonumber \\
V_l      & \hookrightarrow & F \Sigma_{\tilde E}
\end{alignat} 
for a compact submanifold $V_l := \overline{p_{\Sigma_{\tilde E}}(l)}$ in $F \Sigma_{\tilde E}$.
Here $A_l$ is the component of the closure of $N_K$ the image of $\bGamma_{\tilde E}$ in $\Aut(K)$. 
Clearly $A_l$ is an algebraic group. 
Hence, by taking a finite cover if necessary, 
 $\bGamma_{\tilde E}$ is in a Lie group 
\[\bR^l \times Z(\Gamma_1) \times \dots \times Z(\Gamma_k), l \geq k-1\]
for the Zariski closure $Z(\Gamma_i)$ of $\Gamma_i$ by Theorem 1.1 of Benoist \cite{Ben2}. 
By taking a finite index subgroup, we assume that $\bGamma_{\tilde E}$ is a subgroup. 

Note $V_{l}$ has a dimension independent of $l$ since $A_{l}$ acts freely. 

Since $A_l$ is in the product group, we can project to each 
$\Gamma_i$-factor or the central $\bR^{l_0-1}$.
Since the image of $A_l$ is $Z(\Gamma_j)$ is not discrete in $\Aut(K_j)$,  
we obtain that $\Aut(K_{j})$ equals a union of 
components of copies of $PO(n_j, 1)$ or $SO(n_j, 1)$ in $Z(\Gamma_j)$ 
by Theorem 1.1 of \cite{Ben2} since $K_{j}$ is strictly convex. 
The nilpotency implies that the image is a cusp group fixing a unique point in $\Bd K_j$
or an abelian Lie group fixing a unique pair of points in $\Bd K_{j}$. 
Thus, the image is an abelian group since $A_l$ is connected. 
Thus, 
 $A_l$ is an abelian Lie group. 
 

Let $N_l$ be exactly the subgroup of $\pi_1(V_l)$ fixing a leaf $l$ in $FK^o$, 
for each closure $V_l$ of a leaf $l$, the manifold $V_l$ is compact and 
we have an exact sequence 
\[ 1 \ra N_l \ra h(\pi_1(V_l)) \stackrel{\Pi^*_K}{\longrightarrow} A'_l \ra 1.\]
Since the leaf $l$ is dense in $V_l$, it follows that $A'_l$ is dense in $A_l$. 
Each leaf $l'$ of $\tilde \Sigma_{\tilde E}$ has a realization a subset in $\torb$. 
Since $N_l$ fixes every points of $K^o$ and $N$ is in $\pi_1(V_l)$, we obtain $N = N_l$. 
We have the norms of eigenvalues $\lambda_i(g) = 1$ for $g \in N_l$.  
By Proposition \ref{prop-eigSI}, we have that $N=N_l$ is orthopotent
since the norms of eigenvalues equal $1$ identically and $N_l$ is discrete. Then $N$ is easily seen to be virtually nilpotent
since it is of polynomial growth as we can deduce 
from the orthopotent flags.
(See the proof of Theorem \ref{I-thm-affinehoro} of \cite{EDC1} also. )

Hence, $h(\pi_{1}(V_{l}))$ is solvable being an extension of an abelian group by a nilpotent group.


We summarize below: 
\begin{proposition}\label{prop-V_l} 
Let $l$ be a generic fiber of $F\tilde \Sigma_{\tilde E}$ and $p$ be the corresponding point $p$ of $FK^o$. 
Then there exists an algebraic abelian group $A_l$ acting on $FK^o$ so that 
$\tilde \Pi_K^{-1}(A_l(p)) =\tilde V_l$ covers a compact suborbifold $V_l$ in $F\Sigma_{\tilde E}$, 
a conjugate of the image of the holonomy group of $V_l$ is a dense subgroup of $A_l$, and 
the holonomy group of $V_{l}$ is solvable. Moreover, $\tilde V_{l}$ is homotopy equivalent to 
a point or a torus of dimension $\geq 1$. 
\end{proposition} 
\begin{proof} 
We just need to prove the last statement. 
Since $A_{l}$ is homotopy equivalent to a point or a torus of dimension $\geq 0$, 
and $\tilde \Pi_{K}$ has fibres that are $i_{0}$-dimensional open hemispheres, 
this last statement follows. 
\end{proof}


\subsubsection{The subgroup $\pi_{1}(V_{l})$ is normalized by $\bGamma_{\tilde E}$.}\label{subsub-syndetic}



The leaf holonomy acts on $F\tilde \Sigma_{\tilde E}/\mathcal{F}_{\tilde E}$ as an abelian killing field group
without any fixed points. 
Hence, each leaf $l$ is in $\tilde V_l$ with a constant dimension. 
Thus, $\mathcal{F}_{\tilde E}$ is a foliation with leaf closures of the identical dimensions. 

The leaf closures form another foliation $\overline{\mathcal{F}}_{\tilde E}$ with compact leaves by Lemma 5.2 of Molino \cite{Molbook}. 
We let $F\Sigma_{\tilde E}/\overline{\mathcal{F}}_{\tilde E}$ denote the space of closures of leaves has an orbifold structure
where the projection $F\Sigma_{\tilde E} \ra F\Sigma_{\tilde E}/\overline{\mathcal{F}}_{\tilde E}$ is an orbifold morphism
by Proposition 5.2 of \cite{Molbook}.
Since $\Sigma_{\tilde E}$ has a geometric structure induced from the transverse real projective structure, 
$\Sigma_{\tilde E}$ is a very good orbifold. We may assume that $\Sigma_{\tilde E}$ is an $n-1$-manifold
and hence $F\Sigma_{\tilde E}$ is a manifold since we need our results for finite index subgroups only. 
By Lemma 5.2 of \cite{Molbook}, each open neighbourhood of $F\Sigma_{\tilde E}/\overline{\mathcal{F}}_{\tilde E}$ 
is the quotient space of $A_{l}$-invariant open set in $FK^{o}=F\tilde \Sigma_{\tilde E}/\mathcal{F}_{\tilde E}$ by the connected 
abelian group $A_l$ acting properly with trivial stabilizers.
\begin{itemize}
\item Let $X = (FK^{o})/A_{l}$ be a quotient manifold, and 
\item let $G$ be the group of projective automorphisms of $K^{o}$ acting on 
$FK^{o}/A_{l}$ induced from $\bGamma_{\tilde E}$. 
\end{itemize}
Thus, $F\Sigma_{\tilde E}/\overline{\mathcal{F}}_{\tilde E}$ admits a $(G, X)$-geometric structure  induced from the real projective structure 
of $F\tilde \Sigma_{\tilde E}/\mathcal{F}_{\tilde E}$. 
There exists a finite regular manifold-cover $M$ of 
$F\Sigma_{\tilde E}/\overline{\mathcal{F}}_{\tilde E}$
as in Chapter 13 of Thurston \cite{Thnote} (see Theorem 2 (due to Thurston) of \cite{Choi2004} also.)

By pulling back the fiber bundle over orbifolds, we consider the fundamental groups.
We obtain  a regular finite cover $F\Sigma_{\tilde E}^f$ of $F\Sigma_{\tilde E}$ and 
a regular fibration 
\begin{alignat}{3}
V_l  \quad &\longrightarrow &\quad  F\Sigma_{\tilde E}^f & \, \, \longrightarrow & M  \nonumber \\
\downarrow \quad &  & \downarrow\quad & & \downarrow \nonumber \\ 
V_l \quad &\longrightarrow &\quad F\Sigma_{\tilde E} & \ra & F\Sigma_{\tilde E}/\bar F 
\end{alignat}
where $V_l$ is a generic fiber of $F\Sigma_{\tilde E}^f$ for the induced foliation $\bar F^f$ isomorphic to 
a generic fiber of $F\Sigma_{\tilde E}$. 

We obtain an exact sequence 
\[  \pi_1(V_l) \ra \pi_1(F\Sigma_{\tilde E}^f) \stackrel{\pi'_K}{\longrightarrow} \pi_1(M) \ra 1\]
and the image $\pi_1(V_l)$ is a normal subgroup of $\pi_1(F\Sigma_{\tilde E}^f)$.  
Since $F\Sigma_{\tilde E}^f$ is fibered by fibers diffeomorphic to ${\SO}(n-i_0)$ or its cover, 
we have a fibration 
\[\widetilde{\SO}(n-i_0) \ra F\Sigma_{\tilde E}^f \ra \Sigma_{\tilde E}^f\]
where $\Sigma_{\tilde E}^f$ is a finite cover of $\Sigma_{\tilde E}$
and $\widetilde{\SO}(n-i_0)$ is a finite cover of $\SO(n-i_0)$. Thus, 
we also have an exact sequence 
\[ \pi_1(\widetilde{\SO}(n-i_0)) \ra \pi_1(F\Sigma_{\tilde E}^f) \ra \pi_1(\Sigma_{\tilde E}^f) \ra 1.\]
Since $\pi_1(\Sigma_{\tilde E}^f)$ is a quotient group of $\pi_1(F\Sigma_{\tilde E}^f)$, 
the image of $\pi_1(V_l)$ is a normal subgroup of $\pi_1(\Sigma_{\tilde E}^f)$
for the generic $l$. 
We define $\bGamma_l$ as the image $h(\pi_1(V_l))$.
The above sequence tells us that 
$\bGamma_l$ is a normal subgroup of a finite index subgroup of $\bGamma_{\tilde E}$. 

From now on, we will assume that $\bGamma_l$ is a normal subgroup of $\bGamma_{\tilde E}$ by 
taking a finite cover of the end-neighborhood if necessary.

Recall that $\bGamma_{l}$ is virtually solvable, as we showed above. 
We let $Z(\bGamma_{\tilde E})$ and $Z(\bGamma_l)$ denote the Zariski closures in $\Aut(\SI^n)$ of 
$\bGamma_{\tilde E}$ and $\bGamma_l$ respectively. 

By Theorem 1.6 of Fried-Goldman \cite{FG}, 
there exists a closed virtually solvable Lie group $S_l$ containing $\bGamma_l$ 
with the following four properties: 
\begin{itemize}
\item $S_l$ has finitely many components.
\item $\bGamma_l\backslash S_l$ is compact. 
\item The Zariski closure $Z(S_l)$ is the same as 
$Z(\bGamma_l)$. 
\item Finally, we have solvable ranks 
\begin{equation}\label{eqn-rankSl}
{\mathrm{rank}}(S_l) \leq {\mathrm{rank}} (\bGamma_l). 
\end{equation}
\end{itemize}

Since $\bGamma_{\tilde E}$ normalizes $\bGamma_l$ by above, 
$\bGamma_{\tilde E}$ also normalizes 
$Z(\bGamma_l)=Z(S_l)$; However, it maybe not normalize $S_l$ itself. 




Since $\bGamma_l$ acts on an algebraic set $\tilde V_l \subset F\tilde \Sigma_{\tilde E}$, a component of 
the inverse image of an algebraic set  the algebraic orbit $A_l(p)$ in $FK^{o}$.
Thus, $Z(\bGamma_l) = Z(S_l)$ also acts on $\tilde V_l$
and hence so does $S_l$. 
Also, since $\bGamma_l \ra A'_l$ has a dense image, 
$S_l \ra A_l$ is an onto map.

We summarize: 
\begin{lemma}\label{lem-V} 
$h(\pi_1(V_l))$ is virtually solvable and  is contained in a virtually solvable Lie group $S_l := S(h(\pi_1(V_l))$ 
with finitely many components, and $S_l/h(\pi_1(V_l))$ is compact. 
$S_l$ acts on $\tilde V_l$. 
Furthermore, one can modify a p-end-neighborhood $U$ so that 
$S_l$ acts on it. Also the Zariski closure of $h(\pi_1(V_l))$ is the same as that of $S_l$. 
\end{lemma} 
\begin{proof} 
By above, $Z(S_l) = Z(\bGamma_l)$ acts on $\tilde V_l$. 
We need to prove about the p-end-neighborhood only. 
Let $F$ be a compact fundamental domain of $S_l$ under the $\Gamma_l$. 
Then we have 
\[ \bigcap_{g\in S_l} g(U) = \bigcap_{g \in F} g(U).\]
Since $F$ is compact, the latter set is still a p-end-neighborhood. 
\end{proof}

Since $S_l$ acts on $U$ and hence on $\tilde \Sigma_{\tilde E}$ as shown in Lemma \ref{lem-V}, 
we have a homomorphism $S_l \ra \Aut(K)$.  
We define by $S_{l, 0}$ the kernel of this map. 
Then $S_{l, 0}$ acts on each leaf of $\tilde \Sigma_{\tilde E}$.

\subsubsection{The form of $US_{l, 0}$.}






From now on, we will let $S_l$ to denote the only the identity component of itself for simplicity
as $S_l$ has a finitely many components to begin with. 
This will be sufficient for our purpose of getting a cusp group normalized by $\bGamma_{\tilde E}$. 

Let $US_l$ denote the unipotent radical of the Zariski closure
$Z(S_l)$ of $S_l$ in $\Aut(\SI^n)$, which is a solvable Lie group. 
Also, $US_{l, 0}$ denote the unipotent radical of the Zariski closure of $S_{l, 0}$. 
Since $S_{l, 0}$ is normalized by $\bGamma_{\tilde E}$, 
so is $Z(S_{l, 0})$. 

Let $\SI^{i_{0}+1}_{l}$ denote the $i_{0}+1$-dimensional great sphere containing $\SI^{i_{0}}_{\infty}$ 
corresponding to each $i_{0}$-dimensional leaf $l$ of $\mathcal{F}_{\tilde E}$. 

\begin{proposition}\label{prop-ZN} 
Let $l$ be a generic fiber so that $A_l$ acts with trivial stabilizers. 
\begin{itemize} 
\item $S_l$ acts on $\tilde V_l$ and on $\tilde \Sigma_{\tilde E}$ and $\partial U$ freely and properly
and acts as isometries on these spaces with respect to Riemannian metrics. 
\item $S_{l,0}$ acts transitively on each leaf $l$ with a compact stabilizer
and acts on an $i_0$-dimensional ellipsoid $\partial U \cap \SI^{i_{0}+1}_{l}$ 
passing $\bv_{\tilde E}$ with an invariant Euclidean metric. 
\item $S_{l, 0}$ is an $i_0$-dimensional cusp group
and the unipotent radical $US_{l, 0}$ equals $S_{l, 0}$. 
\item $US_{l, 0}$  is normalized by $\bGamma_{\tilde E}$ also. 
\end{itemize}
\end{proposition}  
\begin{proof} 
Since $Z(S_l) = Z(\bGamma_l)$ acts on $\tilde V_l$ as stated above, it follows 
that $S_l$ and $US_l$ both in the group act on $\tilde V_l$. 

(i) A stabilizer $S_{l, x}$ of each point $x \in \tilde V_l$ for $S_l$ is compact: 
let $F$ be the fundamental domain of $S_l$ with $\Gamma_l$ action. 
Let $F'$ be the image $F(x):= \{g (x)| g \in F\}$ in $\tilde V_l$. This is a compact set. 
Define
\[\Gamma_{l, F'} := \{g \in \Gamma_l|  g(F(x)) \cap F(x) \ne \emp\}.\] 
Then $\Gamma_{l, F'}$ is finite by the properness of the 
action of $\Gamma_l$. 
Since an element of $S_{l, x}$ is a product of an element $g'$ of $\Gamma_l$ and $f \in F$, and
$g' f(x) = x$, it follows that $g'F(x) \cap F(x) \ne \emp$ and $g' \in \Gamma_{l, F}$. 
Hence $S_{l, x} \subset \Gamma_{l, F'}F$ and $S_{l, x}$ is compact.
Similarly, $S_l$ acts properly on $\tilde \Sigma_{\tilde E}$. 
Since $\partial U$ is in one-to-one correspondence with $\tilde \Sigma_{\tilde E}$, 
$S_l$ acts on $\partial U$ properly. Hence, these spaces have 
compact stabilizers with respect to $S_l$. 
The invariant metric follows by Lemma \ref{lem-invmet}. 
Hence, the action is proper and the orbit is closed. 
(Since $\tilde V_l/\Gamma_l$ is compact, $\tilde V_l /S_l$ is compact also. )


(ii) We assume that $\bGamma_{\tilde E}$ is torsion-free by taking a finite index subgroup 
since $\Sigma_{\tilde E}$ is a very good orbifold, admitting a geometric structure. 
Now, we show that $S_l$ acts freely on 
$\tilde \Sigma_{\tilde E}$: 

The strategy is as follows. 
We use the last part of 
Section 1.8 of \cite{FG} where we can replace $H$ there with $S_l$ and $\bR^n$ with $\tilde V_l$ and $\Gamma$ with 
the solvable subgroup $\bGamma_l$, we obtain the results for $\bGamma_l$: 

First, $\bGamma_l$ is solvable and discrete, and hence is virtually polycyclic by Mostow 
(see Proposition 3.7 of \cite{Rag})
and $S_l$ has the same Zariski closure as $\bGamma_l$. 
Take a finite index subgroup of $\bGamma_{\tilde E}$ so that $\bGamma_l$ is now polycyclic. 
We work on the projection of $\tilde V_l$ on 
$\tilde \Sigma_{\tilde E}$, a convex but not properly convex open domain in an affine space $A^{n-1}$.  

Lemma 1.9 of \cite{FG}
shows that the unipotent radical $US_l$ of $Z(S_l)$ acts freely on $\tilde \Sigma_{\tilde E}$:
Being unipotent, $US_l$ is simply connected. The orbit $US_l(x)$ for $x \in \tilde \Sigma_{\tilde E}$ 
is simply connected and invariant under $Z(\bGamma_l)= Z(S_l)$. 
$\bGamma_l\backslash US_l(x)$ is a $K(\bGamma_l, 1)$-space. 
Thus, $\ranK \bGamma_l = cd \bGamma_l \leq \dim US_l$. 
By Lemma 4.36 of \cite{Rag}, $\dim US_l \leq \dim S_l$ and 
by Lemma 1.6 (iv) of \cite{FG}, we have $\dim S_l \leq \ranK \bGamma_l$.
Thus, $\ranK \bGamma_l = \dim S_l$. 

We now show $S_l$ acts freely on $\tilde \Sigma_{\tilde E}$.
We have a fibration sequence 
\[ \bGamma_l \ra S_l \ra \bGamma_l\backslash S_l \] and 
an exact sequence 
\[\pi_1(S_l) \ra \pi_1(\bGamma_l\backslash S_l) \ra \bGamma_l,\]
and hence $\ranK \pi_1(S_l) + \ranK \bGamma_l = \ranK \pi_1(\bGamma_l\backslash S_l) = \dim S_l$
since $S_l$ is solvable and $\bGamma_l\backslash S_l$ is a compact manifold following the argument in Section 
1.8 of \cite{FG}. (See Proposition 3.7 of \cite{Rag} also where we need to take the universal cover of $S_l$.)
Since $\ranK \bGamma_l = \dim S_l$, we have $\ranK \pi_1(S_l) = 0$. 
This means that $\pi_1(S_l)$ is finite. Being solvable, it is trivial. 
Thus, $S_l$ is simply connected. Since $S_l$ is homotopy equivalent to $T^{j_1}$, 
$S_l$ is contractible. (We followed Section 1.8 of \cite{FG} faithfully here.)

Since $S_l$ acts transitively on any of its orbits, 
$S_l$ is homotopy equivalent to a bundle over the orbit with 
fiber homeomorphic to a stabilizer. Since $S_l$ is contractible, the stabilizer is finite. 
Since $S_l$ acts with finite stabilizers on $\tilde \Sigma_{\tilde E}$, it acts so on $\tilde V_l$. 
That is $S_{l}$ finitely covers $\tilde V_{l}$ as a universal cover. That is, $\pi_{1}(\tilde V_{l})$ is finite. 
Since $\tilde V_{l}$ is homotopy equivalent to a point or a torus, $\pi_{1}(\tilde V_{l})$ and the stabilizers are trivial. 
%
We showed that $S_l$ acts freely on $\tilde V_l$.

(iii) Now, we show that $S_l$ acts transitively on
$\tilde V_{l}$:  Choose $x \in \tilde V_l$. 
There is a map $f: \bGamma_l\backslash S_l \ra \tilde V_l/\bGamma_{l}$ given by
sending each $g \in S_l$ to $g(x) \in \tilde V_l$. 
The image of the map is also closed since $\bGamma_l\backslash S_l$ is compact.
Since the map is a homotopy equivalence, 
the map is onto and $S_l$ acts transitively on $\tilde V_l$. 


(iv) Hence, $S_{l, 0}$ acts simply transitively on each $l$;
$S_{l, 0}$ is diffeomorphic to a leaf $l$ and hence is connected
and is a solvable Lie group. 

Since the subset $U_{l}:= U \cap H^{i_0+1}_l$ of $U$ corresponding to $l$ is a strictly convex set containing $v_{\tilde E}$, 
we have $S_{l, 0}$ acting simply transitively on $\partial U_{ l}$. 
Proposition \ref{prop-eigSI} implies that  for $g\in \Gamma_{l}$ 
\[\lambda_1(g) \geq \lambda(g) \geq  \lambda'(g) \geq \lambda_{n+1}(g).\]
Since $S_{l} = F \Gamma_{l}$ for a compact set $F$, 
this inequality 
\begin{equation}\label{eqn-eigen2} 
C_{1}\lambda_1(g) \geq \lambda(g) \geq  C_{2}\lambda'(g) \geq C_{3}\lambda_{n+1}(g), g \in S_{l}
\end{equation} 
holds for constants $C_{1}>1, 1> C_{2} > C_{3}> 0$. 
Since $S_{l, 0}$ acts trivially on $K^{o}$, we have 
$\lambda_{1}(g) = \lambda_{n+1}(g)$ for $g \in S_{l, 0}$.  
Since the maximal norm $\bar \lambda(g)$ of the eigenvalue equals 
$\max\{\lambda_{1}(g), \lambda(g) \}$ 
and the minimal norm $\hat \lambda(g)$ of the eigenvalue equals $\min \{\lambda'(g), \lambda_{n+1}(g)\}$, 
equation \eqref{eqn-eigen2} implies that 
$|\log\bar \lambda(g)|, |\log\hat \lambda(g)|, g\in S_{l, 0}$ are both uniformly bounded above.  
Of course we have
\[|\log\bar \lambda(g^{n})|= |nlog\bar \lambda(g)|,  |\log\hat \lambda(g^{n})| = |n\log\hat \lambda(g)|, g \in S_{l, 0}.\]
Since $S_{l, 0}$ is not compact,
all the eigenvalues of elements are $1$. 
Since $S_{l, 0}$ is a connected Lie group, Fried \cite{Fried86} shows that 
$S_{l, 0}$ is a nilpotent Lie group. 
By Lemma \ref{I-lem-unithoro} of \cite{EDC1}, 
$S_{l, 0}$ acts on an $i_0$-dimensional ellipsoid that has to equal $\partial U_{l}$. 
Since one can identify each leaf with an affine space $S_{l, 0}$ is isomorphic to an affine isometry group 
acting simply transitively on an affine space $\bR^i$.
Let ${\mathcal H}_{\bv_{\tilde E}}$ denote the cusp group acting on the ellipsoid. 
An elementary argument using the cocompact subgroup simultaneously in both groups
shows that $S_{l,0}$ and ${\mathcal H}_{\bv_{\tilde E}}$ are identical.

This shows also that $S_{l, 0}$ is nilpotent and we have $US_{l, 0} = S_{l, 0}$ also. 
Finally, this implies that $US_{l, 0}$ is an $i_0$-dimensional abelian Lie group. 

For $g \in \bGamma_{\tilde E}$, 
$S'_l := g S_l g^{-1}$ is a syndetic hull of $\bGamma_{\tilde E}$. 
Then we define $S'_{l, 0}$ as the subgroup acting trivially on the space of leaves. 
Since $S'_{l, 0}$ has to be the cusp group as above by the same proof, it 
follows that $S'_{l, 0} = S_{l, 0} = gS_{l, 0}g^{-1}$. Thus, $S_{l, 0}$ is a normal subgroup. 

\end{proof}


\subsection{The proof for indiscrete $N_{K}$.}

We can parametrize $US_{l, 0}$ by $\CN(\vec{v})$ for $\vec{v} \in \bR^{i_0}$ by Proposition \ref{prop-ZN}. 
Hypothesis \ref{h-norm} holds now. 
As above by Lemmas  \ref{lem-similarity} and \ref{lem-conedecomp1}, we have that the matrices are of form: 
\renewcommand{\arraystretch}{1.5}
\begin{equation} 
\newcommand*{\temp}{\multicolumn{1}{r|}{}}
\CN(\vec{v}) = \left( \begin{array}{ccccccc} 
\Idd_{n-i_0-1} & \temp & 0 & \temp & 0 & \temp & 0 \\ 
 \cline{1-7}
0 &\temp & 1 &\temp & 0 &\temp & 0 \\ 
 \cline{1-7}
0 &\temp & \vec{v}^T &\temp & \Idd_{i_0} &\temp & 0 \\ 
 \cline{1-7}
c_2({\vec{v}}) &\temp & ||\vec{v}||^2 /2&\temp & \vec{v} &\temp & 1 
\end{array} 
\right), 
\end{equation} 
\begin{equation} \label{eqn:stdg} 
\newcommand*{\temp}{\multicolumn{1}{r|}{}}
g = \left( \begin{array}{ccccccc} 
S(g) & \temp & 0 & \temp & 0 & \temp & 0 \\ 
 \cline{1-7}
0 &\temp & a_1(g) &\temp & 0 &\temp & 0 \\ 
 \cline{1-7}
C_1(g) &\temp & a_4(g) &\temp & a_5(g) O_5(g) &\temp & 0 \\ 
 \cline{1-7}
c_2(g) &\temp & a_7(g) &\temp & a_8(g) &\temp & a_9(g) 
\end{array} 
\right)
\end{equation}
where $g \in \bGamma_{\tilde E}$. Recall 
$\mu_g = a_5(g)/a_1(g) = a_9(g)/a_5(g)$. 
Since $S_l$ is in $Z(\bGamma_l)$ and the orthogonality of normalized $A_5(g)$
is an algebraic condition, the above form also holds for $g \in S_l$.

However, we don't assume Hypothesis \ref{h-qjoin}. 
We continue to assume as in Hypothesis \ref{h-norm} for $G$.  

\begin{proposition} \label{prop-mug}
A center of $\bGamma_{\tilde E}$ maps to $N_{K}$
going to a Zariski dense group of the virtual center of $\Aut(K)$.  
We assume Hypothesis \ref{h-norm} and $N_{K}$ is indiscrete.  Then
we have $\mu_g= 1$ for every $g \in \bGamma_{\tilde E}$. 
\end{proposition}
\begin{proof} 
We write $G = \bGamma_{\tilde E}$. 
We factorize the matrix of $g$, $g \in G$,  
\begin{equation} \label{eqn-eta}
\left(
\begin{array}{c|c|c}
a_{1}(g) & 0 & 0 \\ 
\hline
a_{1}(g) \vec{v}^T_g & a_{5}(g) O_5(g) & 0 \\ 
 \hline
a_7(g) &  a_{5}(g) \vec{v}_g O_5(g) & a_{9}(g)
\end{array} 
\right)  = 
\end{equation} 
\begin{equation}
\left(
\begin{array}{c|c|c}
1 & 0 & 0 \\ 
\hline
0 & 1& 0 \\ 
 \hline
\frac{a_7(g)}{a_{1}(g)} -  \frac{||\vec{v}_{g}||^{2}}{2} &  0 & 1
\end{array} 
\right) 
\left(
\begin{array}{c|c|c}
1 & 0 & 0 \\ 
\hline
\vec{v}^T_g & \Idd & 0 \\ 
 \hline
\frac{||\vec{v}_{g}||^{2}}{2} &  \vec{v}_g & 1
\end{array} 
\right) 
a_{1}(g)
\left(
\begin{array}{c|c|c}
1 & 0 & 0 \\ 
\hline
0 & \mu_{g} & 0 \\ 
 \hline
0 &  0 & \mu_{g}^{2}
\end{array} 
\right) 
\end{equation}

The weak uniform middle eigenvalue condition here means
$a_{1}(g) \geq a_{9}(g)$ or $\mu_{g}  \leq 1$ for $g \in G_{+}$ for $G = \bGamma_{\tilde E}$. 

By Hypothesis \ref{h-norm}, the conclusion of Proposition \ref{prop-decomposition} holds. 
From the proof of Proposition \ref{prop-decomposition}, 
let $K$, $K''$, and $k$ be as in the proof.
We obtain a sequence $\gamma_{m}$ in the virtual center with the same properties. 
We take one as $\eta$ where the largest norm $\hat \lambda_{1}(\eta)$ of the eigenvalues 
for $\hat S(\eta)$ occurs at $k$. 

By Proposition \ref{prop-decomposition}, $\eta$ acts on $K'' \subset \Bd \torb$.
By the weak middle eigenvalue condition and the matrix form \eqref{eqn-eta}, 
$k$ corresponds to $a_{1}(\eta) = \lambda_{1}(g)$, the largest norm eigenvalue. 
Since $\hat \lambda_{1}(\eta)$ is of multiplicity one, 
$k$ corresponds to a unique fixed point $k_{\eta}$. 
Being a fixed point,  $k_{\eta} \in \Bd \torb$. 

Suppose first that $\mu_{\eta} < 1$. 
Since $\bv_{\tilde E}$ has a different eigenvalue $a_{9}(\eta)$ from that $a_{1}(\eta)$ of $k_{\eta}$
as $\mu_{\eta} < 1$, we obtain $k_{\eta} \ne \bv_{\tilde E}$, and $k_{\eta} \in H_{k}^{o}$. 

The convex hull $\hat K$ of $K''$ and $k_{\eta}$ is the join $K'' \ast k_{\eta} \subset \clo(\torb)$. 
Hence, there exists a subspace $\SI^{n-i_{0}} \subset K'' \ast k_{\eta}$ complementary to $\SI^{i_{0}}_{\infty}$. 
We use the coordinates now where $K'' \ast k_{\eta}$ is given 
by $x_{n-i_{0}+1}=0, \dots, x_{n+1} = 0$. 
Points of $K''$ has coordinates 
\begin{gather}
[ \underbrace{\ast, \dots, \ast,}_{n-i_{0}-1} 0, \underbrace{0, \dots, 0}_{i_{0}+1}] \\
k_{\eta} = [\underbrace{0, \dots, 0,}_{n-i_{0}-1} 1, \underbrace{0, \dots, 0}_{i_{0}+1}]
\end{gather}
Since $G$ commute with $\eta$ and $\eta$ have eigenvalues at $K''$ and $k_{\eta}$ distinct from fixed points outside 
$K''$ and $k_{\eta}$ as we can see from equation \eqref{eqn:stdg}, $G$ acts on $K''\ast k_{\eta}$. 
Since $G$ acts on $K'' \ast k_{\eta}$ by Proposition \ref{prop-decomposition}, 
it follows that in this coordinate system
\[\vec{v}_{g} = 0, \frac{a_7(g)}{a_{1}(g)} -  \frac{||\vec{v}_{g}||^{2}}{2} = 0 \hbox{ for all } g \in G. \]
Since $\bv_{\tilde E}$ is not contained in $K''\ast k_{\eta}$, and each leaf hemisphere $H^{i_{0}+1}_{l}$ contains at least one point of 
$K^{o}= (K''\ast k_{\eta})^{o}$, it follows that 
$(K''\ast k_{\eta})^{o}$ projects to a submanifold of $\tilde \Sigma_{\tilde E}$ transversal to each fiber.
Sine $\bGamma_{\tilde E}$ acts on this image, 
$\bGamma_{\tilde E}$ cannot act properly discontinuously on $\tilde \Sigma_{\tilde E}$. 

Now suppose that $\mu_{\eta}=1$ but there exists some $g$ with $\mu_{g}< 1$. 
Then again there exists a fixed point $k_{g}\in H^{o}_{k}$ with an eigenvalue of multiplicity one. 
By commutativity, $\eta$ acts on $k_{g}$ and the arguments are now similar. 

\end{proof}

\begin{proof}[{The proof of Theorem \ref{thm-thirdmain}.}]
Suppose that $\tilde E$ is an NPCC R-end. 
When $N_{K}$ is discrete, Theorem  \ref{thm-NPCCcase}  gives us the result.

When $N_{K}$ is indiscrete, 
Hypothesis \ref{h-norm} holds by Propositions \ref{prop-ZN}. 

By Proposition \ref{prop-mug}, $\mu \equiv 1$ holds.
Lemmas \ref{lem-similarity} and \ref{lem-conedecomp1},
Propositions \ref{prop-decomposition} and \ref{prop-qjoin}
show that we have a joined or quasi-joined end. 
Theorem \ref{thm-NPCCcase2} implies the result. 
\end{proof}



\begin{corollary} \label{cor-NPCChol} 
Let $\orb$ be a properly convex strongly tame real projective orbifold.
Assume that holonomy group is strongly irreducible. 
Let $\tilde E$ be an NPCC p-end of the universal cover $\torb$ or $\orb$. 
Then the holonomy group 
$h(\bGamma_{\tilde E})$ is a group whose element under a coordinate system is of form{\,\rm :} 
\begin{equation}
\newcommand*{\temp}{\multicolumn{1}{r|}{}}
g = \left( \begin{array}{ccccccc} 
S(g) & \temp & 0 & \temp & 0 & \temp & 0 \\ 
 \cline{1-7}
0 &\temp & \lambda(g) &\temp & 0 &\temp & 0 \\ 
 \cline{1-7}
0 &\temp & \lambda(g) v(g)^{T}&\temp & \lambda(g) \Idd &\temp & 0\\ 
 \cline{1-7}
0 &\temp & \lambda(g)\left(\alpha_{7}(g)+ \frac{||v(g)||^{2}}{2}\right) 
&\temp & \lambda(g) v(g) &\temp & \lambda(g) 
\end{array} 
\right)
\end{equation}
where $\{S(g)|g\in \bGamma_{\tilde E}\}$ acts cocompactly on a properly convex domain in $\Bd \torb$
of dimension $n-i_{0}-1$,  
and $\alpha_{7}(g)$ satisfies the uniform positive translation condition given by equation \eqref{eqn:uptc}. 

And $\bGamma_{\tilde E}$ virtually normalizes the group 
\renewcommand{\arraystretch}{1.5}
\begin{equation} 
\newcommand*{\temp}{\multicolumn{1}{r|}{}}
\Bigg \{\CN(\vec{v}) 
= \left( \begin{array}{ccccccc} 
\Idd_{n-i_0-1} & \temp & 0 & \temp & 0 & \temp & 0 \\ 
 \cline{1-7}
0 &\temp & 1 &\temp & 0 &\temp & 0 \\ 
 \cline{1-7}
0 &\temp & \vec{v}^T &\temp & \Idd_{i_0} &\temp & 0 \\ 
 \cline{1-7}
0 &\temp & ||\vec{v}||^2 /2&\temp & \vec{v} &\temp & 1 
\end{array} 
\right)\Bigg | \vec{v} \in \bR^{i_{0}} \Bigg\}. 
\end{equation}

\end{corollary} 
\begin{proof} 
The proof is contained in the proof of Theorem \ref{thm-thirdmain}.
\end{proof}


\section{The dual of NPCC-ends} \label{sec-dualNPCC}

\subsection{The duality}
We repeat some background material from \cite{EDC1} for convenience. 
We recall the projective duality from linear duality. 
Let $\Gamma$ be a group of linear transformations $\GL(n+1, \bR)$. 
Let $\Gamma^*$ be the {\em affine dual group} defined by $\{g^{\ast -1}| g \in \Gamma \}$. 
Suppose that $\Gamma$ acts on a properly convex cone $C$ in $\bR^{n+1}$ with the vertex $O$.

An open convex cone  $C^*$ in $\bR^{n+1 *}$  is {\em dual} to an open convex cone $C $ in $\bR^{n+1}$  if 
$C^* \subset \bR^{n+1 \ast}$ is the set of linear transformations taking positive values on $\clo(C)-\{O\}$.
$C^*$ is a cone with vertex as the origin again. Note $(C^*)^* = C$. 

Now $\Gamma^*$ will acts on $C^*$.
A {\em central dilatational extension} $\Gamma'$ of $\Gamma$ by $\bZ$ is given by adding a dilatation by a scalar 
$s \in \bR_+ -\{1\}$ for the set $\bR_+$ of positive real numbers. 
The dual $\Gamma^{\prime \ast}$ of $\Gamma'$ is a central dilatation extension of $\Gamma^*$. 
 Also, $\Gamma'$ acts cocompactly on $C$ if and only if $\Gamma^{\prime *}$ acts so on $C^*$. 
 (See \cite{wmgnote} for details.)

 Given a subgroup $\Gamma$ in $\PGL(n+1, \bR)$, a {\em lift} in $\GL(n+1, \bR)$ is 
 any subgroup that maps to $\Gamma$ injectively.
 Given a subgroup $\Gamma$ in $\Pgl$, the dual group $\Gamma^*$ is the image in $\Pgl$ of the dual of 
 any linear lift of $\Gamma$. 

A properly convex open domain $\Omega$ in $P(\bR^{n+1})$ is {\em dual} to a properly convex open domain
$\Omega^*$ in $P(\bR^{n+1, \ast})$ if $\Omega$ corresponds to an open convex cone $C$ 
and $\Omega^*$ to its dual $C^*$. We say that $\Omega^*$ is dual to $\Omega$. 
We also have $(\Omega^*)^* = \Omega$ and $\Omega$ is properly convex if and only if so is $\Omega^*$. 

We call $\Gamma$ a {\em divisible group} if a central dilatational extension acts cocompactly on $C$.
$\Gamma$ is divisible if and only if so is $\Gamma^*$. 

Recall $\SI^n := {\mathcal{S}}(\bR^{n+1})$. We define $\SI^{n\ast} := {\mathcal{S}}(\bR^{n+1 \ast})$.

For an open properly convex subset $\Omega$ in $\SI^{n}$, the dual domain is defined as the quotient 
of the dual cone of the cone corresponding to $C_\Omega$ in $\SI^{n\ast}$. The dual set is also open and properly convex
but the dimension may not change.  
Again, we have $(\Omega^*)^* =\Omega$. 

Given a properly convex domain $\Omega$ in $\SI^n$ (resp. $\bR P^n$), 
we define the {\em augmented boundary} of $\Omega$
\[\Bd^{\Ag} \Omega  := \{ (x, h)| x \in \Bd \Omega, h \hbox{ is a supporting hyperplane of } \Omega, 
h \ni x \} .\] 
Each $x \in \Bd \Omega$ has at least one supporting hyperspace, 
a hyperspace is an element of $\bR P^{n \ast}$ since it is represented as a linear functional,   
and an element of $\bR P^n$ represents a hyperspace in $\bR P^{n \ast}$.

The homeomorphism below will be known as the {\em duality map}. 
\begin{proposition}[\cite{EDC1}] \label{prop-duality}
Let $\Omega$ and $\Omega^*$ be dual domains in $\SI^{n \ast}$ {\rm (}resp. $\bR P^{n \ast}${\rm ).} 
\begin{itemize}  
\item[(i)] There is a proper quotient map $\Pi_{\Ag}: \Bd^{\Ag} \Omega \ra \Bd \Omega$
given by sending $(x, h)$ to $x$. 
\item[(ii)] Let a projective automorphism group 
$\Gamma$ acts on a properly convex open domain $\Omega$ if and only 
$\Gamma^*$ acts on $\Omega^*$.
\item[(iii)] There exists a duality homeomorphism 
\[ {\mathcal{D}}: \Bd^{\Ag} \Omega \leftrightarrow \Bd^{\Ag} \Omega^* \] 
given by sending $(x, h)$ to $(h, x)$ for each $(x, h) \in \Bd^{\Ag} \Omega$. 
\item[(iv)] Let $A \subset \Bd^{\Ag} \Omega$ be a subspace and $A^*\subset \Bd^{\Ag} \Omega^*$
be the corresponding dual subspace $\mathcal{D}(A)$. If a group $\Gamma$ acts on $A$ so that $A/\Gamma$ is compact 
if and only if $\Gamma^*$ acts on $A^*$ and $A^*/\Gamma^*$ is compact. 
\end{itemize} 
\end{proposition} 

We have $\orb = \Omega/\Gamma$ for a properly convex domain $\Omega$, 
the dual orbifold $\orb^* = \Omega^*/\Gamma^*$ is a properly convex real projective orbifold 
homotopy equivalent to $\orb$. The dual orbifold is well-defined up to projective diffeomorphisms. 
We call $\orb^*$ a projectively dual orbifold to $\orb$.
Clearly, $\orb$ is projectively dual to $\orb^*$. 

\begin{theorem}[Vinberg]  \label{thm-dualdiff} 
The dual orbifold $\orb^*$ is diffeomorphic to $\orb$.
\end{theorem}
We call the map the {\em Vinberg duality diffeomorphism}.

\subsection{The proof of Corollary \ref{cor-dualNPCC} } \label{sub-dualNPCC}



By Corollary \ref{cor-NPCChol}, we obtain that the dual holonomy group $g^{-1 T} \in \bGamma_{\tilde E}^{\ast}$ 
has form under a coordinate system:  
\begin{equation} \label{eqn-dualg}
\newcommand*{\temp}{\multicolumn{1}{r|}{}}
\Scale[0.85]{g^{-1 T} = \left( \begin{array}{ccccccc} 
S(g)^{-1 T} & \temp & 0 & \temp & 0 & \temp & 0 \\ 
 \cline{1-7}
0 &\temp & \lambda(g)^{-1} &\temp & -\lambda(g)^{-1} v(g) &\temp & \lambda(g)^{-1}\left(-\alpha_{7}(g)+ \frac{||v(g)||^{2}}{2}\right) \\ 
 \cline{1-7}
0 &\temp & 0 &\temp & \lambda(g)^{-1} \Idd &\temp & -\lambda(g)^{-1} v(g)^{T}\\ 
 \cline{1-7}
0 &\temp & 0 &\temp & 0 &\temp & \lambda(g)^{-1} 
\end{array} 
\right).}
\end{equation}
Recall that $\langle S(g), g \in \bGamma_{\tilde E}\rangle $  acts on properly convex set
$K\ast \{k\}$ in $\SI^{n-i_{0}-1}$, a strict join,  for a properly convex set $K \subset \SI^{n-i_{0}-2} \subset \SI^{n-i_{0}-1}$ 
and $k$ from the proof of Theorems \ref{thm-NPCCcase} and \ref{thm-NPCCcase2}. 
$\mathcal{N}$ acts on $\SI^{i_{0}+1}$ containing $\SI^{i_{0}}_{\infty}$ and corresponding to $k$ 
under the projection $\Pi_{K}:\SI^{n} - \SI^{i_{0}}_{\infty} \ra \SI^{n-i_{0}-1}$. 

We have $\bR^{n+1} = V \oplus W$ for subspaces $V$ and $W$ corresponding to $\SI^{n-i_{0}-2}$ and $\SI^{i_{0}+1}$
respectively. 
We let $\SI^{n-i_{0}-2\ast}$ and $\SI^{i_{0}+1\ast}$ denote the dual subspaces in $\SI^{n\ast}$: 
Then $\bR^{n+1\ast} = V^{\ast}\oplus W^{\ast}$ for subspaces $V^{\ast}$ of $1$-forms on $V$ and 
$W^{\ast}$ of $1$-forms of $W$. Then $V^{\ast}$ corresponds to the subspace $\SI^{n-i_{0}-2\ast}$ 
and $W^{\ast}$ corresponds to $\SI^{i_{0}+1\ast}$. 


Let $K^{\ast}\subset \SI^{n-i_{0}-2\ast} \subset \SI^{n-i_{0}-1\ast}$ 
be the dual domain of $K$.  
The subspace $\SI^{n-i_{0}-2}$ is dual to a point $k^{\ast}$ of $\SI^{n-i_{0}-1 \ast}$.
Now, $K\ast \{k\}$ is dual to $K^{\ast} \ast \{k^{\ast}\}$ in $\SI^{n-i_{0}-1\ast}$. 
Then 
$\langle S(g)^{-1T}, g \in \bGamma_{\tilde E}^{\ast}\rangle $  acts on the properly convex set
$K^{\ast} \ast \{k^{\ast}\}$. 


Recall that $\bGamma_{\tilde E}$ and the unipotent group $\CN$ act on 
a p-end neighborhood $U$ of $\tilde E$ and on great spheres 
$\SI^{n-i_{0}-2}$ and $\SI^{i_{0}+1}$.
Then $\bGamma_{\tilde E}^{\ast}$ and the unipotent group $\CN$ act on 
the dual great spheres $\SI^{n-i_{0}-2\ast}$ and $\SI^{i_{0}+1\ast}$
by the matrix forms of the elements. 



Let $P \subset \SI^{n}$  be an oriented hyperplane supporting $\torb$ at $\bv_{\tilde E}$. 
Under $\Pi_{K}$,  $P$ 
goes to a hyperplane in $\SI^{n-i_{0}-1}$ disjoint from $(K\ast \{k\})^{o}$. 
$(K \ast \{k\})^{o}$ is in the orientation direction of the image of $P$. 
Hence, the set of supporting oriented hyperplanes is projectively isomorphic to $K^{\ast} \ast k^{\ast}$. 
Using the map $\mathcal{D}$, we obtain that there exists a totally geodesic 
$n-i_{0}-1$-dimensional domain in $\Bd \torb^{\ast}$ projectively isomorphic to $K^{\ast} \ast k^{\ast}$.
We denote the domain by $K_{1}^{\ast} \ast k_{1}^{\ast}$. 
Here, $k_{1}^{\ast}$ is the dual of the supporting hyperplane containing $K$ and $\SI^{i_{0}}_{\infty}$. 



And $\bGamma_{\tilde E}^{\ast}$ virtually normalizes: 
\renewcommand{\arraystretch}{1.5}
\begin{equation} \label{eqn-seconddualm}
\newcommand*{\temp}{\multicolumn{1}{r|}{}}
\CN(\vec{v})^{-1T} = \left( \begin{array}{ccccccc} 
\Idd_{n-i_0-1} & \temp & 0 & \temp & 0 & \temp & 0 \\ 
 \cline{1-7}
0 &\temp & 1 &\temp &  -\vec{v} &\temp & ||\vec{v}||^2 /2 \\ 
 \cline{1-7}
0 &\temp & 0 &\temp & \Idd_{i_0} &\temp & -\vec{v}^{T} \\ 
 \cline{1-7}
0 &\temp & 0 &\temp & 0 &\temp & 1 
\end{array} 
\right), 
\end{equation} 
By using coordinate change of $n-i_{0}+1$-th coordinate to $n+1$-th coordinate, we can make the lower right matrix 
of $\bGamma_{\tilde E}$ and $\CN$
into a lower triangular form.


Now, $\bGamma_{\tilde E}^{\ast}$ fixes $k_{1}^{\ast}$. 
The eigenvalues show that the dual p-end $\tilde E^{\ast}$ is not complete
by Theorem \ref{I-thm-affinehoro} in \cite{EDC1}.  
Since elements of $\bGamma_{\tilde E}$ is of form \eqref{eqn-seconddualm}, 
$\tilde \Sigma_{\tilde E}$ is not properly convex considering 
the matrices expression of their action on $\SI^{n-1}_{\bv_{\tilde E}}$.
One can check that the uniform positive translation condition holds.

Recall $\CN$ acts on a quasi-joined end neighborhood $U \supset \torb$
with $i_{0}$-dimensional orbits in $\Bd U$. 
We can find a properly convex open set $U_{1}\supset \torb$ 
by expanding along radial lines and taking a convex hull
and Proposition \ref{prop-qjoin}. (This step is similar to ones  in Lemma \ref{II-lem-expand} 
in \cite{EDC2} and we skip details.)
For the dual properly convex open set $U_{1}^{\ast}$ we have
$U_{1}^{\ast} \subset \torb^{\ast}$ 
by the reversal of inclusion relations under duality. 
$\bGamma_{\tilde E}^{\ast}$ also fixes $k_{1}^{\ast}$ by the new form of the matrices.
The space of radial lines from $k_{1}^{\ast}$ to $\torb^{\ast}$ is same as that of $U_{1}^{\ast}$ 
Since $\CN^{\ast}$ acts on $U_{1}^{\ast}$, 
$\tilde E^{\ast}$ is an NPCC-end with complete affine leaves of dimension $i_{0}$. 

By Theorem \ref{thm-dualdiff}, 
each end neighborhood of $\orb$ goes to an end neighbourhood of $\orb^{*}$. 
Hence, the weak uniform middle eigenvalue condition is satisfied by the form of the matrices. 
Also the uniform positive translation condition holds by the matrix forms again. 
Proposition \ref{prop-qjoin} completes the proof.


\bibliographystyle{plain}



\end{document}